\pdfoutput=1
\documentclass[a4paper,english,11pt,draftmargins]{texbase/shinyart} 
\usepackage[utf8]{inputenc}
\usepackage[T1]{fontenc}
\usepackage{csquotes}
\usepackage{babel}
\usepackage{tikz}
\usepackage{float}
\usepackage{algorithmic}
\usepackage{stmaryrd}
\usepackage{enumitem}
\usepackage{texbase/nicecleveref}

\newtheorem{assumption}[definition]{Assumption}

\crefname{figure}{Figure}{Figures}
\crefname{assumption}{Assumption}{Assumptions}

\usepackage{bbm}
\usepackage{amssymb}
\usepackage{mdframed}

\makeatletter
\newenvironment{taggedsubequations}[1]
 {%
  \addtocounter{equation}{-1}%
  \begin{subequations}%
  \def\@currentlabel{#1}%
  \renewcommand{\theequation}{#1.\alph{equation}}%
 }
 {\end{subequations}}
\makeatother 

\mdfdefinestyle{examplebg}{
    backgroundcolor=black!7,%
    hidealllines=true,%
    innertopmargin=-.3em,%
    innerbottommargin=.7em,%
    innerleftmargin=.7em,%
    innerrightmargin=.7em,%
}

\surroundwithmdframed[style=examplebg]{example}
\surroundwithmdframed[style=examplebg]{algorithm}

\newcounter{Hequation}

\makeatletter
\g@addto@macro\equation{\stepcounter{Hequation}}
\makeatother

\newcommand{\term}{\emph}

\newcommand{\field}[1]{\mathbb{#1}}
\newcommand{\N}{\mathbb{N}}

\newcommand{\R}{\field{R}}
\newcommand{\extR}{\overline \R}

\newcommand{\norm}[1]{\|#1\|}

\newcommand{\abs}[1]{|#1|}

\newcommand{\inv}[1]{#1^{-1}}
\newcommand{\grad}{\nabla}
\newcommand{\freevar}{\,\boldsymbol\cdot\,}

\newcommand{\Union}\bigcup
\newcommand{\Isect}\bigcap
\newcommand{\union}\cup
\newcommand{\isect}\cap
\newcommand{\bigunion}\bigcup
\newcommand{\bigisect}\bigcap
\newcommand{\powerset}{\mathcal{P}}
\newcommand{\defeq}{:=}

\newcommand{\subdiff}{\partial}

\newcommand{\MIN}[1]{{\underline {#1}}}
\newcommand{\MAX}[1]{{\overline {#1}}}

\makeatletter
\def \uminus@sym{\setbox0=\hbox{$\cup$}\rlap{\hbox 
        to\wd0{\hss\raise0.5ex\hbox{$\scriptscriptstyle{-}$}\hss}}\box0}
    \def \uminus    {\mathrel{\uminus@sym}}
\makeatother

\newcommand{\mathvar}[1]{\textup{#1}}

\newcommand{\iprod}[2]{\langle #1,#2\rangle}

\makeatletter
\def \weaktostar@sym{\setbox0=\hbox{$\rightharpoonup$}\rlap{\hbox 
        to\wd0{\hss\raise1ex\hbox{$\scriptscriptstyle{*\,}$}\hss}}\box0}
    \def \weaktostar    {\mathrel{\weaktostar@sym}}
\makeatother

\renewcommand{\tilde}{\widetilde}

\newcommand{\Eabs}{\mathcal{E}}

\renewcommand{\L}{\mathcal{L}}

\renewcommand{\d}{\,d} 

\newcommand{\TGV}{\mathvar{TGV}}
\newcommand{\TV}{\mathvar{TV}}

\usepackage[textsize=footnotesize,color=DarkOrange!40]{todonotes}

\usepackage{mathtools}

\floatstyle{ruled}
\newfloat{Algorithm}{tbp!}{loa}
\crefname{Algorithm}{Algorithm}{Algorithms}

\def\extR{\overline \R}
\def\linear{\mathcal{L}}

\newcommand{\linearLArrow}[1][]{\linear_{\triangleleft\ifx&#1&\else,\,#1\fi}}
\newcommand{\linearLArrowSpecial}[1][]{\linear^{\star}_{\triangleleft\ifx&#1&\else,\,#1\fi}}

\def\realopt#1{\widehat #1}
\def\this#1{#1^i}
\def\nexxt#1{#1^{i+1}}
\def\overnext#1{\bar #1^{i+1}}

\def\boundfx{\delta^x}
\def\boundfy{\delta^y}

\def\realoptu{{\realopt{u}}}
\def\realoptv{{\realopt{v}}}
\def\realoptx{{\realopt{x}}}

\def\realopty{{\realopt{y}}}

\def\nextu{\nexxt{u}}
\def\nextx{\nexxt{x}}

\def\nexty{\nexxt{y}}

\def\thisu{\this{u}}
\def\thisx{\this{x}}

\def\thisy{\this{y}}

\def\overnextx{\overnext{x}}

\def\E{\mathbb{E}}
\def\P{\mathbb{P}}

\def\wcase{\mathbb{w}}

\def\Tau{T}
\def\TauTest{\Phi}
\def\tauTest{\phi}
\def\SigmaTest{\Psi}
\def\sigmaTest{\psi}

\def\gap{\mathcal{G}}

\def\Neigh{\mathcal{V}}

\def\GammaLift#1{\Xi_{#1}}
\def\SAlg{\mathcal{O}}

\def\iset#1{\mathring{#1}}

\newcommand{\Test}{Z}
\newcommand{\Precond}{M}
\newcommand{\Step}{W}
\newcommand{\Random}{\mathcal{R}}
\def\MAX{\overline}
\def\MIN{\underline}

\def\kappafamily{\mathcal{K}(K, \mathcal{P}, \mathcal{Q})}

\newcommand{\convex}{\mathcal{C}}

\DeclareFontFamily{U}{mathx}{\hyphenchar\font45}
\DeclareFontShape{U}{mathx}{m}{n}{<-> mathx10}{}
\DeclareSymbolFont{mathx}{U}{mathx}{m}{n}
\DeclareMathAccent{\widebar}{0}{mathx}{"73}

\newcommand{\Penalty}{\Delta}

\newcommand{\MetricDiff}{D}

\def\TAGSTRUCT{S-}

\makeatother

\title{Block-proximal methods with spatially adapted acceleration}

\makeatletter
\AtBeginDocument{
\hypersetup{
    pdftitle = {Block-proximal methods with spatially adapted acceleration},
    pdfauthor = {T.~Valkonen}
}
}
\makeatother

\begin{document}

\date{V4: 2018-11-19\thanks{V3: 2017-11-06, V2: 2017-03-16, V1: 2016-09-16}}
\author{
    Tuomo Valkonen\thanks{ModeMat, Escuela Politécnica Nacional, Quito, Ecuador; \emph{previously} Department of Mathematical Sciences, University of Liverpool, United Kingdom. \email{tuomo.valkonen@iki.fi}}
    }

\maketitle

\begin{abstract}
    We study and develop (stochastic) primal--dual block-coordinate descent methods for convex problems based on the method due to Chambolle and Pock.
    Our methods have known convergence rates for the iterates and the ergodic gap: $O(1/N^2)$ if each block is strongly convex, $O(1/N)$ if no convexity is present, and more generally a mixed rate $O(1/N^2)+O(1/N)$ for strongly convex blocks, if only some blocks are strongly convex.
    Additional novelties of our methods include blockwise-adapted step lengths and acceleration, as well as the ability to update both the primal and dual variables randomly in blocks under a very light compatibility condition. In other words, these variants of our methods are doubly-stochastic.
	We test the proposed methods on various image processing problems, where we employ pixelwise-adapted acceleration.

    \textbf{Get the version from \url{http://tuomov.iki.fi/publications/}: citations are broken/poorly formatted in this one due arXiv being stuck in the 70s and not supporting biblatex (or 80s bibtex for that matter), hence not modern bibliography styles or utf8.}
\end{abstract}


\section{Introduction}
\label{sec:intro}

We want to efficiently solve optimisation problems of the form
\begin{equation}
    \label{eq:form}
    \tag{P$_0$}
    \min_x~G(x)+F(Kx),
\end{equation}
arising, in particular, from 
image processing and inverse problems. We assume $G: X \to \extR$ and $F: Y \to \extR$ to be convex, proper, and lower semicontinuous on Hilbert spaces $X$ and $Y$ and $K \in \linear(X; Y)$ to be a bounded linear operator. We are particularly interested in block-separable%
\begin{equation}
    \label{eq:g-fstar-separable}
    \tag{GF}
    G(x) = \sum_{j=1}^m G_j(P_j x),
    \quad\text{and}\quad
    F^*(y)  = \sum_{\ell=1}^n F^*_\ell(Q_\ell y),
\end{equation}
where $F^*$ is the Fenchel conjugate of $F$.
The operators $P_1,\ldots,P_m$ are projections in $X$ with $\sum_{j=1}^m P_j = I$ and $P_jP_i=0$ if $i \ne j$. Likewise, $Q_1,\ldots,Q_n$ are projection operators in $Y$.
We assume all the component functions $G_j$ and $F^*_\ell$ to be convex, proper, and lower semicontinuous, and the subdifferential sum rule to hold for the expressions \eqref{eq:g-fstar-separable}.

Several first-order optimisation methods have been developed for \eqref{eq:form} without block-separable structure, typically both $G$ and $F$ convex and $K$ linear. Recently also some non-convexity and non-linearity has been introduced \cite{bolte2014proximal,mollenhoff2014primal,tuomov-cpaccel,ochs2014ipiano}.
In applications to image processing and data science, one of $G$ or $F$ is typically non-smooth. Effective algorithms operating directly on the primal problem \eqref{eq:form}, or its dual, therefore tend to be a form of classical forward--backward splitting, occasionally called iterative soft-thresholding \cite{daubechies2004surrogate,beck2009fista}.

In big data optimisation several forward--backward block-coordinate descent methods have been developed for \eqref{eq:form} with block-separable $G$.
On each step, the methods update only a random subset of blocks $x_j \defeq P_j x$ in parallel; see the review \cite{wright2015coordinate} and the original articles
\cite{nesterov2012efficiency,
      richtarik2012parallel,
      fercoq2013approx,
      richtarik2013distributed,
      qu2014randomized,
      zhao2014stochastic,
      shalev-shwartz2014accelerated,
      csiba2015stochastic,
      combettes2016stochastic,
      peng2015arock,
      bertsekas2015incremental}.
Typically $F$ is assumed smooth, and, often, each $G_j$ strongly convex.
Besides parallelism, an advantage of these methods is the exploitation of \emph{blockwise} factors of smoothness and strong convexity. These can help convergence by being better than the global factor.

Unfortunately, primal-only and dual-only stochastic methods, as discussed above, are rarely applicable to image processing problems. These, and many other problems, do not satisfy the assumed separability and smoothness assumptions. On the other hand, additional Moreau--Yosida (aka.~Huber, aka.~Nesterov) regularisation of the problem, which would provide the required smoothness, would alter the problem, losing essential non-smooth characteristics.
Generally, even without the splitting of the problem into blocks and the introduction of stochasticity, primal-only or dual-only methods can be inefficient on more complicated problems. Proximal steps, which are typically used to deal with non-smooth components of the problem, can in particular be as expensive as the original optimisation problems itself. In order to make these steps cheap, the problem has to be formulated appropriately.
Such a reformulation can often be provided through primal--dual approaches.

With the Fenchel conjugate $F^*$, we can write \eqref{eq:form} as
\begin{equation}
    \label{eq:form-saddle}
    \min_x \max_y~G(x)+\iprod{Kx}{y}-F^*(y).
\end{equation}
The method of Chambolle and Pock \cite{chambolle2010first,pock2009mumford} is popular for this formulation. It is also called the PDHGM (Primal-Dual Hybrid Gradient Method, Modified) in \cite{esser2010general} and the PDPS (Primal–Dual Proximal Splitting) in \cite{tuomov-inertia}.
It consists of alternating proximal steps on $x$ and $y$ combined with an over-relaxation step to ensure convergence.
The method is closely related to the classical ADMM and Douglas--Rachford splitting. The acronym PDHGM arises from the earlier PDHG \cite{zhu2008efficient} that is convergent only in special cases \cite{he2014convergence}. These connections are discussed in \cite{esser2010general}.


While early block-coordinate methods only worked with a primal or a dual variable, recently stochastic primal--dual approaches based on the ADMM and the PDHGM have been proposed %
\cite{suzuki2013stochastic,
      zhang2014stochastic,
      fercoq2015coordinate,
      bianchi2015stochastic,
      peng2016coordinate,
      pesquet2014class,
      yu2015doubly}. %
Moreover, variants of the ADMM that deterministically update multiple blocks in parallel and afterwards combine the results for the Lagrange multiplier update have been introduced \cite{he2015block}. As with the primal- or dual-only methods, these algorithms can improve convergence by exploiting local properties of the problem.
Besides \cite{suzuki2013stochastic,zhang2014stochastic,yu2015doubly} that have restrictive smoothness and strong convexity requirements, little is known about convergence rates.

\emph{In this paper, we will derive block-coordinate descent variants of the PDHGM with known convergence rates:} $O(1/N^2)$ if each $G_j$ is strongly convex, $O(1/N)$ without any strong convexity, and mixed $O(1/N^2)+O(1/N)$ if some of the $G_j$ are strongly convex. These rates apply to an ergodic duality gap and strongly convex blocks of the iterates.
Our methods have the novelty of blockwise-adapted step lengths.
In the imaging applications of \cref{sec:numerical} we will even employ pixelwise-adapted step lengths.
Moreover, we can update random subsets of \emph{both} primal and dual blocks under a light ``nesting condition'' on the sampling scheme. Such ``doubly-stochastic'' updates have previously been possible only in very limited settings \cite{yu2015doubly}.

Our present work is based on \cite{tuomov-cpaccel} on the acceleration of the PDHGM when $G$ is strongly convex only on a subspace: the deterministic two-block case $m=2$ and $n=1$ of \eqref{eq:g-fstar-separable}.
Besides enabling (doubly-)stochastic updates and an arbitrary number of both primal \emph{and} dual blocks, in the present work, we derive simplified step length rules through a more careful analysis.

The more abstract basis of our present work has been split out in \cite{tuomov-proxtest}. There we study preconditioning of abstract proximal point methods and ``testing'' by suitable operators as means of obtaining convergence rates. We recall the relevant aspects of this theory through the course of \cref{sec:background,sec:block}.
In the first of these sections, we start by going through notation and previous research on the PDHGM in more detail. Then we develop the rough structure of our proposed method. This will depend on several structural conditions that we introduce in \cref{sec:background}. Afterwards in \cref{sec:block} we develop convergence estimates based on technical conditions on the various step length and testing parameters.
These conditions need to be verified through the development of explicit parameter update rules. We do this in \cref{sec:block-continued} along with proving the claimed convergence rates (\cref{thm:convergence-basic} and its corollaries). We also present there the final, detailed, versions of our proposed algorithms: \cref{alg:alg-blockcp} (doubly stochastic) and \cref{alg:alg-blockcp-fulldual} (simplified).
We finish with numerical experiments in \cref{sec:numerical}.




\section{Background and overall structure of the algorithm}
\label{sec:background}

To make the notation definite, we write $\linear(X; Y)$ for the space of bounded linear operators between Hilbert spaces $X$ and $Y$. The identity operator we denote by $I$.
For $T, S \in \linear(X; X)$, we use $T \ge S$ to mean that $T-S$ is positive semi-definite; in particular $T \ge 0$ means that $T$ is positive semi-definite.
Also for possibly non-self-adjoint $T$, we introduce the inner product and norm-like notations
\begin{equation}
    \label{eq:iprod-def}
    \iprod{x}{z}_T \defeq \iprod{Tx}{z},
    \quad
    \text{and}
    \quad
    \norm{x}_T \defeq \sqrt{\iprod{x}{x}_T},
\end{equation}
the latter only defined for positive semi-definite $T$.
We write $T \simeq T'$ if 
$\iprod{x}{x}_{T'-T}=0$ for all $x$.

We denote by $\convex(X)$ the set of convex, proper, lower semicontinuous functionals from a Hilbert space $X$ to $\extR \defeq [-\infty,\infty]$.
With $G \in \convex(X)$, $F^* \in \convex(Y)$, and $K \in \linear(X; Y)$, we then wish to solve the minimax problem%
\begin{equation}
    \label{eq:problem}
    \tag{P}
    \min_{x \in X} \max_{y \in Y} \ G(x) + \iprod{K x}{y} - F^*(y),
\end{equation}
assuming the existence of a solution $\realoptu=(\realoptx, \realopty)$ satisfying the optimality conditions
\begin{equation}
    \label{eq:oc}
    \tag{OC}
    -K^* \realopty \in \subdiff G(\realoptx),
    \quad\text{and}\quad
    K \realoptx \in \subdiff F^*(\realopty).
\end{equation}

For the stochastic aspects of our work, we denote by $(\Omega, \SAlg, \P)$ the \term{probability space} consisting of the set $\Omega$ of possible realisation of a random experiment, by $\SAlg$ a $\sigma$-algebra on $\Omega$, and by $\P$ a probability measure on $(\Omega, \SAlg)$.
We denote the expectation corresponding to $\P$ by $\E$, the conditional probability with respect to a sub-$\sigma$-algebra $\SAlg' \subset \SAlg$ by $\P[\freevar|\SAlg']$, and the conditional expectation by $\E[\freevar|\SAlg']$. We refer to \cite{shiryaev1996probability} for more details.

We also use the next non-standard notation: If $\SAlg$ is a $\sigma$-algebra on the space $\Omega$, we denote by $\Random(\SAlg; V)$ the space of $V$-valued random variables $A$, such that $A: \Omega \to V$ is $\SAlg$-measurable. 

\subsection{Preconditioned proximal point methods; testing for rates}

We use the notation
\[
    u=(x, y)
\]
to combine the primal variable $x$ and dual variable $y$ into a single variable $u$.
Following \cite{he2012convergence,tuomov-cpaccel}, the primal--dual method of Chambolle and Pock \cite{chambolle2010first} (PDHGM) may then be written in proximal point form as%
\begin{equation}
    \label{eq:prox-update}
    \tag{PP$_0$}
    0 \in H(\nextu) + L_i(\nextu-\thisu)
\end{equation}
for a monotone operator $H$ encoding the optimality conditions \eqref{eq:oc} as $0 \in H(\realoptu)$, and a \emph{preconditioning} or \emph{step length operator} $L_i=L^0_i$. These are
\begin{equation}
    \label{eq:cpoper0}
    H(u) \defeq
        \begin{pmatrix}
            \subdiff G(x) + K^* y \\
            \subdiff F^*(y) -K x
        \end{pmatrix},
    \quad\text{and}\quad
    L^0_i \defeq
        \begin{pmatrix}
            \inv\tau_i & -K^* \\
            -\omega_i K & \inv\sigma_{i+1}
        \end{pmatrix}.
\end{equation}
Here $\tau_i, \sigma_{i+1}>0$ are step length parameters, and $\omega_i>0$ an over-relaxation parameter. 
In the basic version of the algorithm, $\omega_i=1$, $\tau_i \equiv \tau_0$, and $\sigma_i \equiv \sigma_0$, assuming $\tau_0 \sigma_0 \norm{K}^2 < 1$.
Observe that we may equivalently parametrise the algorithm by $\tau_0$ and $\delta = 1 - \norm{K}^2 \tau_0\sigma_0 > 0$. 
The method has $O(1/N)$ rate for the ergodic duality gap that we will return to in \cref{sec:gap}.

If $G$ is strongly convex with factor $\gamma>0$, we may for $\tilde\gamma \in (0,\gamma]$ accelerate
\begin{equation}
    \label{eq:cpaccel}
    \omega_i \defeq 1/\sqrt{1+2\tilde\gamma\tau_i},
    \quad
    \tau_{i+1} \defeq \tau_i\omega_i,
    \quad\text{and}\quad
    \sigma_{i+1} \defeq \sigma_i/\omega_i.
\end{equation}
This gives $O(1/N^2)$ convergence of  $\norm{x^N-\realoptx}^2$ to zero. If $\tilde\gamma \in (0, \gamma/2]$, we also obtain $O(1/N^2)$ convergence of an ergodic duality gap.

In \cite{tuomov-cpaccel}, we extended the PDHGM to partially strongly convex problems: in \eqref{eq:g-fstar-separable} this corresponded to the primal two-block and dual single-block case $m=2$ and $n=1$ with only $G_1$ assumed strongly convex.
This extension was based on taking in \eqref{eq:prox-update} the preconditioner
\begin{equation}
    \label{eq:cpoper-m}
    L_i =
        \begin{pmatrix}
            \inv\Tau_i & -K^* \\
            -\omega_i K & \inv\Sigma_{i+1}
        \end{pmatrix}
\end{equation}
for \emph{invertible} $\Tau_i =  \tau_{1,i} P_1 + \tau_{2,i} P_2 \in \L(X; X)$ and $\Sigma_{i+1} = \sigma_{i+1} I \in \L(Y; Y)$. After simple rearrangements of \eqref{eq:prox-update}, the resulting algorithm could be written more explicitly as
\begin{subequations}%
\label{eq:block-cp-basic}%
\begin{align}%
    \nextx & \defeq \inv{(I+\Tau_i\subdiff G)}(\thisx -  \Tau_i K^*\thisy), \\
    \nexty & \defeq \inv{(I+\Sigma_{i+1}\subdiff F^*)}(\thisy + \Sigma_{i+1} K((1+\omega_i)\nextx-\omega_i\thisx)).
\end{align}%
\end{subequations}%
Since $G$ is as assumed separable, the first, primal update, splits into separate updates for $\nextx_1 \defeq P_1\nextx$ and $\nextx_2 \defeq P_2\nextx$.
Note that this explicit form of the algorithm does not require $\Tau_i$ and $\Sigma_{i+1}$ to be invertible, unlike \eqref{eq:prox-update} with the choice \eqref{eq:cpoper-m}, so suggests we could develop stochastic methods that randomly choose one, two, or no primal blocks to update.

To study convergence, it is, however, more practical to work with implicit formulations, such as \eqref{eq:prox-update}. We will shortly see how this works. To make  \eqref{eq:prox-update} work with non-invertible $\Tau_i$ and $\Sigma_{i+1}$, let us reformulate it slightly. In fact, let us define%
\begin{equation}
    \label{eq:test-precond-first}
    \Step_{i+1} \defeq
    \begin{pmatrix}
        \Tau_i & 0 \\
        0 & \Sigma_{i+1}
    \end{pmatrix},
    \quad\text{and (for the moment)}\quad
    \Precond_{i+1} =
    \begin{pmatrix}
        I & -\Tau_i K^* \\
        -\tilde \omega_i \Sigma_{i+1}  K & I
    \end{pmatrix}.
\end{equation}
With this, whether or not $\Tau_i$ and $\Sigma_{i+1}$ are invertible, \eqref{eq:block-cp-basic} can be written as the preconditioned proximal point iteration%
\begin{equation}
    \label{eq:pp}
    \tag{PP}
    \Step_{i+1} H(\nextu) + \Precond_{i+1}(\nextu-\thisu) \ni 0,
\end{equation}
This will be the abstract form of the algorithms that we will develop, however, with the exact form of $\Tau_{i+1}$, $\Sigma_{i+1}$, and $\Precond_{i+1}$ still to be refined.

To study the convergence of \eqref{eq:pp}, we apply to the \emph{testing} framework introduced in \cite{tuomov-cpaccel,tuomov-proxtest}. The idea is to apply $\iprod{\freevar}{\nextu-\realoptu}_{\Test_{i+1}}$ for a \term{testing operator} $\Test_{i+1}$ to \eqref{eq:pp} to ``test'' it.
Thus
\begin{equation}
    \label{eq:testing0}
    0 \in \iprod{\Step_{i+1} H(\nextu) + \Precond_{i+1}(\nextu-\thisu)}{\nextu-\realoptu}_{\Test_{i+1}}.
\end{equation}
We need $\Test_{i+1} \Precond_{i+1}$ to be self-adjoint and positive semi-definite. This guarantees that $\Test_{i+1}\Precond_{i+1}$ can be used to form the local semi-norm $\norm{\freevar}_{\Test_{i+1}\Precond_{i+1}}$. Indeed, assuming for some linear operator $\Xi_{i+1}$ that $H$ has the operator-relative (strong) monotonicity property
\begin{equation}
    \label{eq:xi-monotonicity}
    \iprod{H(u')-H(u)}{u'-u}_{\Test_{i+1}\Step_{i+1}} \ge \norm{u-u'}_{\Test_{i+1}\Xi_{i+1}}^2
    \quad (u, u' \in X \times Y),
\end{equation}
then a simple application of Pythagoras' identity
\[
    \iprod{\nextu-\thisu}{\nextu-\realoptu}_{\Test_{i+1}\Precond_{i+1}}
    = \frac{1}{2}\norm{\nextu-\thisu}_{\Test_{i+1}\Precond_{i+1}}^2
    - \frac{1}{2}\norm{\thisu-\realoptu}_{\Test_{i+1}\Precond_{i+1}}^2
    + \frac{1}{2}\norm{\nextu-\realoptu}_{\Test_{i+1}\Precond_{i+1}}^2
\]
yields
\begin{equation*}
  \frac{1}{2}\norm{\nextu-\realoptu}_{\Test_{i+1}(\Precond_{i+1}+2\Xi_{i+1})}^2
  +
  \frac{1}{2}\norm{\nextu-\thisu}_{\Test_{i+1}\Precond_{i+1}}^2
  \le
  \frac{1}{2}\norm{\thisu-\realoptu}_{\Test_{i+1}\Precond_{i+1}}^2.
\end{equation*}
If $\Test_{i+2}\Precond_{i+2} \le \Test_{i+1}(\Precond_{i+1}+2\Xi_{i+1})$ for all $i$, then summing over $i=0,\ldots,N-1$ gives
\begin{equation}
  \label{eq:basic-convergence-estimate}
  \frac{1}{2}\norm{u^N-\realoptu}_{\Test_{N+1}\Precond_{N+1}}^2
  +
  \sum_{i=0}^{N-1}\frac{1}{2}\norm{\nextu-\thisu}_{\Test_{i+1}\Precond_{i+1}}^2
  \le
  \frac{1}{2}\norm{u^0-\realoptu}_{\Test_{1}\Precond_{1}}^2.
\end{equation}
We therefore see that $\Test_{i+1}\Precond_{i+1}$ measures the rates of convergence of the iterates.
If our iterations are stochastic, to obtain deterministic estimates, we can simply take the expectation in \eqref{eq:basic-convergence-estimate}. However, to obtain estimates on a duality gap, we need to work significantly more. We will, therefore, in the beginning of \cref{sec:block}, after introducing all the relevant concepts and finalising the setup for the present work, quote the appropriate results from \cite{tuomov-proxtest}.

\subsection{Stochastic and deterministic block updates}

We want to update any subset of any number of primal and dual blocks stochastically. Compatible with the separable structure \eqref{eq:g-fstar-separable} of $G$ and $F^*$, we therefore construct from individual (possibly random) step length and testing parameters, $\tau_{j,i},\sigma_{\ell,i+1} \ge 0$ and $\tauTest_{j,i}, \sigmaTest_{\ell,i+1}>0$, as well as random subsets $S(i) \subset \{1,\ldots,m\}$ and $V(i+1) \subset \{1,\ldots,n\}$ the step length and testing operators
\begin{taggedsubequations}{S}%
\label{eq:main-structure}
\begin{align}%
    \Step_{i+1} &\defeq
    \begin{pmatrix}
        \Tau_i & 0 \\
        0 & \Sigma_{i+1}
    \end{pmatrix}
    \quad\text{and}
    &
    \Test_{i+1} & \defeq
    \begin{pmatrix}
        \TauTest_i & 0 \\
        0 & \SigmaTest_{i+1}
    \end{pmatrix}
    \quad\text{for}
    \\
    \Tau_i & \defeq \sum_{j \in S(i)} \tau_{j,i}P_{j},
    &
    \Sigma_{i+1} & \defeq \sum_{\ell \in V(i+1)} \sigma_{\ell,i+1} Q_\ell,
    \\
    \TauTest_i & \defeq \sum_{j=1}^m \tauTest_{j,i} P_j,
    \quad\text{and}
    &
    \SigmaTest_{i+1} & \defeq \sum_{\ell=1}^n \sigmaTest_{\ell,i+1} Q_\ell
    \quad (i \ge 0).
\intertext{We moreover take as the preconditioner}
  \label{eq:precond-lambda}
  \Precond_{i+1} & \defeq \begin{pmatrix}
    I & -\inv\TauTest_i\Lambda_i^* \\
    -\inv\SigmaTest_{i+1}\Lambda_i & I
  \end{pmatrix}
  \quad\text{for}
  &
  \Lambda_i & \defeq K\iset\Tau_i^*\TauTest_i^*  - \SigmaTest_{i+1}\iset \Sigma_{i+1} K
  \quad\text{with}
  \\
  \label{eq:iset-tau}
  \iset\Tau_i & \defeq \sum_{j \in \iset S(i)} \tau_{j,i} P_j,
  &
  \iset S(i) & \subset S(i),
  \\
  \label{eq:iset-sigma}
  \iset \Sigma_{i+1} & \defeq \sum_{\ell \in \iset V(i+1)} \sigma_{\ell,i+1} Q_\ell,
  \quad\text{and} &
  \iset V(i+1) & \subset V(i+1).
\end{align}
\end{taggedsubequations}%

The subsets $S(i)$ and $V(i+1)$ are the indices of the blocks
\begin{equation}
  \label{eq:coord-notation}
    x_j \defeq P_j x,
    \quad\text{and}\quad
    y_\ell \defeq Q_\ell y
\end{equation}
of the variables $x$ and $y$ that are to be updated on iteration $i$.\footnote{The iteration index is off-by-one for $\sigma_{\ell, i+1}$ and $\sigmaTest_{\ell,i+1}$ for reasons of historical development of the Chambolle--Pock method, when it was not written as a preconditioned proximal point method.}
Hence $\Tau_i$ and $\Sigma_{i+1}$ will not be invertible unless we update all the blocks. Clearly $\TauTest_i$, $\SigmaTest_{i+1}$, $\Tau_i$, and $\Sigma_{i+1}$ are self-adjoint and positive semi-definite with $\TauTest_i$ and $\SigmaTest_{i+1}$ invertible.
The subsets $\iset S(i)$ and $\iset V(i+1)$ indicate those blocks of $\nextx$ and of $\nexty$ that are to be updated ``independently'' of the other variable. We will explain these subsets and the choice of $\Lambda_i$ in more detail in \cref{sec:lambda}.

The iterate $\nextu=(\nextx, \nexty)$ has to be computable based on random sampling at iteration $i$ and the information gathered (random variable realisations) before commencing the iteration. For the algorithm to be realisable, it cannot depend on the future. We therefore need to be explicit about the space of each random variable. We model the information available just before commencing iteration $i$ by the $\sigma$-algebra $\SAlg_{i-1}$. Thus $\SAlg_{i-1} \subset \SAlg_{i}$. More precisely, $\SAlg_i$ is the smallest sub-$\sigma$-algebra of $\SAlg$ satisfying for all $k=0,\ldots,i$, $j=1,\ldots,m$, and $\ell=1,\ldots,n$ that
\begin{taggedsubequations}{$\Random$}%
\label{eq:sigma-algebra-cond}
\begin{align}%
    \tau_{j,k},\, & \sigma_{\ell,k+1} \in \Random(\SAlg_i; [0, \infty)), &
    \tauTest_{j,k},\, & \sigmaTest_{\ell,k+1} \in \Random(\SAlg_i; (0, \infty)), \\
    S(k) &\in \Random(\SAlg_i; \powerset(\{1,\ldots,m\})), &
    V(k+1) &\in \Random(\SAlg_i; \powerset(\{1,\ldots,n\})).
    \\
    \iset S(k) &\in \Random(\SAlg_i; \powerset(\{1,\ldots,m\})),
    \quad\text{and}&
    \iset V(k+1) &\in \Random(\SAlg_i; \powerset(\{1,\ldots,n\})).
\end{align}%
\end{taggedsubequations}%
Here and only here $\powerset$ denotes the power set.
Any other variables can only be random by being constructed from these variables.
We thus deduce from \eqref{eq:main-structure}, and \eqref{eq:pp} that
\begin{align*}
    \Tau_k & \in \Random(\SAlg_i; \linear(X; X)),
    &
    \TauTest_k & \in \Random(\SAlg_i; \linear(X; X))
    &
    \nextx &\in \Random(\SAlg_i; X),
    \\
    \Sigma_{k+1} & \in \Random(\SAlg_i; \linear(Y; Y)),
    &
    \SigmaTest_{k+1} & \in \Random(\SAlg_i; \linear(Y; Y)),
    \quad\text{and}
    &
    \nexty & \in \Random(\SAlg_i; Y).
\end{align*}

We will also need to assume the \term{nesting conditions} on sampling,
\begin{taggedsubequations}{$\Neigh$}%
\label{eq:nesting}%
\begin{gather}%
    \label{eq:nesting-iset-s-v-cond}
    \Neigh(\iset S(i)) \isect \iset V(i+1)=\emptyset,
    \quad
    \Neigh(S(i) \setminus \iset S(i)) \isect (V(i+1) \setminus \iset V(i+1))=\emptyset,
    \\
    \label{eq:nesting-s-v-cond}
    \iset S(i) \union \inv\Neigh(\iset V(i+1)) \subset S(i) ,
    \quad\text{and}\quad
    \iset V(i+1) \union \Neigh(\iset S(i)) \subset V(i+1),
\end{gather}%
\end{taggedsubequations}%
where the set
\begin{equation}
    \label{eq:neigh}
    \Neigh(j) \defeq \{ \ell \in \{1,\ldots,n\} \mid Q_\ell KP_j \ne 0 \}.
\end{equation}
consists of the dual blocks that are ``connected'' by $K$ to the primal block with index $j$. Vice versa, $\inv \Neigh(\ell)$ consists of the primal blocks that are ``connected'' by $K$ to the dual block with index $\ell$.
Thus \eqref{eq:nesting-s-v-cond} states that the independent updates ($\iset S(i)$ and $\iset V(i+1)$) must propagate from primal to dual and vice versa as non-independent updates ($S(i)$ and $V(i+1)$).
The condition \eqref{eq:nesting-iset-s-v-cond} restricts connections between primal and dual updates: the first part says that the independently updates blocks cannot be connected. By the second part neither can non-independent updates. If we use \eqref{eq:nesting-s-v-cond} as an equality to define $S(i)$ and $V(i+1)$, then the second part of  \eqref{eq:nesting-iset-s-v-cond} holds if $\Neigh(\inv\Neigh(\iset V(i+1))) \isect \Neigh(\iset S(i))=\emptyset$, that is, the condition restricts second-degree connections between the independently updated blocks.

To facilitate referring to all the above structural conditions, we introduce:

\begin{assumption}[main structural condition]
    \label{ass:structural}
    We assume the structure \eqref{eq:g-fstar-separable} and \eqref{eq:main-structure} with the the limitations \eqref{eq:sigma-algebra-cond} and \eqref{eq:nesting} on randomness.
\end{assumption}

Clearly
\begin{equation}
  \label{eq:zm}
  \Test_{i+1} \Precond_{i+1} =
  \begin{pmatrix}
      \TauTest_i
      &
      -\Lambda_{i}^*
      \\
      -\Lambda_{i}
      &
      \SigmaTest_{i+1}
  \end{pmatrix}
\end{equation}
is self-adjoint. We need to prove that it is positive semi-definite. We will do this in \cref{sec:block-continued} using the functions $\kappa_\ell$ introduced next. As we show afterwards in \cref{ex:kappa-max}, these functions are a block structure adapted generalisation of the simple bound $K \le \norm{K} I$.

\begin{definition}
    \label{def:kappa}
    Let $\mathcal{P} \defeq \{P_1,\ldots,P_m\}$, and $\mathcal{Q} \defeq \{Q_1,\ldots,Q_n\}$. We write $(\kappa_1,\ldots,\kappa_n) \in \kappafamily$ if each $\kappa_\ell: [0,\infty)^m \to [0,\infty)$ is monotone ($\ell = 1,\ldots,n$) and the following hold:
    \begin{enumerate}[label=(\roman*)]
    \begin{taggedsubequations}{{\TAGSTRUCT}$\kappa$}
    \label{eq:kappa}
    \item\label{item:kappa-estimate} (Estimation) For all $(z_{\ell,1},\ldots,z_{\ell,m}) \subset [0, \infty)^m$ and $\ell=1,\ldots,n$ the estimate
    \begin{equation*}
        \sum_{j=1}^m \sum_{\ell,k=1}^n z_{\ell,j}^{1/2} z_{k,j}^{1/2} Q_\ell K P_j K^* Q_k
        \le
        \sum_{\ell=1}^n
        \kappa_\ell(z_{\ell,1},\ldots,z_{\ell,m}) Q_\ell.
    \end{equation*}
    \item\label{item:kappa-bounded} (Boundedness) For some $\MAX \kappa>0$ and all $(z_{1},\ldots,z_{m}) \subset [0, \infty)^m$ the bound
    \begin{equation*}
        \kappa_\ell(z_1,\ldots,z_m) \le \MAX \kappa \sum_{j=1}^m z_j.
    \end{equation*}
    \item\label{item:kappa-nondegeneracy} (Non-degeneracy) There exists $\MIN\kappa>0$ and for all $j=1,\ldots,m$) a choice of $\ell^*(j) \in \{1,\ldots,n\}$  such that for all $(z_{1},\ldots,z_{m}) \subset [0, \infty)^m$,
    \begin{equation*}
        \MIN\kappa z_{j} \le \kappa_{\ell^*(j)}(z_1,\ldots,z_m)
        \quad (j = 1,\ldots,m).
    \end{equation*}
    \end{taggedsubequations}
    \end{enumerate}
\end{definition}

The choice of $\kappa$ allows us to construct different algorithms. 
Here we consider a few possibilities, first an easy one, and then a more optimal one.

\begin{example}[Worst-case $\kappa$]
    \label{ex:kappa-max}
    We may estimate
    \begin{equation}
        \notag
        \sum_{j=1}^m 
        \sum_{\ell,k=1}^n
        z_{\ell,j}^{1/2}z_{k,j}^{1/2} Q_\ell K P_j K^* Q_k
        \le
        \sum_{\ell,k=1}^n
        \MAX z_{\ell}^{1/2}\MAX z_{k}^{1/2} Q_\ell K K^* Q_k
        \le
        \sum_{\ell=1}^n
        \MAX z_{\ell} \norm{K}^2 Q_\ell.
    \end{equation}    
    Therefore \cref{def:kappa}\,\ref{item:kappa-estimate} and \ref{item:kappa-bounded} hold with $\MAX\kappa=\norm{K}^2$ for the monotone choice
    \[
        \kappa_\ell(z_1,\ldots,z_m) \defeq \norm{K}^2 \max\{z_1,\ldots,z_m\}.
    \]
    Clearly also $\MIN\kappa = \MAX\kappa$ for any choice of $\ell^*(j) \in \{1,\ldots,n\}$.
    This choice of $\kappa_\ell$ corresponds to the rule $\tau\sigma\norm{K}^2 < 1$ in the PDHGM method.
\end{example}

\begin{example}[Balanced $\kappa$]
    \label{ex:kappa-refined}
    Take minimal $\kappa_\ell$ satisfying \cref{def:kappa} and the balancing condition
    $
        \kappa_\ell(z_{\ell,1},\ldots,z_{\ell,m})
        =
        \kappa_k(z_{k,1},\ldots,z_{k,m}),
        (\ell, k=1,\ldots,n).
    $
    This requires problem-specific analysis, but tends to perform well, as we will see in \cref{sec:numerical}.
\end{example}


\subsection{Justification of the preconditioner and sampling restrictions}
\label{sec:lambda}

With $\Precond_{i+1}$ of the form \eqref{eq:precond-lambda} for any $\Lambda_i$, the implicit method \eqref{eq:pp} expands as%
\begin{subequations}%
\label{eq:pp-expanded-lambda}
\begin{align}%
  \label{eq:pp-expanded-lambda-x}
  0 & \in \Tau_i \subdiff G(\nextx) + \Tau_i K^*\nexty + (\nextx-\thisx) -\inv\TauTest_i\Lambda_i^*(\nexty-\thisy), \quad\text{and} \\
  \label{eq:pp-expanded-lambda-y}
  0 & \in \Sigma_{i+1} \subdiff F^*(\nexty) -\Sigma_{i+1}K\nextx -\inv\SigmaTest_{i+1}\Lambda_i(\nextx-\thisx) + (\nexty-\thisy).
\end{align}%
\end{subequations}%
This can easily be rearranged as
\begin{subequations}
\label{eq:pp-general-as-algorithm}
\begin{align}
    \label{eq:pp-general-as-algorithm-x}
    \nextx & \defeq \inv{(I+\Tau_i\subdiff G)}(\thisx + \inv\TauTest_i \Lambda_i^*(\nexty-\thisy)-  \Tau_i K^*\nexty), \\
    \label{eq:pp-general-as-algorithm-y}
    \nexty & \defeq \inv{(I+\Sigma_{i+1}\subdiff F^*)}(\thisy + \inv\SigmaTest_{i+1} \Lambda_i(\nextx-\thisx) + \Sigma_{i+1}K\nextx).
\end{align}
\end{subequations}
This method is still not explicit as the primal and dual updates potentially depend on each other. We will now show how the removal of this dependency leads to our choice of $\Lambda_i$ in \eqref{eq:precond-lambda}.

Indeed, due to the compatible block-separable structures \eqref{eq:main-structure} and \eqref{eq:g-fstar-separable}, multiplying \eqref{eq:pp-expanded-lambda-x} by $P_j$ for $j=1,\ldots,m$, and \eqref{eq:pp-expanded-lambda-y} by $Q_\ell$ for $\ell=1,\ldots,n$, \eqref{eq:pp-general-as-algorithm} can be split blockwise as
\begin{align*}
    \nextx_j & = \inv{(I+\chi_{S(i)}(j)\tau_{j,i}\subdiff G_j)}(\thisx_j + P_j[\inv\TauTest_i \Lambda_i^*(\nexty-\thisy)-  \Tau_i K^*\nexty])
    \quad (j=1,\ldots,m),
    \\
    \nexty_\ell & = \inv{(I+\chi_{V(i+1)}(\ell)\sigma_{\ell,i+1}\subdiff F_\ell^*)}(\thisy_\ell + Q_\ell[\inv\SigmaTest_{i+1} \Lambda_i(\nextx-\thisx) + \Sigma_{i+1}K\nextx])
    \quad (\ell=1,\ldots,n).
\end{align*}
For $S(i)$ and $V(i+1)$ to have the intended meaning that only the corresponding blocks are updated, we need to choose $\Lambda_i$ such that
\begin{align}
    \label{eq:step-compatibility}
    \nextx_j&=\thisx_j\quad
    (j \not\in S(i)),
    \quad\text{and} & \nexty_\ell&=\thisy_\ell\quad (\ell \not\in V(i+1)).
\end{align}
Since $P_j\Tau_i=0$ for $j \not \in S(i)$, and $Q_\ell\Sigma_{i+1}=0$ for $\ell \not \in V(i+1)$, this is to say
\begin{subequations}
\label{eq:sanity-conditions}
\begin{align}
  \label{eq:sanity-conditions:primal}
  P_j
    \inv\TauTest_i \Lambda_i^*(\nexty-\thisy)
    &=0 && (j \not \in S(i)), \quad\text{and}
  \\
  \label{eq:sanity-conditions:dual}
  Q_\ell
    \inv\SigmaTest_{i+1} \Lambda_i(\nextx-\thisx)
    & = 0 && (\ell \not \in V(i+1)).
\end{align}
\end{subequations}

Taking $\Lambda_i=K\Tau_i^*\TauTest_i^*$ would allow computing $\nextx$ before $\nexty$ and satisfy \eqref{eq:sanity-conditions:primal}.
If we further had $K\Tau_i^*\TauTest_i^*=\omega_i\SigmaTest_{i+1}\Sigma_{i+1}K$, then also \eqref{eq:sanity-conditions:dual} would hold and \eqref{eq:precond-lambda} would reproduce $\Precond_{i+1}$ of \eqref{eq:test-precond-first}.
Symmetrically, $\Lambda_i=-\Sigma_{i+1}\SigmaTest_{i+1}K$ would make $\nexty$ independent of $\nextx$.
Such conditions will, however, rarely be satisfiable unless, deterministically, $S(i) \equiv \{1,\ldots,m\}$ and $V(i+1) \equiv \{1,\ldots,n\}$.
Nevertheless, motivated by this, we designate subsets of blocks of $\nextx$ and $\nexty$ to be updated independently of the other variable.
This are the subsets $\iset S(i)$ and $\iset V(i+1)$ in \eqref{eq:iset-tau} and  \eqref{eq:iset-sigma}. Then picking $\Lambda_i$ as in \eqref{eq:precond-lambda} achieves our objective:

\begin{lemma}
    \label{lemma:step-compatibility}
    Suppose \cref{ass:structural} holds. Then \eqref{eq:step-compatibility} and \eqref{eq:sanity-conditions} hold.
\end{lemma}

\begin{proof}
    We already know that \eqref{eq:sanity-conditions} implies\eqref{eq:step-compatibility}.
    Using \eqref{eq:precond-lambda},  \eqref{eq:sanity-conditions} can be rewritten
    \begin{align*}
      P_j
        \inv\TauTest_i[ \TauTest_i\iset\Tau_iK^*  - K^*\iset \Sigma_{i+1}^*\SigmaTest_{i+1}^*](\nexty-\thisy)
        &=0 && (j \not \in S(i)), \quad\text{and}
      \\
      Q_\ell
        \inv\SigmaTest_{i+1}[ K\iset\Tau_i^*\TauTest_i^*  - \SigmaTest_{i+1}\iset \Sigma_{i+1} K](\nextx-\thisx)
        & = 0 && (\ell \not \in V(i+1)).
    \end{align*}
    Clearly $P_j\iset\Tau_iK^*$ for $j \not \in S(i)$.
    Therefore, the first condition holds if $P_j K^*\iset \Sigma_{i+1}^*=0$ for $j \not\in S(i)$. This is to say $j \not \in \inv\Neigh(\iset V(i+1))$, which is guaranteed by \eqref{eq:nesting-s-v-cond}. Likewise, the second condition holds if $Q_\ell K\iset \Tau_i^*=0$ for $\ell \not\in V(i+1)$, which is also guaranteed by \eqref{eq:nesting-s-v-cond}.
\end{proof}

\subsection{Overall structure of the proposed method}

Setting $\Tau_i^\perp \defeq \Tau_i - \iset \Tau_i$, and $\Sigma_{i+1}^\perp \defeq \Sigma_{i+1} - \iset \Sigma_{i+1}$, we can now rewrite \eqref{eq:pp-general-as-algorithm} as
\begin{subequations}
\label{eq:cpoper-lambda-solved0}
\begin{align}
    \label{eq:cpoper-lambda-solved0-v}
    \nexxt{q} & \defeq \inv\TauTest_i K^* \iset \Sigma_{i+1}^* \SigmaTest_{i+1}^*(\nexty-\thisy)+\Tau_i^\perp K^*\nexty,
    \\
    \label{eq:cpoper-lambda-solved0-x}
    \nextx & \defeq \inv{(I+\Tau_i\subdiff G)}(\thisx - \iset \Tau_i K^*\thisy - \nexxt{q}),
    \\
    \label{eq:cpoper-lambda-solved0-z}
    \nexxt{z} & \defeq \inv\SigmaTest_{i+1}K\iset \Tau_i^*\TauTest_i^*(\nextx-\thisx)+\Sigma_{i+1}^\perp K\nextx,
    \\
    \label{eq:cpoper-lambda-solved0-y}
    \nexty & \defeq \inv{(I+\Sigma_{i+1}\subdiff F^*)}(\thisy+\iset \Sigma_{i+1}K\thisx+\nexxt{z}).
\end{align}
\end{subequations}
Due to the first part of \eqref{eq:nesting-iset-s-v-cond}, $P_j \nexxt{q}=0$ and $Q_\ell \nexxt{z}=0$ for $j \in \iset S(i)$ and $\ell \in \iset V(i+1)$.
The second part of \eqref{eq:nesting-iset-s-v-cond} moreover implies $\Tau_i^\perp K^*Q_\ell=0$ and $\Sigma_{i+1}^\perp KP_j$ for $\ell \in V(i+1) \setminus \iset V(i+1)$ and $j \in S(i) \setminus \iset S(i)$.
The first part of \eqref{eq:nesting-iset-s-v-cond} and \eqref{eq:nesting-s-v-cond} moreover imply $\Sigma_{i+1}^\perp\Theta_{i+1}=\Sigma_{i+1}\Theta_{i+1}=\inv\SigmaTest_{i+1} K \iset \Tau_i^*\TauTest_i^*$ and $\Tau_i^\perp B_{i+1}^*=\Tau_i B_{i+1}^*=\inv\TauTest_i K^* \iset \Sigma_{i+1} \SigmaTest_{i+1}$ for%
\begin{align*}
    \Theta_{i}
    &
    \defeq
    \sum_{j \in \iset S(i)} \sum_{\ell \in \Neigh(j)} \theta_{\ell,j,i} Q_\ell KP_{j}
    &&\text{with}\quad
    \theta_{\ell,j,i+1}
    \defeq
    \frac{\tau_{j,i}\tauTest_{j,i}}{\sigma_{\ell,i+1}\sigmaTest_{\ell,i+1}}
    \quad\text{and}
    \\
    B_{i}
    &
    \defeq
    \sum_{\ell \in \iset V(i+1)} \sum_{j \in \inv\Neigh(\ell)} b_{\ell,j,i} Q_\ell KP_{j}
    &&\text{with}\quad
    b_{\ell,j,i+1}
    \defeq
    \frac{\sigma_{\ell,i+1}\sigmaTest_{\ell,i+1}}{\tau_{j,i}\tauTest_{j,i}},
\end{align*}
Setting $\nexxt{\iset x} \defeq \sum_{j \in \iset S(i)} P_j \nextx$ and $\nexxt{\iset y} \defeq \sum_{\ell \in \iset V(i+1)} Q_\ell \nextx$ we therefore obtain
\begin{subequations}%
\label{eq:cpoper-lambda-solved1}%
\begin{align}%
   \nexxt{q} & \defeq \inv\TauTest_i K^* \iset \Sigma_{i+1}^* \SigmaTest_{i+1}^*(\nexxt{\iset y}-\thisy)+\Tau_i^\perp K^*\nexxt{\iset y}
   =\Tau_i^\perp[B_{i+1}^*(\nexxt{\iset y}-\thisy)+\nexxt{\iset y}]
   \quad\text{and}
   \\
   \nexxt{z} & \defeq \inv\SigmaTest_{i+1}K\iset \Tau_i^*\TauTest_i^*(\nexxt{\iset x}-\thisx)+\Sigma_{i+1}^\perp K\nexxt{\iset x}
   =\Sigma_{i+1}^\perp[\Theta_{i+1}(\nexxt{\iset x}-\thisx)+\nexxt{\iset x}].
\end{align}%
\end{subequations}%

Introducing $\nexxt{w}$ and $\nexxt{v}$ such that $\SigmaTest_{i+1}^\perp\nexxt{w}=\nexxt{z}$ and $\TauTest_{i}^\perp\nexxt{v}=\nexxt{q}$, and using \eqref{eq:g-fstar-separable}, we can write the method given by \eqref{eq:cpoper-lambda-solved0} and \eqref{eq:cpoper-lambda-solved1} as%
\begin{subequations}%
\label{eq:cpoper-theta}%
\begin{align}%
    \nexxt{\iset x} & \defeq \inv{(I+\iset \Tau_i\subdiff G)}(\thisx - \iset \Tau_i K^*\thisy), \\
    \nexxt{\iset y} & \defeq \inv{(I+\iset \Sigma_{i+1}\subdiff F^*)}(\thisy+\iset \Sigma_{i+1}K\thisx), \\
    \nexxt{w} & \defeq \Theta_{i+1}(\nexxt{\iset x}-\thisx)+\nexxt{\iset x}, \\
    \nexxt{v} & \defeq B_{i+1}^*(\nexxt{\iset y}-\thisy)+\nexxt{\iset y}, \\
    \nextx & \defeq \inv{(I+\Tau_i^\perp\subdiff G)}(\nexxt{\iset x}-\Tau_i^\perp \nexxt{v}), \\
    \nexty & \defeq \inv{(I+\Sigma_{i+1}^\perp\subdiff F^*)}(\nexxt{\iset y}+\Sigma_{i+1}^\perp\nexxt{w}).
\end{align}%
\end{subequations}%
Due to \eqref{eq:g-fstar-separable}, these operations can further be split into blockwise operations with no dependencies on so far uncomputed blocks. We delay doing this explicitly until we are ready to present our final \cref{alg:alg-blockcp,alg:alg-blockcp-fulldual} towards the end of the theoretical part of the paper.

We conclude the present structural development by explicitly stating what we have proved in the preceding paragraphs:

\begin{lemma}
    \label{lemma:pp-equiv}
    Suppose \cref{ass:structural} holds.
    Then \eqref{eq:cpoper-theta} is equivalent to \eqref{eq:pp}.
\end{lemma}

\section{A basic convergence estimate}
\label{sec:block}

Now that we have the overall structure of the proposed algorithms established in \eqref{eq:cpoper-theta}, we need to develop rules for the step length and testing parameters that yield a convergent method. This will require, in particular, the positive semi-definiteness of $\Test_{i+1}\Precond_{i+1}$, as we recall from the discussion leading up to \eqref{eq:basic-convergence-estimate}. Indeed, since in general, without any strong convexity, we can only obtain gap (and weak) convergence, we need to refine that argument.

In \crefrange{sec:gap}{sec:gap-satisfaction}, we will derive some quite technical conditions that the step length parameters, testing parameters, and block selection probabilities need to satisfy. From these basic estimates, we then develop explicit convergence rates in the next section.
In the final \cref{sec:sampling}, we will also discuss permissible sampling patterns.

\subsection{A bound on ergodic duality gaps}
\label{sec:gap}

Recall the basis of the testing technique, \eqref{eq:testing0}. In the single-block case ($\Tau_i=\tau_i I$, $\Sigma_{i+1}=\sigma_{i+1} I$, $\TauTest_{i+1}=\tauTest_{i+1} I$, and $\SigmaTest_{i+1}=\sigmaTest_{i+1} I$), instead of using $\realoptu \in \inv H(0)$ and the operator-relative monotonicity \eqref{eq:xi-monotonicity} to eliminate $H$, using the convexity of $G$ and $F^*$ we can also estimate
\begin{equation}
    \label{eq:h-gap0}
    \begin{split}
    \iprod{H(\nextu)}{\nextu-\realoptu}_{\Test_{i+1}\Step_{i+1}}
    &
    \ge
    \tauTest_i\tau_i[G(\nextx)-G(\realoptx)]
    +\sigmaTest_{i+1}\sigma_{i+1}[F^*(\nexty)-F^*(\realopty)]
    \\\MoveEqLeft[-1]
    +\tauTest_i\tau_i\iprod{K^*\nexty}{\nextx-\realoptx}
    -\sigmaTest_{i+1}\sigma_i\iprod{K\nextx}{\nexty-\realopty} =: \tilde\gap_{i+1}.
    \end{split}
\end{equation}
With this, \eqref{eq:basic-convergence-estimate} can be improved to
\begin{equation*}
  \frac{1}{2}\norm{u^N-\realoptu}_{\Test_{N+1}\Precond_{N+1}}^2
  +
  \sum_{i=0}^{N-1}\left(\tilde\gap_{i+1} + \frac{1}{2}\norm{\nextu-\thisu}_{\Test_{i+1}\Precond_{i+1}}^2\right)
  \le
  \frac{1}{2}\norm{u^0-\realoptu}_{\Test_{1}\Precond_{1}}^2.
\end{equation*}

We would like to develop the ``preliminary gaps'' $\tilde\gap_{i+1}$ into a (Lagrangian) duality gap
\begin{equation}
    \label{eq:gap}
    \gap(x, y) \defeq
    \bigl(G(x) + \iprod{\realopty}{K x}  - F^*(\realopty)\bigr)
    -\bigl(G(\realoptx) + \iprod{y}{K \realoptx} -  F^*(y)\bigr).
\end{equation}
The first obstacle we face are the differing factors in front of $G$ and $F^*$. This suggests to impose $\tauTest_i\tau_i=\sigmaTest_{i+1}\sigma_{i+1}$.
For the PDHGM, it however turns out that $\tauTest_i\tau_i=\sigmaTest_i\sigma_i$.
After taking care of some technical details, this can be dealt with by an index realignment argument \cite{tuomov-proxtest}.

With multiple blocks, we can get a similar estimate as \eqref{eq:h-gap0} with factor the $\tauTest_{j,i}\tau_{j,i}$ in front of $G_j$ and $\sigmaTest_{\ell,i+1}\sigma_{\ell,i+1}$ in front of $F^*_\ell$. To derive a gap estimate, the preceding discussion suggests to impose $\Tau_i\TauTest_i=\bar\eta_i I$ and $\Sigma_{i+1}\SigmaTest_{i+1}=\bar\eta_i I$ or $\Sigma_i\SigmaTest_i=\bar\eta_i I$ for some scalar $\bar\eta_i>0$.
This kind of \emph{coupling} between the blocks will be one of the main restrictions that we face in the development our method.
In the stochastic setting, it turns out \cite{tuomov-proxtest} that we can relax the coupling slightly: do it in expectation. Correspondingly, we assume for some $\bar\eta_i > 0$ either
\begin{subequations}%
\label{eq:eta-conds}
\begin{align}%
    \label{eq:cond-eta}
    \E[\Tau_i^*\TauTest_i^*] & = \bar\eta_i I, 
    \quad\text{and} &
    \E[\SigmaTest_{i+1}\Sigma_{i+1}] & =\bar\eta_i I,
    \quad(i \ge 1),
    \qquad\text{or}
    \\
    \label{eq:cond-etatwo}
    \E[\Tau_i^*\TauTest_i^*] & = \bar\eta_i I, 
    \quad\text{and} &
    \E[\SigmaTest_{i}\Sigma_{i}] &=\bar\eta_i I,
    \quad(i \ge 1).
\end{align}%
\end{subequations}%
The second condition is an extension of what we saw the standard PDHGM to satisfy.
The first condition, which is off-by-one compared to the second, will, however, be the only alternative that doubly-stochastic methods can satisfy.

A further difficulty with developing  \eqref{eq:h-gap0} into a gap estimate are the remaining terms involving $K$. Even after rearrangements we can only get an ergodic estimate \cite{tuomov-proxtest}.
To express such estimates, corresponding to the conditions \eqref{eq:cond-eta} and \eqref{eq:cond-etatwo}, we introduce
\begin{equation}
    \label{eq:zeta}
    \zeta_N \defeq \sum_{i=0}^{N-1} \bar\eta_i
    \quad\text{and}\quad
    \zeta_{*,N} \defeq \sum_{i=1}^{N-1} \bar\eta_i,
\end{equation}
and the ergodic sequences
\begin{align}
    \label{eq:tildexnyn}
    \tilde x_{N} & \defeq \inv\zeta_{N}\E\Biggl[\sum_{i=0}^{N-1} \Tau_i^*\TauTest_i^* \nextx\Biggr],
    &
    \tilde y_{N} & \defeq \inv\zeta_{N}\E\Biggl[\sum_{i=0}^{N-1} \Sigma_{i+1}^*\SigmaTest_{i+1}^* \nexty\Biggr],
    \\
    \label{eq:tildexnyntwo}
    \tilde x_{*,N} & \defeq \inv\zeta_{*,N}\E\Biggl[\sum_{i=1}^{N-1} \Tau_i^*\TauTest_i^* \nextx\Biggr],
     &
    \tilde y_{*,N} & \defeq \inv\zeta_{*,N}\E\Biggl[\sum_{i=1}^{N-1} \Sigma_i^*\SigmaTest_i^* \thisy\Biggr].
\end{align}
The coupling conditions \eqref{eq:cond-eta} and \eqref{eq:cond-etatwo} then produce two different ergodic gaps, $\gap(\tilde x_N, \tilde y_N)$ and $\gap(\tilde x_{*,N}, \tilde y_{*,N})$.
We demonstrate this in the next theorem from \cite{tuomov-proxtest}. It forms the basis for our work in the remaining sections.
The fundamental arguments for the proof are those that led to \eqref{eq:basic-convergence-estimate}, however, the gap estimate requires significant additional technical work.

\begin{theorem}
    \label{cor:convergence-result-stochastic}
    Suppose \cref{ass:structural} (main structural condition) holds with $\Test_{i+1}\Precond_{i+1}$ positive semi-definite. Write $\Gamma \defeq \sum_{j=1}^m \gamma_j P_j$ for $\gamma_j \ge 0$ the factor of (strong) convexity of $G_j$.
    With $\tilde\Gamma=\sum_{j=1}^m \tilde\gamma_j P_j \in \linear(X; X)$, assuming one of the following alternatives to hold, let
    \begin{equation}
        \label{eq:tildegn}
        \tilde g_N \defeq
        \begin{cases}
             0, & 0 \le \tilde\Gamma \le \Gamma, \\
             \zeta_N \gap(\tilde x_N, \tilde y_N), & 0 \le \tilde\Gamma \le \Gamma/2; \text{ \eqref{eq:cond-eta} holds}, \\
             \zeta_{*,N}\gap(\tilde x_{*,N}, \tilde y_{*,N}), & 0 \le \tilde\Gamma \le \Gamma/2; \text{ \eqref{eq:cond-etatwo} holds}.
        \end{cases}
    \end{equation}
    Also define
    \begin{align*}
        \GammaLift{i+1}(\tilde\Gamma)& \defeq 
        \begin{pmatrix}
            2\Tau_i\tilde\Gamma
            & 2\Tau_i K^*
            \\
            - 2\Sigma_{i+1} K
            & 0
        \end{pmatrix}
        \quad\text{and}&
        \MetricDiff_{i+1}(\tilde\Gamma)
        &\defeq
        \Test_{i+2}\Precond_{i+2}-\Test_{i+1}(\GammaLift{i+1}(\tilde\Gamma)+\Precond_{i+1}).
    \end{align*}
    Then the iterates $\thisu=(\thisx, \thisy)$ of \eqref{eq:pp} satisfy for any $\realoptu \in \inv H(0)$ the estimate
    \begin{align}
        \label{eq:convergence-result-main-stochastic}
        \frac{1}{2}\E\bigl[
            \norm{u^N-\realoptu}_{\Test_N\Precond_N}^2\bigr]
        + \tilde g_N
        &
        \le
        \frac{1}{2}
        \norm{u^0-\realoptu}_{\Test_1 \Precond_1}^2
        \\ \MoveEqLeft[-1] \notag
        + \sum_{i=0}^{N-1}
        \frac{1}{2}\E\bigl[
            \norm{\nextu-\realoptu}_{\MetricDiff_{i+1}(\tilde\Gamma)}^2
            -\norm{\nextu-\thisu}_{\Test_{i+1}\Precond_{i+1}}^2
        \bigr].
    \end{align}
\end{theorem}

\begin{proof}
    This is \cite[Theorem 5.5]{tuomov-proxtest} with $\Penalty_{i+1}(\tilde\Gamma) \defeq \frac{1}{2}\norm{\nextu-\realoptu}_{\MetricDiff_{i+1}(\tilde\Gamma)}^2-\frac{1}{2}\norm{\nextu-\thisu}_{\Test_{i+1}\Precond_{i+1}}^2$ and the condition $\tilde\Gamma=\Gamma$ relaxed to $0 \le \tilde\Gamma \le \Gamma$, which is possible because if $g_j$ is strongly convex with factor $\gamma_j>0$, it is strongly convex with any smaller non-negative factor.
    Moreover, \cite[Example 5.1]{tuomov-proxtest} shows that the blockwise structure \eqref{eq:g-fstar-separable}, \eqref{eq:main-structure} has an ergodic convexity property that produces the gaps $\gap(\tilde x_N, \tilde y_N)$ and $\gap(\tilde x_{*,N}, \tilde y_{*,N})$
\end{proof}

 In the standard PDHGM we can ensure $\MetricDiff_{i+1}(\tilde\Gamma) \simeq 0$ \cite{tuomov-proxtest}. However, in our present setting, we will not generally be able to enforce this, so these operators will introduce a penalty in \eqref{eq:convergence-result-main-stochastic}. A lot of our remaining work will consist of controlling this penalty. We also need to estimate from below and show that $\Test_{N}\Precond_{N}$ is positive semi-definite.

\subsection{Notations and assumptions}

For convenience, we introduce
\[
    \begin{aligned}
    \hat\tau_{j,i} & \defeq \tau_{j,i}\chi_{S(i)}(j),
    &
    \hat\sigma_{\ell,i} &\defeq \sigma_{\ell,i}\chi_{V(i)}(\ell),
    \\
    \pi_{j,i} & \defeq \P[j \in S(i) \mid \SAlg_{i-1}],
    &
    \nu_{\ell,i+1} & \defeq \P[\ell \in V(i+1) \mid \SAlg_{i-1}],
    \\
    \iset\pi_{j,i} & \defeq \P[j \in \iset S(i) \mid \SAlg_{i-1}],
    \quad\text{and} 
    &
    \iset\nu_{\ell,i+1} & \defeq \P[\ell \in \iset V(i+1) \mid \SAlg_{i-1}].
    \end{aligned}
\]
The first two denote ``effective'' step lengths on iteration $i$, while the rest are shorthands for the probabilities of the primal block $j$ or the dual block $\ell$ being contained in the corresponding set on iteration $i$.
Recalling \eqref{eq:precond-lambda}, we also write
\begin{align}
    \label{eq:lambda-choice}
    \Lambda_{i} & =  \sum_{j=1}^m \sum_{\ell \in \Neigh(j)} \lambda_{\ell,j,i} Q_\ell KP_{j}
    \quad\text{with}
    &
    \lambda_{\ell,j,i}
    & \defeq
        \tauTest_{j,i}\hat\tau_{j,i}\chi_{\iset S(i)}(j)
        -\sigmaTest_{\ell,i+1}\hat\sigma_{\ell,i+1}\chi_{\iset V(i+1)}(\ell).
\end{align}

We require the following technical assumption, which we will verify through explicit step length and testing parameter update rule development in the next section. We indicate the rough intended use of each condition in parentheses after the statement.

\begin{assumption}[step length parameter restrictions]
    \label{ass:step}
    We assume for each $i \in \N$ the following, constants independent of $i$, and  the \emph{same alternatives} holding for each $i$:
    \begin{enumerate}[label=(\alph*)]
        \item\label{item:sigmatest-choice-1}
        We are given $(\kappa_1,\ldots,\kappa_n) \in \kappafamily$ (see \cref{def:kappa}), and for some $\delta \in (0,1)$,
        \begin{equation*}
            (1-\delta)\sigmaTest_{\ell,i+1}
            \ge
            \kappa_\ell(\lambda_{\ell,1,i}^2\inv\tauTest_{1,i}, \ldots, \lambda_{\ell,m,i}^2\inv\tauTest_{m,i}) 
            \quad (\ell=1,\ldots,n).
        \end{equation*}
        (This condition generalises the condition $\tau\sigma\norm{K}^2 < 1$ for the standard PDHGM, needed to ensure the positivity of the local metric $\Test_{i+1}\Precond_{i+1}$.)

        \item\label{item:eta-relationships}
        We are given $\eta_i \in \Random(\SAlg_{i-1}; (0, \infty))$ and $\eta_{\tau,i}^\perp, \eta_{\sigma,i}^\perp \in \Random(\SAlg_{i-1}; [0, \infty))$ such that $\eta_{i+1} \ge \eta_i$,
        \begin{equation*}
            \eta_i \cdot \min_j(\pi_{j,i}-\iset \pi_{j,i}) \ge \eta_{\tau,i}^\perp,
            \quad\text{and}\quad
            \eta_{i+1} \cdot \min_\ell (\nu_{\ell,i+1}-\iset \nu_{\ell,i+1}) \ge \eta_{\sigma,i}^\perp.
        \end{equation*}
        (This is needed to annihilate the off-diagonal of $\MetricDiff_{i+1}$ in the penalty term.)

        \item\label{item:eta-options}
        Either
        \begin{enumerate}[label=(c-\roman*),widest={(c-ii)},leftmargin=*]
            \item \label{item:eta-doubly}
            $\E[\eta_{\tau,i}^\perp-\eta_{\sigma,i}^\perp]=\text{constant}$; or

            \item\label{item:eta-singly}
            $\eta_{\tau,i}^\perp=0$ and $\eta_{\sigma,i}^\perp = \eta_{i+1}$.
        \end{enumerate}
        (These are needed to ensure the coupling conditions \eqref{eq:cond-eta} or \eqref{eq:cond-etatwo}, respectively.)

        \item\label{item:step-update-formulas}
        The step lengths parameters
        \begin{subequations}
        \label{eq:step-update-formulas}
        \begin{align}
            \label{eq:step-update-formulas-x}
            \tau_{j,i} & =
            \begin{cases}
                \frac{\eta_{i} - \tauTest_{j,i-1}\tau_{j,i-1}\chi_{S(i-1) \setminus \iset S(i-1)}(j)}{\tauTest_{j,i}\iset \pi_{j,i}},
                &
                j \in \iset S(i), 
                \\
                \frac{\eta_{\tau,i}^\perp}{\tauTest_{j,i}(\pi_{j,i}-\iset \pi_{j,i})}, 
                &
                j \in S(i) \setminus \iset S(i),
            \end{cases}
            \\
            \label{eq:step-update-formulas-y}
            \sigma_{j,i+1} & =
            \begin{cases}
                \frac{\eta_{i} - \sigmaTest_{j,i}\sigma_{j,i}\chi_{V(i) \setminus \iset V(i)}(j)}{\sigmaTest_{j,i+1}\iset \nu_{\ell,i+1}},
                &
                j \in \iset V(i+1),
                \\
                \frac{\eta_{\sigma,i}^\perp}{\sigmaTest_{j,i+1}(\nu_{\ell,i+1}-\iset\nu_{\ell,i+1})},
                &
                j \in V(i+1) \setminus \iset V(i+1).
            \end{cases}
        \end{align}%
        \end{subequations}%
        For $i=0$ we take $\tau_{j,-1} \defeq 0$ and $\sigma_{j,0} \defeq 0$.\\
        (This rule is also needed to annihilate the off-diagonal of $\MetricDiff_{i+1}$ in the penalty term.)

        \item\label{item:c2-x-primeprime}
        Let $\gamma_j \ge 0$ by the factor of (strong) convexity of $G_j$, and  $\tilde\gamma_j \in [0,\gamma_j]$, ($j=1,\ldots,m$). Also let $\alpha_i>0$ and define%
        \begin{subequations}%
        \begin{align}%
            \label{eq:q-j}
            q_{j,i+2}(\tilde\gamma_j)
            & \defeq
            \bigl(
            \E[\tauTest_{j,i+1} - \tauTest_{j,i}(1 + 2\hat\tau_{j,i}\tilde\gamma_j)|\SAlg_i]
            \\
            \notag
            &
            \phantom { = = }
            +
            \alpha_i\abs{\E[\tauTest_{j,i+1} - \tauTest_{j,i}(1 + 2\hat\tau_{j,i}\tilde\gamma_j)|\SAlg_i]}
            -\delta\tauTest_{j,i}\bigr)\chi_{S(i)}(j),
            \quad\text{and}
            \\
            \label{eq:h-j}
            h_{j,i+2}(\tilde\gamma_j)
            & \defeq
            \E[\tauTest_{j,i+1} - \tauTest_{j,i}(1 + 2\hat\tau_{j,i}\tilde\gamma_j)|\SAlg_{i-1}]
            \\
            \notag
            &
            \phantom{ = }
            +\inv\alpha_i\abs{\E[\tauTest_{j,i+1} - \tauTest_{j,i}(1 + 2\hat\tau_{j,i}\tilde\gamma_j)|\SAlg_i]}.
        \end{align}%
        \end{subequations}%
        Then for some $C_x>0$ either
        \begin{subequations}
        \label{eq:c2-x-primeprime}
        \begin{align}
            \label{eq:c2-x-primeprime-bound}
            \norm{\nextx_j-\realoptx_j}^2 & \le C_x
            && (j=1,\ldots,m)
            \quad\text{or}\quad
            \\
            \label{eq:c2-x-primeprime-zero}
            h_{j,i+2}(\tilde\gamma_j) & \le 0
            \quad\text{and}\quad
            q_{j,i+2}(\tilde\gamma_j) \le 0
            && (j=1,\ldots,m),
        \end{align}
        \end{subequations}
        (This is needed to bound the primal components in the penalty term.)

        \item\label{item:c2-y-primeprime}
        For some $C_y>0$ either
        \begin{subequations}
        \label{eq:c2-y-primeprime}
        \begin{align}
            \label{eq:c2-y-primeprime-bound}
            \E[\sigmaTest_{\ell,i+2}-\sigmaTest_{\ell,i+1}|\SAlg_i] & \ge 0,
            \quad
            \norm{\nexty_\ell-\realopty_\ell}^2 \le C_y
            && (\ell=1,\ldots,n)
            \quad\text{or}\quad
            \\
            \label{eq:c2-y-primeprime-zero} 
            \E[\sigmaTest_{\ell,i+2}-\sigmaTest_{\ell,i+1}|\SAlg_i] & =0 
            && (\ell=1,\ldots,n).
        \end{align}
        \end{subequations}%
        (This is needed to bound the dual components in the penalty term.)
    \end{enumerate}
\end{assumption}

It is important that \cref{ass:step} is consistent with \cref{ass:structural}, in particular that the step lengths generated by the former are non-negative. We will prove this in \cref{lemma:c2}. Before this, we start the main goal of the present section, the following specialisation of \cref{cor:convergence-result-stochastic}.

\begin{proposition}
    \label{prop:general-block-conditions}
    Suppose \cref{ass:structural} (main structural condition) and \cref{ass:step} (step length restrictions) hold.
    Then the iterates of \eqref{eq:pp} satisfy for any $\realoptu \in \inv H(0)$ the estimate%
    \begin{equation}
        \label{eq:convergence-result-stochastic-block}
        \sum_{j=1}^m \frac{\delta}{2\E[\inv\tauTest_{j,N}]} \cdot \E\bigl[\norm{x_j^N-\realoptx_j}\bigr]^2
        + \tilde g_N
        \le
        \frac{1}{2}\norm{u^0-\realoptu}_{\Test_0 \Precond_0}^2
        +
        \sum_{j=1}^m
        \frac{1}{2}d_{j,N}^x(\tilde\gamma_j)
        +
        \sum_{\ell=1}^n \frac{1}{2}d_{\ell,N}^y,
    \end{equation}
    where%
    \begin{subequations}%
    \label{eq:convergence-result-stochastic-block-parts}
    \begin{align}
        \label{eq:tildegn2}
        \tilde g_N &\defeq
        \begin{cases}
             \zeta_N \gap(\tilde x_N, \tilde y_N), & \text{case \ref{item:eta-doubly} and } \tilde\gamma_j \le \gamma_j/2 \text{ for all } j, \\
             \zeta_{*,N}\gap(\tilde x_{*,N}, \tilde y_{*,N}), & \text{case \ref{item:eta-singly} and } \tilde\gamma_j \le \gamma_j/2 \text{ for all } j, \\
             0, & \text{otherwise},
        \end{cases}
        \\
        d_{j,N}^x(\tilde\gamma_j) & \defeq
        \sum_{i=0}^{N-1} \boundfx_{j,i+2}(\tilde\gamma_j)
        \quad
        d_{\ell,N}^y \defeq \sum_{i=0}^{N-1} \boundfy_{\ell,i+2},
        \\
        \label{eq:bx-j-choice}
        \boundfx_{j,i+2}(\tilde\gamma_j)
        &
        \defeq
        4C_x \E[\max\{0, q_{j,i+2}(\tilde\gamma_j)\}]
        +
        C_x \E[\max\{0, h_{j,i+2}(\tilde\gamma_j)\}],
        \quad\text{and}
        \\
        \label{eq:by-ell-choice}
        \boundfy_{\ell,i+2}
        & \defeq
        9 C_y
        \E[\sigmaTest_{\ell,i+2}-\sigmaTest_{\ell,i+1}].
    \end{align}%
    \end{subequations}
\end{proposition}

\begin{proof}
    We use \cref{lemma:verif-eta} or \cref{lemma:verif-etatwo} (to follow) to verify one of the coupling conditions \eqref{eq:cond-eta} or \eqref{eq:cond-etatwo}. Then we obtain \eqref{eq:convergence-result-main-stochastic} from
    \cref{cor:convergence-result-stochastic}.
    Next, we use \cref{lemma:cond-u-to-x-prime-block,lemma:cond-delta-expectation-x-primeprime} (to follow) to estimate $\Test_{N+1}\Precond_{N+1} \ge \begin{psmallmatrix} \delta \TauTest_{N} & 0 \\ 0 & 0 \end{psmallmatrix}$ and
    \[
        \E\bigl[\norm{\nextu-\realoptu}_{\MetricDiff_{i+1}(\tilde\Gamma)}^2-\norm{\nextu-\thisu}_{\Test_{i+1}\Precond_{i+1}}^2\bigr]
        \le \sum_{j=1}^m \boundfx_{j,i+2}(\tilde\gamma_j) + \sum_{\ell=1}^n  \boundfy_{\ell,i+2}.
    \]
    Therefore \eqref{eq:convergence-result-main-stochastic} yields
    \begin{equation*}
        \frac{\delta}{2}\E\bigl[\norm{x^N-\realoptx}_{\TauTest_N}^2\bigr]
        +
        \tilde g_N
        \le
        \frac{1}{2}\norm{u^0-\realoptu}_{\Test_0 \Precond_0}^2
        +
        \frac{1}{2}\sum_{i=0}^{N-1}\biggl(\sum_{j=1}^m \boundfx_{j,i+2}(\tilde\gamma_j) + \sum_{\ell=1}^n  \boundfy_{\ell,i+2}\biggr).
    \end{equation*}
    By Hölder's inequality
    \[
        \E\bigl[\norm{x^N-\realoptx}_{\TauTest_N}^2\bigr]
        =
        \sum_{j=1}^m
        \E\bigl[
             \tauTest_{j,N} \norm{x_j^N-\realoptx_j}^2
        \bigr]
        \ge
        \sum_{j=1}^m
        \E\bigl[\norm{x_j^N-\realoptx_j}\bigr]^2 / \E[\inv\tauTest_{j,N}].
    \]
    The estimate \eqref{eq:convergence-result-stochastic-block} is now immediate.
\end{proof}

\subsection{Lower bound on the local metric}
\label{sec:stageone-c0}

\begin{lemma}
    \label{lemma:cond-u-to-x-prime-block}
    Suppose \cref{ass:structural} (main structural condition) and \cref{ass:step}\,\ref{item:sigmatest-choice-1} hold.
    Then $\Test_{i+1}\Precond_{i+1} \ge \begin{psmallmatrix} \delta \TauTest_i & 0 \\ 0 & 0 \end{psmallmatrix}$.
\end{lemma}
\begin{proof}
    Since $\TauTest_{i+1}$ is self-adjoint and positive definite, using \eqref{eq:zm} and Cauchy's inequality, for any $\delta \in (0, 1)$ we deduce
    \begin{equation}
        \label{eq:test-precond-first-estim}
        \Test_{i+1} \Precond_{i+1}
        =
        \begin{pmatrix}
            \TauTest_{i}
            &
            -\Lambda_{i}^*
            \\
            -\Lambda_{i}
            &
            \SigmaTest_{i+1}
        \end{pmatrix}
        \ge
        \begin{pmatrix}
            \delta \TauTest_{i}
            &
            0
            \\
            0
            &
            \SigmaTest_{i+1}
            - \frac{1}{1-\delta} \Lambda_{i}\inv\TauTest_{i}\Lambda_{i}^*
        \end{pmatrix}.
    \end{equation}
    We therefore require $(1-\delta)\SigmaTest_{i+1} \ge \Lambda_i\inv\TauTest_i\Lambda_i^*$, which can be expanded as
    \begin{equation}
        \label{eq:c1-prime-lambda-expanded1}
        (1-\delta)\sum_{\ell=1}^n \sigmaTest_{\ell,i+1} Q_\ell
        \ge
        \sum_{j=1}^m \sum_{\ell,k=1}^n \lambda_{\ell,j,i}\lambda_{k,j,i}\inv\tauTest_{j,i} Q_\ell K P_j K^* Q_k.
    \end{equation}
    This follows from \cref{def:kappa}\,\ref{item:kappa-estimate} with $z_{\ell,j} \defeq \lambda_{\ell,j,i}^2\inv\tauTest_{j,i}$.
\end{proof}

\subsection{Bounds on the penalty terms}
\label{sec:stageone-c2}

The structural setup \eqref{eq:main-structure} gives
\begin{align}
    \label{eq:deltax-expansion}
    \MetricDiff_{i+1}(\tilde\Gamma)  & =
    \begin{pmatrix}
         \TauTest_{i+1} - \TauTest_i(I + 2\Tau_i\tilde\Gamma)
        &
        \Lambda_{i}^* - \Lambda_{i+1}^* - 2 \TauTest_i \Tau_i K^* 
        \\
        2\SigmaTest_{i+1} \Sigma_{i+1}K + \Lambda_{i} - \Lambda_{i+1}
        & \SigmaTest_{i+2} - \SigmaTest_{i+1}
    \end{pmatrix}
    \\
    \notag
    &
    \simeq
    \begin{pmatrix}
         \TauTest_{i+1} - \TauTest_i(I + 2\Tau_i\tilde\Gamma)
        &
        A_{i+2}^*
        \\
        A_{i+2}
        & \SigmaTest_{i+2} - \SigmaTest_{i+1}
    \end{pmatrix}%
\quad\text{for}
    \\
    \notag
    A_{i+2} & \defeq
    (\SigmaTest_{i+1} \Sigma_{i+1} K- \Lambda_{i+1})
    +(\Lambda_{i} - K\Tau_i^*\TauTest_i^*).
\end{align}

\begin{lemma}
    \label{lemma:c2}
    Suppose \cref{ass:structural} (main structural condition) and \cref{ass:step}\,\ref{item:eta-relationships} \& \ref{item:step-update-formulas} hold.
    Then
    \begin{equation}
        \label{eq:cond-ai-expectation0-all}
        \E[A_{i+2}|\SAlg_i](\nextx-\thisx) = 0,
        \quad
        \E[A_{i+2}^*|\SAlg_i](\nexty-\thisy) = 0,
        \quad
        \E[A_{i+2}^*|\SAlg_{i-1}] = 0.
    \end{equation}
    Moreover, if $\tauTest_{j,i},\sigmaTest_{\ell,i+1} > 0$ for all $\in \N$, then $\tau_{j,i},\sigma_{\ell,i+1} \ge 0$ for all $i \in \N$. In particular, \cref{ass:step} is consistent with \cref{ass:structural} requiring  $\tau_{j,i},\sigma_{\ell,i+1} \ge 0$ and $\tauTest_{j,i},\sigmaTest_{\ell,i+1} > 0$ for all $i \in \N$; $j=1,\ldots,m$; and $\ell=1,\ldots,n$.
\end{lemma}

\begin{proof}
    We start by claiming that
    \begin{equation}
        \label{eq:lambda-cond0}
        \E[\lambda_{\ell,j,i+1}|\SAlg_i]
        =
        \sigmaTest_{\ell,i+1}\hat\sigma_{\ell,i+1}(1-\chi_{\iset V(i+1)}(\ell))
        -\tauTest_{j,i}\hat\tau_{j,i}(1-\chi_{\iset S(i)}(j))
    \end{equation}
    whenever $\ell \in \Neigh(j)$.
    Indeed, inserting \eqref{eq:lambda-choice} into  \eqref{eq:lambda-cond0}, we see the former to be satisfied if (for any given $\eta_{i+1}$),
    \begin{subequations}
    \label{eq:step-conds}
    \begin{align}
        \label{eq:step-conds-primal}
        \E[\tauTest_{j,i+1}\hat\tau_{j,i+1}\chi_{\iset S(i+1)}(j)|\SAlg_i]
        & = \eta_{i+1} - \tauTest_{j,i}\hat\tau_{j,i}(1-\chi_{\iset S(i)}(j)) \ge 0, \quad\text{and} \\
        \label{eq:step-conds-dual}
        \E[\sigmaTest_{\ell,i+2}\hat\sigma_{\ell,i+2}\chi_{\iset V(i+2)}(\ell)|\SAlg_i]
        & = \eta_{i+1} - \sigmaTest_{\ell,i+1}\hat\sigma_{\ell,i+1}(1-\chi_{\iset V(i+1)}(\ell)) \ge 0,
    \end{align}
    \end{subequations}
    over $j=1,\ldots,m$; $\ell=1,\ldots,n$; and $i \ge -1$, taking $\iset S(-1)=\{1,\ldots,m\}$ and $\iset V(0)=\{1,\ldots,n\}$.

    We can also write \eqref{eq:step-conds-primal} as
    \begin{equation}
        \label{eq:step-conds-primal-v2}
        \E[\tauTest_{j,i+1}\hat\tau_{j,i+1}\chi_{\iset S(i+1)}(j)|\SAlg_i]
        = \eta_{i+1} - \tauTest_{j,i}\tau_{j,i}\chi_{S(i) \setminus \iset S(i)}(j) \ge 0.
    \end{equation}
    If $j \not\in S(i)\setminus \iset S(i)$, since $\eta_{i+1} \ge 0$ by \cref{ass:step}\,\ref{item:eta-relationships}, it is clear that the inequality in \eqref{eq:step-conds-primal-v2} holds.
    If $j \in S(i) \setminus \iset S(i)$, using the corresponding case of \eqref{eq:step-update-formulas-x}, we rewrite the inequality as $\eta_{i+1} \ge \eta_{\tau,i}^\perp/(\pi_{j,i}-\iset \pi_{j,i})$. This is verified by \cref{ass:step}\,\ref{item:eta-relationships}.  Comparing to \cref{ass:step}\,\ref{item:step-update-formulas}, the inequality in \eqref{eq:step-conds-primal-v2} now inductively verifies, as claimed, $\tau_{j,i+1} \ge 0$ for all $i \in \N$ provided $\tauTest_{j,i} > 0$ for all $i \in \N$

    To verify the equality in \eqref{eq:step-conds-primal-v2}, let $\SAlg_i^+ \supset \SAlg_i$ be the smallest $\sigma$-algebra also containing the set $\{\omega \in \Omega \mid j \in \iset S(\omega)(i+1)\}$ (now not abusing notation for random variables, with $\omega$ standing for the random realisation that we typically omit).
    By \cref{ass:step}\,\ref{item:step-update-formulas}, more precisely \eqref{eq:step-update-formulas-y} shifted from $i$ to $i+1$, we see that $\tauTest_{j,i+1}\tau_{j,i+1}$ is $\SAlg_i^+$-measurable. Therefore, by standard properties of conditional expectations (see, e.g., \cite{shiryaev1996probability})
    \begin{equation}
        \label{eq:c2-stoch-argument1}
        \begin{split}
        \E[\tauTest_{j,i+1}\hat\tau_{j,i+1}\chi_{\iset S(i+1)}(j)|\SAlg_i]
        &
        =
        \E[\E[\tauTest_{j,i+1}\hat\tau_{j,i+1}\chi_{\iset S(i+1)}(j)|\SAlg_i^+]|\SAlg_i]
        \\
        &
        =
        \E[\E[\tauTest_{j,i+1}\tau_{j,i+1}|\SAlg_i^+]|\SAlg_i]
        \\
        &
        =
        \E[\E[1|\SAlg_i^+]\tauTest_{j,i+1}\tau_{j,i+1}|\SAlg_i]
        =
        \E[\iset\pi_{j,i+1}\tauTest_{j,i+1}\tau_{j,i+1}|\SAlg_i].
        \end{split}
    \end{equation}
    Further expanding with \eqref{eq:step-update-formulas-y} shifted from $i$ to $i+1$, and d using $\eta_{i+1} \in \Random(\SAlg_i; (0, \infty))$ from \cref{ass:step}\,\ref{item:eta-relationships}, we obtain
    \begin{equation}
         \label{eq:c2-stoch-argument2}
        \E[\iset\pi_{j,i+1}\tauTest_{j,i+1}\tau_{j,i+1}|\SAlg_i]
        =
        \E[\eta_{i+1}-\tauTest_{j,i}\tau_{j,i}\chi_{S(i) \setminus \iset S(i)}(j)|\SAlg_i]
        =
        \eta_{i+1}-\tauTest_{j,i}\tau_{j,i}\chi_{S(i) \setminus \iset S(i)}(j).
    \end{equation}
    This verifies the equality in \eqref{eq:step-conds-primal-v2}.
    Thus \eqref{eq:step-conds-primal} holds.

    Similarly we can verify \eqref{eq:step-conds-dual} and $\sigma_{\ell,i+1} \ge 0$. Thus \eqref{eq:lambda-cond0} holds, as does the non-negativity claim on the dual step lengths.

    Using \eqref{eq:nesting-s-v-cond} and \eqref{eq:lambda-cond0}, we now observe that $\lambda_{\ell,j,i}$ satisfies%
    \begin{subequations}%
    \label{eq:cond-ai-expectation-block0-simplified}%
    \begin{align}%
        \label{eq:cond-ai-expectation-block0-simplified-lambda}
        \lambda_{\ell,j,i}&=0,
            &&
            (j \not\in S(i) \text{ or } \ell \not\in V(i+1)),
            \quad\text{and}
        \\
        \label{eq:cond-ai-expectation-block0-simplified-exp}
        \E[\lambda_{\ell,j,i+1}|\SAlg_i]
            &=
            \tilde\lambda_{\ell,j,i+1},
        &&
        (j=1,\ldots,m;\, \ell \in \Neigh(j)),
    \end{align}%
    \end{subequations}%
    for $\tilde\lambda_{\ell,j,i+1} \defeq
            \sigmaTest_{\ell,i+1}\hat\sigma_{\ell,i+1}
            + \lambda_{\ell,j,i} - \tauTest_{j,i}\hat\tau_{j,i}$.
    Using \eqref{eq:step-compatibility}, which follows from \cref{lemma:step-compatibility}, \eqref{eq:cond-ai-expectation0-all} expands as%
    \begin{subequations}%
    \label{eq:cond-ai-expectation-block0}%
    \begin{align}%
        \label{eq:cond-ai-expectation-block0-x}
            \E[\lambda_{\ell,j,i+1}|\SAlg_i]
                &=
                \tilde\lambda_{\ell,j,i+1}
            && 
            (j \in S(i),\, \ell \in \Neigh(j)),
        \\
        \label{eq:cond-ai-expectation-block0-y}
            \E[\lambda_{\ell,j,i+1}|\SAlg_i]
                &=
                \tilde\lambda_{\ell,j,i+1}
            &&
            (\ell \in V(i+1),\, j \in \inv\Neigh(\ell)),
            \quad\text{and}
        \\
        \label{eq:cond-ai-expectation-block0-xprev}
            \E[\lambda_{\ell,j,i+1}|\SAlg_{i-1}]
                &=
                \E[
                    \tilde\lambda_{\ell,j,i+1}
                |\SAlg_{i-1}],
            &&
            (j=1,\ldots,m;\, \ell \in \Neigh(j)).
    \end{align}%
    \end{subequations}%
    Clearly \eqref{eq:cond-ai-expectation-block0-simplified-exp} implies \eqref{eq:cond-ai-expectation-block0-x} and \eqref{eq:cond-ai-expectation-block0-y}.
    Moreover, applying $\E[\freevar|\SAlg_{i-1}]$ to \eqref{eq:cond-ai-expectation-block0-simplified-exp} and using standard properties of nested conditional expectations we obtain \eqref{eq:cond-ai-expectation-block0-xprev}.
    We have therefore verified \eqref{eq:cond-ai-expectation0-all}.
\end{proof}

\begin{corollary}
    \label{cor:step-fullex}
    Suppose \cref{ass:step}\,\ref{item:eta-relationships} \& \ref{item:step-update-formulas} hold.
    Then
    \begin{align*}
        \E[\tauTest_{j,i+1}\hat\tau_{j,i+1}|\SAlg_i] & = \eta_{i+1} + \eta_{\tau,i+1}^\perp-\eta_{\tau,i}^\perp,
        \quad\text{and} \\
        \E[\sigmaTest_{\ell,i+2}\hat\sigma_{\ell,i+2}|\SAlg_i] & = \eta_{i+1} + \eta_{\sigma,i+1}^\perp - \eta_{\sigma,i}^\perp.
    \end{align*}
\end{corollary}
\begin{proof}
    Arguing analogously to \eqref{eq:c2-stoch-argument1} and  \eqref{eq:c2-stoch-argument2} with the cases $j \in S(i) \setminus \iset S(i)$ and $\ell \in V(i+1) \setminus \iset V(i+1)$ of \cref{ass:step}\,\ref{item:step-update-formulas}, we deduce
    \begin{align*}
        \E[\tauTest_{j,i+1}\hat\tau_{j,i+1}(1-\chi_{\iset S(i+1)}(j))|\SAlg_i] & = \eta_{\tau,i+1}^\perp,
        \quad\text{and} \\
        \E[\sigmaTest_{\ell,i+2}\hat\sigma_{\ell,i+2}(1-\chi_{\iset V(i+2)}(\ell))|\SAlg_{i}] & = \eta_{\sigma,i+1}^\perp.
    \end{align*}%
    Combined with \eqref{eq:step-conds} (in the proof of \cref{lemma:c2}) these imply the claim.
\end{proof}

For the next lemma we recall the coordinate notation $x_j$ and $y_\ell$ from \eqref{eq:coord-notation}.

\begin{lemma}
    \label{lemma:cond-delta-expectation-x-primeprime}
    Suppose \cref{ass:structural} (main structural condition) and \cref{ass:step} (step length parameter restrictions) hold.
    Then
    \[
        \E[\norm{\nextu-\realoptu}_{\MetricDiff_{i+1}(\tilde\Gamma)}^2-\norm{\nextu-\thisu}_{\Test_{i+1}\Precond_{i+1}}^2] \le \sum_{j=1}^m \boundfx_{j,i+2}(\tilde\gamma_j) + \sum_{\ell=1}^n  \boundfy_{\ell,i+2},
    \]
    where $\boundfx_{j,i+2}(\tilde\gamma_j)$ and $\boundfy_{\ell,i+2}$ are given in \eqref{eq:bx-j-choice} and \eqref{eq:by-ell-choice}, respectively.
\end{lemma}
\begin{proof}
  Since $\nextu \in \Random(\SAlg_i; X \times Y)$ and $\thisu \in \Random(\SAlg_{i-1}; X \times Y)$, standard nesting properties of conditional expectations show
  \begin{equation}
      \label{eq:nextu-realoptu-delta-expanded2}
      \begin{split}
      \E[\norm{\nextu-\realoptu}_{\MetricDiff_{i+1}(\tilde\Gamma)}^2]
      &
      =
      \E\bigl[
          \norm{\nextu-\thisu}_{\E[\MetricDiff_{i+1}(\tilde\Gamma)|\SAlg_i]}^2
          + \norm{\thisu-\realoptu}_{\E[\MetricDiff_{i+1}(\tilde\Gamma)|\SAlg_{i-1}]}^2
          \\
          & \phantom{ === }
          +
          2\iprod{\nextu-\thisu}{\thisu-\realoptu}_{\E[\MetricDiff_{i+1}(\tilde\Gamma)|\SAlg_i]}
      \bigr].
      \end{split}
  \end{equation}
  By \cref{lemma:c2}, \eqref{eq:cond-ai-expectation0-all} holds. Using \eqref{eq:deltax-expansion}, we therefore expand \eqref{eq:nextu-realoptu-delta-expanded2} into
  \begin{equation*}
    \begin{split}
      \E[\norm{\nextu-\realoptu}_{\MetricDiff_{i+1}(\tilde\Gamma)}^2]
      &=
      \E\bigl[
      \norm{\nextx-\thisx}_{\E[\TauTest_{i+1} - \TauTest_i(I + 2\Tau_i\tilde\Gamma)|\SAlg_i]}^2
      +
      \norm{\thisx-\realoptx}_{\E[\TauTest_{i+1} - \TauTest_i(I + 2\Tau_i\tilde\Gamma)|\SAlg_{i-1}]}^2    
      \\
      \MoveEqLeft[-2]
      +
      \norm{\nexty-\thisy}_{\E[\SigmaTest_{i+2} - \SigmaTest_{i+1}|\SAlg_i]}^2
      +
      \norm{\thisy-\realopty}_{\E[\SigmaTest_{i+2} - \SigmaTest_{i+1}|\SAlg_{i-1}]}^2
      \\
      \MoveEqLeft[2]
      +
      2
      \iprod{\nextx-\thisx}{\thisx-\realoptx}_{\E[\TauTest_{i+1} - \TauTest_i(I + 2\Tau_i\tilde\Gamma)|\SAlg_i]}
      +
      2
      \iprod{\nexty-\thisy}{\thisy-\realopty}_{\E[\SigmaTest_{i+2} - \SigmaTest_{i+1}|\SAlg_i]}
      \bigr].
    \end{split}
  \end{equation*}
  By \cref{ass:step}\,\ref{item:c2-y-primeprime},  $\E[\SigmaTest_{i+2} - \SigmaTest_{i+1}|\SAlg_i] \ge 0$. Standard properties of conditional expectations guarantee $\E[\E[\SigmaTest_{i+2} - \SigmaTest_{i+1}|\SAlg_i]|\SAlg_{i-1}]=\E[\SigmaTest_{i+2} - \SigmaTest_{i+1}|\SAlg_{i-1}]$.
  By \cref{lemma:cond-u-to-x-prime-block}, moreover
  \[
    -\norm{\nextu-\thisu}_{\Test_{i+1}\Precond_{i+1}}^2 \le -\delta \norm{\nextx-\thisx}_{\TauTest_i}.
  \]
  Use of Cauchy's inequality for arbitrary factors $\alpha_i,\beta_i>0$ therefore yields
  \begin{equation*}
    \begin{split}
      \E\bigl[\norm{\nextu&-\realoptu}_{\MetricDiff_{i+1}(\tilde\Gamma)}^2-\norm{\nextu-\thisu}_{\Test_{i+1}\Precond_{i+1}}^2\bigr]
      \\
      &
      =
      \E\bigl[
      \norm{\nextx-\thisx}_{\E[\TauTest_{i+1} - \TauTest_i(I + 2\Tau_i\tilde\Gamma)|\SAlg_i]+\alpha_i\abs{\E[\TauTest_{i+1} - \TauTest_i(I + 2\Tau_i\tilde\Gamma)|\SAlg_i]} - \delta\TauTest_i}^2
      \\
      \MoveEqLeft[-2]
      +
      \norm{\thisx-\realoptx}_{\E[\TauTest_{i+1} - \TauTest_i(I + 2\Tau_i\tilde\Gamma)|\SAlg_{i-1}]+\inv\alpha_i\abs{\E[\TauTest_{i+1} - \TauTest_i(I + 2\Tau_i\tilde\Gamma)|\SAlg_i]}}^2
      \\
      \MoveEqLeft[-2]
      +(1+\beta_i)\norm{\nexty-\thisy}_{\E[\SigmaTest_{i+2} - \SigmaTest_{i+1}|\SAlg_i]}^2
      +(1+\inv\beta_i)\norm{\thisy-\realopty}_{\E[\SigmaTest_{i+2} - \SigmaTest_{i+1}|\SAlg_{i-1}]}^2
      \bigr].
      \end{split}
  \end{equation*}
  Here we write $\abs{\sum_{j=1}^m c_j P_j} \defeq \sum_{j=1}^m \abs{c_j} P_j$.
  Therefore, choosing $\beta_i=1/2$, splitting the estimates into blocks, and using  \cref{ass:step}\, \ref{item:c2-x-primeprime} \& \ref{item:c2-y-primeprime}, we obtain the claim.
\end{proof}

It is relatively easy to satisfy \cref{ass:step}\,\ref{item:c2-y-primeprime} and to bound $\boundfy_{\ell,i+2}$. To estimate $\boundfx_{j,i+2}(\tilde\gamma_j)$, we need to derive more involved update rules. We next construct one example.

\begin{example}[Random primal test updates]
    \label{example:tautest-rule-rnd}
    If \eqref{eq:c2-x-primeprime-bound} holds, take $\rho_j \ge 0$, otherwise take $\rho_j=0$ ($j=1,\ldots,m$). Set
    \begin{equation}
        \label{eq:tautest-rule}
        \tauTest_{j,i+1} \defeq
        \tauTest_{j,i}(1+2\tilde \gamma_j\hat\tau_{j,i})+2\rho_j \inv\pi_{j,i}\chi_{S(i)}(j),
        \quad (j=1,\ldots,m; i \in \N).
    \end{equation}
    Then it is not difficult to show that $\tauTest_{j,i+1} \in \Random(\SAlg_i; (0, \infty))$ and
    $
        \boundfx_{j,i+2}(\tilde\gamma_j)
        =
        18 C_x \rho_j.
    $

    If we set $\rho_j=0$ and have just a single deterministically updated block, \eqref{eq:tautest-rule} is the standard rule \eqref{eq:cpaccel} with $\tauTest_i = \tau_i^{-2}$. The role of $\rho_j>0$ is to ensure some (slower) acceleration on non-strongly-convex blocks with $\tilde\gamma_j=0$. This is necessary for convergence rate estimates.
\end{example}

The difficulty with \eqref{eq:tautest-rule} is that the coupling parameter $\eta_{i+1}$ will depend on the random realisations of $S(i)$ through $\tauTest_{j,i+1}$. This will require communication in a parallel implementation of the algorithm. We therefore desire to update $\tauTest_{j,i+1}$ deterministically.
We delay the introduction of an appropriate update rule to \cref{sec:block-continued}.

\subsection{Satisfaction of the coupling conditions}
\label{sec:gap-satisfaction}

We still need to satisfy either of the coupling conditions \eqref{eq:eta-conds} to obtain gap estimates.

\begin{lemma}
    \label{lemma:verif-eta}
    Suppose \cref{ass:structural} (main structural condition), \cref{ass:step}\,\ref{item:step-update-formulas}, \ref{item:eta-relationships} \& \ref{item:eta-doubly} hold.
	Then the coupling condition \eqref{eq:cond-eta} holds.
\end{lemma}

\begin{proof}
	The condition \eqref{eq:cond-eta} holds if $\E[\tauTest_{j,i+1}\hat\tau_{j,i+1}]=\bar\eta_{i+1}=\E[\sigmaTest_{\ell,i+2}\hat\sigma_{\ell,i+2}]$ for some $\bar\eta_{i+1}$ for all $j=1,\ldots,m$ and $\ell=1,\ldots,n$.
	Taking $\bar\eta_{i+1} \defeq \E[\eta_{i+1} + \eta_{\tau,i+1}^\perp-\eta_{\tau,i}^\perp]$, the claim follows from \cref{cor:step-fullex} and \cref{ass:step}\,\ref{item:eta-doubly}.
\end{proof}


The alternative coupling condition \eqref{eq:cond-etatwo} requires $\E[\tauTest_{j,i+1}\hat\tau_{j,i+1}]=\bar\eta_{i+1}=\E[\sigmaTest_{\ell,i+1}\hat\sigma_{\ell,i+1}]$ for some $\bar\eta_{i+1}$. By \cref{cor:step-fullex}, this holds when
\begin{equation}
    \label{eq:cond-etatwo-verif}
	\E[\eta_{i+1} + \eta_{\tau,i+1}^\perp-\eta_{\tau,i}^\perp]
	= \bar\eta_i =
	\E[\eta_{i} + \eta_{\sigma,i}^\perp - \eta_{\sigma,i-1}^\perp].
\end{equation}
It is not clear how to satisfy this simultaneously with \cref{ass:step}\,\ref{item:eta-doubly}, so we use \ref{item:eta-singly}.

\begin{lemma}
    \label{lemma:verif-etatwo}
    Suppose \cref{ass:structural} (main structural condition), \cref{ass:step}\,\ref{item:step-update-formulas} \& \ref{item:eta-singly} hold.
    Then \ref{item:eta-relationships} holds if and only if $\iset V(i+1)=\emptyset$ and $V(i+1)=\{1,\ldots,n\}$.
    When this is the case, the coupling condition \eqref{eq:cond-etatwo} holds, necessarily $S(i)=\iset S(i)$, and
    \begin{subequations}
  	\label{eq:step-update-formulas-fulldual}
    \begin{align}
    	\label{eq:step-update-formulas-x-fulldual}
    	\tau_{j,i} & =
    		\eta_{i}/(\tauTest_{j,i}\iset \pi_{j,i})
    		&& (j \in S(i)),
 		\\
 		\label{eq:step-update-formulas-y-fulldual}
 		\sigma_{j,i+1} & =
            \eta_{i+1}/\sigmaTest_{j,i+1}
            &&
            (j \in \Neigh(S(i))).
    \end{align}%
    \end{subequations}%
\end{lemma}

\begin{proof}
    \Cref{ass:step}\,\ref{item:eta-singly}, i.e., $\eta_{\tau,i}^\perp=0$ and $\eta_{\sigma,i}^\perp=\eta_{i+1}$  reduces \ref{item:eta-relationships} to $\min_\ell (\nu_{\ell,i+1}-\iset \nu_{\ell,i+1}) \ge 1$.
    This holds if and only hold if $\nu_{\ell,i+1} \equiv 1$ and $\iset \nu_{\ell,i+1} \equiv 0$ for all $\ell=1,\ldots,m$.
    This holds, as claimed, if and only if $\iset V(i+1)=\emptyset$, $V(i+1)=\{1,\ldots,n\}$.
    Clearly in this case \eqref{eq:nesting} holds if and only if $S(i)=\iset S(i)$ this implies \eqref{eq:nesting}.
    With \cref{ass:step}\,\ref{item:eta-relationships} verified, \cref{cor:step-fullex} now rewrites \eqref{eq:cond-etatwo}  as \eqref{eq:cond-etatwo-verif}.
    This is clearly verified by $\eta_{\tau,i}^\perp=0$ and $\eta_{\sigma,i}^\perp=\eta_{i+1}$.
    Finally, \eqref{eq:step-update-formulas-fulldual} is a specialisation of \cref{ass:step}\,\ref{item:step-update-formulas} to the choices of \ref{item:eta-singly}.
\end{proof}

\begin{remark}
	We had to impose full dual updates to satisfy \eqref{eq:cond-etatwo}. This is akin to most existing primal--dual coordinate descent methods \cite{suzuki2013stochastic,bianchi2015stochastic,fercoq2015coordinate}. The algorithms in \cite{pesquet2014class,peng2016coordinate,yu2015doubly} are more closely related to our method, however, only \cite{yu2015doubly} provides convergence rates for single-block sampling schemes under full strong convexity of both $G$ and $F^*$.
\end{remark}

\subsection{Sampling patterns}
\label{sec:sampling}

There are not many possible fully deterministic sampling patterns allowed by \eqref{eq:nesting} with \cref{ass:step}. Indeed, \eqref{eq:step-conds-primal} reads in the deterministic setting
\[
    \tauTest_{j,i+1}\tau_{j,i+1}\chi_{\iset S(i+1)}(j)
    +  \tauTest_{j,i}\hat\tau_{j,i}\chi_{S(i) \setminus \iset S(i)}(j)) = \eta_{i+1}.
\]
Since $\eta_{i+1}>0$, $j \not \in S(i) \setminus \iset S(i)$ implies $j \in \iset S(i+1)$, which implies $j \not \in \subset S(i+i) \setminus \iset S(i+1)$.
Therefore, once in the independently updated set, the block $j$ will always stay there.
Due to \eqref{eq:nesting-s-v-cond}, if $\iset V(i+1) \ne \emptyset$ consistently, for most $K$, the set $S(i)$ will grow. Therefore, after a small number of iterations $N$, either $j \in \iset S(i)$ for $i \ge N$, or $j \in S(i)=\{1,\ldots,n\}$. Similar considerations hold for the dual blocks. Therefore, the way each block is updated in deterministic methods is, after a small number of iterations, fixed. There does not, therefore, appear to be significant improvements possible over consistently taking $S(i)=\iset S(i)=\{1,\ldots,m\}$, $\iset V(i+1)=\emptyset$ and $V(i+1)=\{1,\ldots,m\}$ (or the converse dual-first order).

Regarding stochastic algorithms, we start with a few options for sampling $S(i)$ in \cref{alg:alg-blockcp-fulldual} with iteration-independent probabilities $\pi_{j,i} \equiv \pi_j$.

\begin{example}[Independent probabilities]
    \label{ex:sampling-indep}
    If all the blocks $\{1,\ldots,m\}$ are chosen independently, we have $\P(\{j, k\} \subset S(i))=\pi_j\pi_k$ for $j \ne k$, where $\pi_j \in (0, 1]$.
\end{example}

\begin{example}[Fixed number of random blocks]
    \label{ex:sampling-fixed-number}
    If we have a fixed number $M$ of processors, we may want to choose a subset $S(i) \subset \{1,\ldots,m\}$ such that $\#S(i)=M$.
\end{example}

The next example gives a simple way to satisfy \eqref{eq:nesting-iset-s-v-cond} for \cref{alg:alg-blockcp}.

\begin{example}[Alternating x-y and y-x steps]
    \label{ex:sampling-alternating}
    Let us randomly alternate between $\iset S(i)=\emptyset$ and $\iset V(i+1)=\emptyset$.
    That is, with some probability $\mathbb{p}_x$, we choose to take an $x$-$y$ step that omits lines \ref{step:xorth} and \ref{step:yiset} in \cref{alg:alg-blockcp}, and with probability $1-\mathbb{p}_x$, an $y$-$x$ step that omits the lines \ref{step:xiset} and \ref{step:yorth}. If $\tilde\pi_j=\P[j \in \iset S|\iset S\ne\emptyset]$, and $\tilde\nu_\ell=\P[\ell \in \iset V|\iset V \ne \emptyset]$ denote the probabilities of the rule used to sample $\iset S = \iset S(i)$ and $\iset V = \iset V(i+1)$ when non-empty, then \eqref{eq:nesting} gives
    \begin{align}
        \notag
        \iset\pi_j &= \mathbb{p}_x \tilde\pi_j,
        &
        \pi_j & = \mathbb{p}_x \tilde\pi_j + (1-\mathbb{p}_x)\P[j \in \inv\Neigh(\iset V)|\iset V \ne \emptyset],
        \\
        \notag
        \iset\nu_\ell &= (1-\mathbb{p}_x) \tilde\nu_\ell,
        &
        \nu_\ell & = (1-\mathbb{p}_x) \tilde\nu_j + \mathbb{p}_x \P[\ell \in \Neigh(\iset S)|\iset S \ne \emptyset].
    \end{align}
    To compute $\pi_j$ and $\nu_\ell$ we thus need to know $\Neigh$ and the exact sampling pattern.
\end{example}

\begin{remark}
    Based on \cref{ex:sampling-alternating}, we can derive an algorithm where the only randomness comes from alternating between full $x$-$y$ and full $y$-$x$ steps.
\end{remark}

\section{Rates of convergence}
\label{sec:block-continued}

We now need to satisfy \cref{ass:step}. This involves choosing update rules for $\eta_{i+1}$, $\eta_{\tau,i+1}^\perp$, $\eta_{\sigma,i+1}^\perp$, $\tauTest_{j,i+1}$ and $\sigmaTest_{\ell,i+1}$.
At the same time, to obtain good convergence rates, we need to make $d_{j,N}^x(\tilde\gamma_j)$ and $d_{\ell,N}^y$ small in \eqref{eq:convergence-result-stochastic-block}.
We do these tasks here, including stating two final versions of our algorithm \cref{alg:alg-blockcp} (doubly stochastic) and \cref{alg:alg-blockcp-fulldual} (full dual updates).
Specifically, in \cref{sec:primal-det} we introduce and study a deterministic alternative to the example random update rule for $\tauTest_{j,i+1}$ in \cref{example:tautest-rule-rnd}. 
The analysis of the new rule is easier, and it allows the computation of $\eta_i$, which will also be deterministic, without communication in parallel implementations of our algorithms. 
Afterwards, in \cref{sec:etaperp} we look at possible choices for the parameters $\eta_{\tau,i}^\perp$ and $\eta_{\sigma,i}^\perp$, \emph{which are only needed in stochastic variants of \cref{alg:alg-blockcp}}.
In \crefrange{sec:worstcase}{sec:sigmatest-fullstrong-rnd} we then give various useful choices of $\eta_i$ and $\sigmaTest_{\ell,i}$ that yield concrete convergence rates.

We assume for simplicity that the sampling pattern is independent of iteration,
\begin{align}
	\label{eq:constant-probabilities}
	\iset\pi_{j,i} \equiv \iset\pi_j>0,\quad
	\iset\nu_{\ell,i} \equiv \iset\nu_\ell, \quad
    \pi_{j,i} \equiv \pi_j,
    \quad\text{and}\quad
    \nu_{\ell,i} \equiv \nu_\ell.
\end{align}

\subsection{Deterministic primal test updates}
\label{sec:primal-det}

The next lemma gives a deterministic alternative to \cref{example:tautest-rule-rnd}.
We recall that $\gamma_j \ge 0$ is the factor of (strong) convexity of $G_j$.

\begin{lemma}
	\label{lemma:tautest-rule-det}
	Suppose \cref{ass:step}\,\ref{item:eta-relationships} \& \ref{item:step-update-formulas}, and \eqref{eq:constant-probabilities} hold, and that $i \mapsto \eta_{\tau,i}^\perp$ is non-decreasing.
    Suppose, moreover, that either \eqref{eq:c2-x-primeprime-bound} holds or  $\sup_{j=1,\ldots,m} \rho_j = 0$.
	Also take $\tau_{j,0},\tauTest_{j,0}>0$ and $\bar\gamma_j \ge 0$ such that $\rho_j+\bar\gamma_j>0$, and set
    \begin{equation}
        \label{eq:tautest-rule-det}
        \tauTest_{j,i+1}
        \defeq
        \tauTest_{j,i} + 2(\bar\gamma_j\eta_i+\rho_j),
        \quad (j=1,\ldots,m; i \in \N).
    \end{equation}
    Then for some $c_j >0$ and all $N \ge 1$ holds
	\begin{subequations}
	\label{eq:tautest-conditions-for-eta}
	\begin{align}
		\label{eq:tautest-sigma-algebra}
		\tauTest_{j,N+1} & \in \Random(\SAlg_{N-1}; (0, \infty)),
		\\
		\label{eq:tautest-expectation}
		\E[\tauTest_{j,N}] & = \tauTest_{j,0} + 2\rho_jN+2\bar\gamma_j \sum_{i=0}^{N-1} \E[\eta_i],
		\quad\text{and}
		\\
		\label{eq:invtautest-estim}
		\E[\inv\tauTest_{j,N}] & \le c_j \inv N, \quad (N \ge 1).
    \intertext{If $\tilde\gamma_j \in [\bar\gamma_j,\gamma_j]$, ($j=1,\ldots,m$), satisfy}
        \label{eq:det-rule-delta-cond}
        2\tilde\gamma_j\bar\gamma_j \eta_i &
        \le (\tilde\gamma_j-\bar\gamma_j)\delta\tauTest_{j,i}, \quad (j \in S(i),\, i \in \N),
        \\
	\intertext{then \cref{ass:step}\,\ref{item:c2-x-primeprime} holds, and}
	    \label{eq:djnx-estim}
	    d_{j,N}^x(\tilde\gamma_j) & = 18 \rho_j C_x N.
    \intertext{Finally, if $\eta_i \ge b_j \min_j \tauTest_{j,i}^p$ for some $p,b_j>0$, then for some $\tilde c_j > 0$ holds}
        \label{eq:invtautest-estim2}
        1 & \ge \bar\gamma_j \tilde c_j N^{p+1}\E[\inv\tauTest_{j,N}],
        \quad (N \ge 4).
    \end{align}
	\end{subequations}
\end{lemma}

\begin{proof}
    Since \cref{ass:step}\,\ref{item:eta-relationships} guarantees $\eta_{i} \in \Random(\SAlg_{i-1}; (0, \infty))$, we deduce \eqref{eq:tautest-sigma-algebra} from \eqref{eq:tautest-rule-det}.
    In fact, $\tauTest_{j,i+1}$ is deterministic as long as $\eta_i$ is deterministic. 

    The claim \eqref{eq:tautest-expectation} is immediate from using \eqref{eq:tautest-rule-det} to compute
	\begin{equation}
		\label{eq:tautest-rule-det-expanded}
	    \tauTest_{j,N}
	    =\tauTest_{j,N-1} + 2(\bar\gamma_j\eta_{N-1}+\rho_j)
	    =\tauTest_{j,0} + 2\rho_jN+2\bar\gamma_j \sum_{i=0}^{N-1} \eta_i.
	\end{equation}

	Since $i \mapsto \eta_i$ is non-decreasing, clearly	$\tauTest_{j,N} \ge 2N\tilde\rho_j$ for $\tilde\rho_j \defeq \rho_j+\bar\gamma_j \eta_0 > 0$.
	Then $\inv\tauTest_{j,N} \le \frac{1}{2\tilde\rho_j N}$.
	Taking the expectation proves \eqref{eq:invtautest-estim}.

	Clearly \eqref{eq:invtautest-estim2} holds if $\bar\gamma_j=0$, so assume $\bar\gamma_j>0$.	
	Using the assumption $\eta_i \ge b_j \min_j \tauTest_{j,i}^p$ and $\tauTest_{j,i} \ge 2i\tilde\rho_j$ that we just proved in \eqref{eq:tautest-rule-det-expanded}, we estimate
    \[
        \tauTest_{j,N}
        \ge \tauTest_{j,0} + b_j(2\tilde\rho_j)^p \sum_{i=1}^N i^p
        \ge \tauTest_{j,0} + b_j(2\tilde\rho_j)^p \int_{2}^N x^p \d x
        \ge \tauTest_{j,0} + \inv p b_j(2\tilde\rho_j)^p (N^{p+1}-2).
    \]
	Thus $\inv\tauTest_{j,N} \le 1/(\bar\gamma_j\tilde c_j N^{1+p})$ for some $\tilde c_j>0$.
	Taking the expectation proves \eqref{eq:invtautest-estim2}.

	It remains to prove \eqref{eq:djnx-estim} and \cref{ass:step}\,\ref{item:c2-x-primeprime}.
	Abbreviating $\gamma_{j,i} \defeq \bar\gamma_j + \rho_j\inv\eta_i$, we write $\tauTest_{j,i+1} = \tauTest_{j,i}+2\gamma_{j,i}\eta_i$.  
    Since $i \mapsto \eta_{\tau,i}^\perp$ is non-decreasing, \cref{cor:step-fullex} gives
    \begin{equation}
        \label{eq:det-rule-eta-estim}
        \E[\tauTest_{j,i}\hat\tau_{j,i}|\SAlg_{i-1}]=\eta_i+\eta_{\tau,i}^\perp-\eta_{i-1,\tau}^\perp \ge \eta_i.
    \end{equation}
    Expanding the defining equation \eqref{eq:h-j} of $h_{j,i+2}(\tilde\gamma_j)$, with the help of \eqref{eq:det-rule-eta-estim} we estimate
    \[
        \begin{split}
        h_{j,i+2}(\tilde\gamma_j)
        &
        =
        2\E[\gamma_{j,i}\eta_i-\tilde\gamma_j\tauTest_{j,i}\hat\tau_{j,i}|\SAlg_{i-1}]+2\inv\alpha_i\abs{\E[\gamma_{j,i}\eta_i-\tilde\gamma_j\tauTest_{j,i}\hat\tau_{j,i}|\SAlg_i]}
        \\
        &
        \le
        2(\gamma_{j,i}-\tilde\gamma_j)\eta_i+2\inv\alpha_i\abs{\gamma_{j,i}\eta_i-\tilde\gamma_j\tauTest_{j,i}\hat\tau_{j,i}}
        \\
        &
        \le
        2(1+\inv\alpha_i)\rho_j
        +
        2(\bar\gamma_j-\tilde\gamma_j)\eta_i
        +
        2\inv\alpha_i\abs{\bar\gamma_j\eta_i-\tilde\gamma_j\tauTest_{j,i}\hat\tau_{j,i}}.
        \end{split}
    \]
    Since \eqref{eq:det-rule-delta-cond} implies $\bar\gamma_j \le \tilde\gamma_j$, if also
    \begin{gather}
        \label{eq:det-rule-alphacond1}
        \inv\alpha_i\abs{\bar\gamma_j\eta_i-\tilde\gamma_j\tauTest_{j,i}\hat\tau_{j,i}}
        \le (\tilde\gamma_j-\bar\gamma_j)\eta_i,
    \shortintertext{then}
        \label{eq:det-h-j-est}
        \E[\max\{0,h_{j,i+2}(\tilde\gamma_j)\}]
        \le
        2(1+\inv\alpha_i)\rho_j.
    \end{gather}
    We claim \eqref{eq:det-rule-alphacond1} this to hold for
    \begin{equation}
    	\label{eq:det-rule-alpha1}
        \alpha_i
        \defeq
        \begin{cases}
        	\min_j \bar\gamma_j/(\tilde\gamma_j-\bar\gamma_j),
        	& \bar\gamma_j\eta_i>\tilde\gamma_j\tauTest_{j,i}\hat\tau_{j,i}, 
        	\\
        	\min_j
        	(\tilde\gamma_j\inv{\iset\pi_j}+\bar\gamma_j)/(\tilde\gamma_j-\bar\gamma_j),
        	&
        	\bar\gamma_j\eta_i \le \tilde\gamma_j\tauTest_{j,i}\hat\tau_{j,i}.
        \end{cases}
    \end{equation}
    The case $\bar\gamma_j\eta_i>\tilde\gamma_j\tauTest_{j,i}\hat\tau_{j,i}$ is clear. Otherwise, to justify the case $\bar\gamma_j\eta_i \le \tilde\gamma_j\tauTest_{j,i}\hat\tau_{j,i}$, we observe that \eqref{eq:det-rule-alphacond1} can in this case be rewritten as $\tilde\gamma_j\tauTest_{j,i}\hat\tau_{j,i} \le (\alpha_i(\tilde\gamma_j-\bar\gamma_j)-\bar\gamma_j)\eta_i$.
	With the choice of $\alpha_i$ in \eqref{eq:det-rule-alpha1}, we see this to hold if $\tauTest_{j,i}\hat\tau_{j,i} \le \inv{\iset\pi_j}\eta_i$. We consider the cases $j \in \iset S(i)$ and $j \in S(i) \setminus \iset S(i)$ separately.
    In the case $j \in \iset S(i)$, this is inequality immediate from \eqref{eq:step-update-formulas-x} in \cref{ass:step}\,\ref{item:step-update-formulas} and \cref{lemma:c2}.
    If $j \in S(i) \setminus \iset S(i)$, \eqref{eq:step-update-formulas-x} and \cref{ass:step}\,\ref{item:eta-relationships} give
    \[
        \tauTest_{j,i}\hat\tau_{j,i}(\pi_j-\iset \pi_j) \le \eta_{\tau,i}^\perp
        \le \min_{j'}(\pi_{j'}-\iset \pi_{j'})\eta_i
        \le (\pi_{j}-\iset \pi_{j})\eta_i
        \le (\pi_{j}-\iset \pi_{j})\inv{\iset\pi_j}\eta_i.
    \]
    In the last step we have used that $\iset\pi_j \in (0, 1]$ by \eqref{eq:constant-probabilities}.
    This finishes verifying \eqref{eq:det-h-j-est}.

    Next, we expand \eqref{eq:q-j}, obtaining
    \[
        \begin{split}
        q_{j,i+2}(\tilde\gamma_j)
        &
        =
        \bigl(
            2\E[\gamma_{j,i}\eta_i-\tilde\gamma_j\tauTest_{j,i}\hat\tau_{j,i}|\SAlg_i]
            +2\alpha_i\abs{\E[\gamma_{j,i}\eta_i-\tilde\gamma_j\tauTest_{j,i}\hat\tau_{j,i}|\SAlg_i]}-\
            \delta\tauTest_{j,i}
        \bigr)\chi_{S(i)}(j),
        \\
        &
        =
        \bigl(
            2(\gamma_{j,i}\eta_i-\tilde\gamma_j\tauTest_{j,i}\hat\tau_{j,i})
            +2\alpha_i\abs{\gamma_{j,i}\eta_i-\tilde\gamma_j\tauTest_{j,i}\hat\tau_{j,i}}
            -\delta\tauTest_{j,i}
        \bigr)\chi_{S(i)}(j),
        \\
        &
        \le
        \bigl(
            2(1+\alpha_i)\rho_j
            +
            2(\bar\gamma_{j}\eta_i-\tilde\gamma_j\tauTest_{j,i}\hat\tau_{j,i})
            +2\alpha_i\abs{\bar\gamma_{j}\eta_i-\tilde\gamma_j\tauTest_{j,i}\hat\tau_{j,i}}
            -\delta\tauTest_{j,i}
        \bigr)\chi_{S(i)}(j).
        \end{split}
    \]
    Since $\eta_i$ and $\tauTest_{j,i}\tau_{j,i}$ are increasing, if also
    \begin{gather}
        \label{eq:det-q-j-alphacond1}
        2(\bar\gamma_{j}\eta_i-\tilde\gamma_j\tauTest_{j,i}\hat\tau_{j,i})
        +2\alpha_i\abs{\bar\gamma_{j}\eta_i-\tilde\gamma_j\tauTest_{j,i}\hat\tau_{j,i}}
        \le
        \delta\tauTest_{j,i}
        \quad (j \in S(i)),
    \shortintertext{then}
        \label{eq:det-q-j-est}
        \E[q_{j,i+2}(\tilde\gamma_j)] \le 2(1+\alpha_i)\rho_j.
    \end{gather}
    Inserting $\alpha_i$ from \eqref{eq:det-rule-alpha1}, we see \eqref{eq:det-q-j-alphacond1} to follow from \eqref{eq:det-rule-delta-cond}.
	Finally, \eqref{eq:det-h-j-est} and \eqref{eq:det-q-j-est} show that \eqref{eq:c2-x-primeprime-zero} holds with $\rho_j=0$.
    Thus \cref{ass:step}\,\ref{item:c2-x-primeprime} holds.
    From \cref{prop:general-block-conditions} now
    \[
        \boundfx_{j,i+2}(\tilde\gamma_j)
        =
        8(1+\alpha_i)\rho_j C_x
        +2(1+\inv\alpha_i)\rho_j C_x.
    \]
    Clearly $\alpha_i$ defined in \eqref{eq:det-rule-alpha1} is bounded above and below, so we obtain \eqref{eq:djnx-estim}.
\end{proof}

\subsection{The parameters $\eta_{\tau,i}^\perp$ and $\eta_{\sigma,i}^\perp$}
\label{sec:etaperp}

We now want to satisfy \cref{ass:step}\,\ref{item:eta-doubly} for doubly-stochastic methods. As it turns out, the parameters $\eta_{\tau,i}^\perp$ and $\eta_{\sigma,i}^\perp$, do not have any effect on convergence rates. Here are a few options.

\begin{lemma}
	\label{lemma:etaperp-examples}
	Assume \eqref{eq:constant-probabilities} and that $i \mapsto \eta_i$ is non-decreasing with $\eta_i \in \Random(\SAlg_{i-1}; (0,\infty))$.
    Then \cref{ass:step}\,\ref{item:eta-relationships} \& \ref{item:eta-doubly} hold and both $i \mapsto \eta_{\tau,i}^\perp$ and $i \mapsto \eta_{\sigma,i}^\perp$ are non-decreasing if either:
	\begin{enumerate}[label=(\roman*)]
		\item \label{item:etaperp-constant}
			(Constant rule) We take $\eta_{\tau,i}^\perp \equiv \eta_\tau^\perp$ and $\eta_{\sigma,i}^\perp \equiv \eta_{\sigma}^\perp$ for constant $\eta_\sigma^\perp, \eta_\tau^\perp > 0$ satisfying 
		    \begin{equation*}
		        \eta_0 \cdot \min_j(\pi_{j}-\iset \pi_{j}) \ge \eta_\tau^\perp,
		        \quad\text{and}\quad
		        \eta_0 \cdot \min_\ell (\nu_{\ell}-\iset \nu_{\ell}) \ge \eta_\sigma^\perp.
		    \end{equation*}

		\item \label{item:etaperp-proportional}
			(Proportional rule) For some $\alpha \in (0, 1)$ we take $\eta_{\tau,i}^\perp \defeq \eta_{\sigma,i}^\perp \defeq \alpha\eta_i$ satisfying
		    \begin{equation*}
		        \min_j~(\pi_{j}-\iset \pi_{j}) \ge \alpha,
		        \quad\text{and}\quad
		        \min_\ell~(\nu_{\ell}-\iset \nu_{\ell}) \ge \alpha.
		    \end{equation*}
	\end{enumerate}
\end{lemma}

\begin{proof}
	Clearly both rules satisfy \cref{ass:step}\,\ref{item:eta-relationships} \& \ref{item:eta-doubly}.
	That $i \mapsto \eta_{\tau,i}^\perp$ and $i \mapsto \eta_{\sigma,i}^\perp$ are non-decreasing and belong to $\Random(\SAlg_{i-1}; [0, \infty))$ is obvious.
\end{proof}


\subsection{Worst-case rules for $\eta_i$}
\label{sec:worstcase}

To verify \cref{ass:step}\,\ref{item:sigmatest-choice-1} we take deterministic worst-case bounds $\wcase_j, \wcase_{j,\ell} \ge 0$ such that
\begin{equation}
    \label{eq:wcase}
    \wcase_j \defeq \max_\ell \wcase_{\ell,j}
    \quad\text{and}\quad
    \wcase_{\ell,j} \ge
        \inv{\iset\pi_{j}}\chi_{\iset S(i)}(j)
        +
        \inv{\iset\nu_{\ell}}\chi_{\iset V(i+1)}(\ell)
    \quad (i \in \N).
\end{equation}
Since we assume iteration-independent probabilities \eqref{eq:constant-probabilities}, such bounds exist.

\begin{lemma}
    \label{lemma:eta-choice-basic}
    Suppose \cref{ass:step}\,\ref{item:step-update-formulas} and \eqref{eq:constant-probabilities} hold. With $(\kappa_1,\ldots,\kappa_n) \in \kappafamily$ take%
    \begin{equation}
        \label{eq:eta-choice-basic}
        \eta_i 
        \defeq
        \min_{\ell=1,\ldots,n} \sqrt{\frac{(1-\delta)\sigmaTest_{\ell,i+1}}{\kappa_\ell(\wcase_{\ell,1}^{2}\inv\tauTest_{1,i}, \ldots, \wcase_{\ell,m}^{2}\inv\tauTest_{m,i})}}
        \quad (i \ge 0).
    \end{equation}
    Then \cref{ass:step}\,\ref{item:sigmatest-choice-1} holds.
    Moreover, $\eta_i \in \Random(\SAlg_{i-1}; (0,\infty))$ provided $\sigmaTest_{i+1} \in \Random(\SAlg_{i-1}; (0,\infty))$.
\end{lemma}

\begin{proof}
    Recalling the expression for $\lambda_{\ell,j,i}$ in \eqref{eq:lambda-choice}, \cref{ass:step}\,\ref{item:step-update-formulas}, and \eqref{eq:constant-probabilities} imply $\lambda_{\ell,j,i} \le \eta_i\wcase_{\ell,j}$ for $\ell \in \Neigh(j)$.
    By the monotonicity of $\kappa_\ell$ (assumed in \cref{def:kappa}), \cref{ass:step}\,\ref{item:sigmatest-choice-1} will therefore hold if
    \begin{equation}
        \label{eq:sigmatest-cond-worstcase}
        \sigmaTest_{\ell,i+1}
        \ge
        \frac{\eta_i^2}{1-\delta} \kappa_\ell(\wcase_{\ell,1}^{2}\inv\tauTest_{1,i}, \ldots, \wcase_{\ell,m}^{2}\inv\tauTest_{m,i}).
    \end{equation}
    This is verified by inserting $\eta_i$ from \eqref{eq:eta-choice-basic}.
    Clearly \eqref{eq:eta-choice-basic} also verifies $\eta_i \in \Random(\SAlg_{i-1}; (0,\infty))$ when $\sigmaTest_{i+1} \in \Random(\SAlg_{i-1}; (0,\infty))$.
\end{proof}

The next lemma provides a choice of $\sigmaTest_{i+1} \in \Random(\SAlg_{i-1}; (0,\infty))$ that also satisfies \cref{ass:step}\,\ref{item:c2-y-primeprime}.
The resulting $\eta_i$ we express later in \eqref{eq:eta-choice-second} to collect all rules in one place.

\begin{lemma}
	\label{lemma:eta-lower-estim}
    Let $(\kappa_1,\ldots,\kappa_n) \in \kappafamily$ and take $\eta_i$ according to \eqref{eq:eta-choice-basic} with $\sigmaTest_{\ell,i+1}=\eta_i^{2-1/p} \sigmaTest_{\ell,0}$ for some $\sigmaTest_{\ell,0}>0$ and $p \in (0, 1]$.
    If $\tauTest_{j,i} \in \Random(\SAlg_{i-1}; (0, \infty))$, then $\eta_i,\sigmaTest_{i+1} \in \Random(\SAlg_{i-1}; (0,\infty))$.
	If, moreover, \eqref{eq:invtautest-estim} holds, then $\E[\eta_i] \ge c_\eta^{p} i^{p}$ and $\eta_i \ge b_\eta^{p} \min_j \tauTest_{j,i}^{p}$ for some constants $c­_\eta, b_\eta>0$ independent of $p$.
\end{lemma}

\begin{proof}
    That $\eta_i,\sigmaTest_{i+1} \in \Random(\SAlg_{i-1}; (0,\infty))$  is clear from \eqref{eq:eta-choice-basic} and $\tauTest_{j,i} \in \Random(\SAlg_{i-1}; (0, \infty))$.

	With $\MIN \sigmaTest_0 \defeq \min_{\ell=1,\ldots,n} \sigmaTest_{\ell,0}$,	from \eqref{eq:eta-choice-basic} also
	\[
   	    \eta_i^{1/p} 
	    \ge
	    \frac{(1-\delta)\MIN{\sigmaTest_{0}}}{\max_{\ell=1,\ldots,n} \kappa_\ell(\wcase_{\ell,1}^{2}\inv\tauTest_{1,i}, \ldots, \wcase_{\ell,m}^{2}\inv\tauTest_{m,i})}.
	\]
	Since $\hat \mu_{\ell,j,i}=0$ for $\ell \not \in \Neigh(j)$, using \cref{def:kappa}\,\ref{item:kappa-bounded}
    , we get
	\begin{equation*}
	    \eta_i^{1/p}
	    \ge
	    {\frac{(1-\delta)\MIN{\sigmaTest_{0}}}{\MAX \kappa \sum_{j=1}^n \max_\ell \wcase_{\ell,j}^{2}\inv\tauTest_{j,i}}}
	    \ge
	    {\frac{1}{\sum_{j=1}^n \inv b_j\inv\tauTest_{j,i}}}
	\end{equation*}
	for $b_j \defeq (1-\delta)\MIN{\sigmaTest_{0}}/(\MAX \kappa \wcase_{j}^{2})$. This shows $\eta_i \ge \min_j b_j^{p} \tauTest_{j,i}^{p}$.
	Since $x \mapsto 1/x$ and $x \mapsto x^q$ are convex on $[0, \infty)$ for $q \ge 1$, Jensen's inequality gives
	\begin{equation*}
	    \E[\eta_i]
	    \ge
	    \frac{1}{
	    \E\bigl[(\textstyle\sum_{j=1}^n b_j^{-1}\inv\tauTest_{j,i})^{p}\bigr]}
	    \ge
	    \frac{1}{
	    \bigl(\textstyle\sum_{j=1}^n b_j^{-1}\E[\inv\tauTest_{j,i}]\bigr)^{p}}.
	\end{equation*}
	By an application of \eqref{eq:invtautest-estim} we obtain $\E[\eta_i] \ge c_\eta^{p} i^{p}$ for $c_\eta \defeq 1/\sum_{j=1}^m b_j^{-1}c_j$.
\end{proof}


\subsection{Mixed rates under partial strong convexity}
\label{sec:partconv}

We are finally ready to state our main result and algorithms.
We recall that by \cref{lemma:pp-equiv}, \eqref{eq:pp} is equivalent to \eqref{eq:cpoper-theta} under the structural conditions of \cref{ass:structural}.
Dividing the updates of \eqref{eq:cpoper-theta} into individual block updates, and taking the step length rules from \cref{ass:step}\,\ref{item:step-update-formulas}, we obtain the steps of the doubly-stochastic method \cref{alg:alg-blockcp}. If we perform full dual updates, i.e., force \cref{ass:step}\,\ref{item:eta-singly} and following \cref{lemma:verif-etatwo} take $\iset V(i+1)=\{1,\ldots,n\}$ and $\iset S(i)=S(i)$, we get the simpler steps of \cref{alg:alg-blockcp-fulldual}. Regarding the updates of the remaining parameters that are not specified directly in the algorithm skeletons, we start with:

\begin{theorem}
    \label{thm:convergence-basic}
    Assume the block-separable structure \eqref{eq:g-fstar-separable}, writing $\gamma_j \ge 0$ for the factor of (strong) convexity of $G_j$. Let $\delta \in (0, 1)$ and $(\kappa_1,\ldots,\kappa_n) \in \kappafamily$ (see \cref{def:kappa}).
    In \cref{alg:alg-blockcp} or \cref{alg:alg-blockcp-fulldual}, take
    \begin{enumerate}[label=(\roman*)]
        \item\label{item:convergence-basic-tautest}
        $\tauTest_{j,0} > 0$ freely and $\tauTest_{j,i+1} \defeq \tauTest_{j,i} + 2(\bar\gamma_j\eta_i+\rho_j)$ for some $\rho_j\ge 0$ and $\bar\gamma_j \in [0,\gamma_j]$ with $\rho_j+\bar\gamma_j>0$.
        \item\label{item:convergence-basic-sigmatest}
        $\sigmaTest_{\ell,0}>0$ freely and $\sigmaTest_{\ell,i} \defeq \sigmaTest_{\ell,0}\eta_i^{2-1/p}$ for some fixed $p \in [1/2, 1]$.
        \item\label{item:convergence-basic-eta}
        $\eta_{\tau,i}^\perp,\eta_{\sigma,i}^\perp > 0$ (in \cref{alg:alg-blockcp}) following \cref{lemma:etaperp-examples}, and, with $\wcase_j$ given by \eqref{eq:wcase},
        \begin{equation}
            \label{eq:eta-choice-second}
            \eta_i \defeq \min_{\ell=1,\ldots,n} \left(\frac{(1-\delta)\sigmaTest_{\ell,0}}{\kappa_\ell(\wcase_{\ell,1}^{2}\inv\tauTest_{1,i}, \ldots, \wcase_{\ell,m}^{2}\inv\tauTest_{m,i})}\right)^{p}.
        \end{equation}
    \end{enumerate}
    Let $\realoptu \in \inv H(0)$, i.e., solve \eqref{eq:oc}, and suppose the following hold:
    \begin{enumerate}[label=(\Alph*)]
        \item\label{item:conv:primal} $\sup_{j=1,\ldots,m} \rho_j=0$ or $\sup_{j=1,\ldots,m;\, i \in \N} \norm{\nextx_j-\realoptx_j}^2  \le C_x$ for a constant $C_x>0$.

        \item\label{item:conv:dual} $p = \frac{1}{2}$ or both $\sup_{\ell=1,\ldots,n;\, i \in \N} \norm{\nexty_\ell-\realopty_\ell}^2 \le C_y$ and $\bar\gamma_{j^*}=0$ for some $j^* \in \{1,\ldots,m\}$.

        \item\label{item:conv:gamma} With $\ell^*(j)$ and $\MIN\kappa$ given by  \cref{def:kappa}, for some $\tilde\gamma_j \in [\bar\gamma_j, \gamma_j]$ for all $j=1,\ldots,m$ we have the initialisation bound
    	\begin{equation}
    	    \tilde\gamma_j=\bar\gamma_j=0\quad\text{or}\quad
    	    \frac{2\tilde\gamma_j\bar\gamma_j}{\tilde\gamma_j-\bar\gamma_j}
    	    \left(\frac{1-\delta}{\MIN\kappa \wcase_j}\right)^p
    	    \le \delta \sigmaTest_{\ell^*(j),0}^{-p}\tauTest_{j,0}^{1-p}.
    	\end{equation}	
    \end{enumerate}
    Then
    \begin{equation}
        \label{eq:convergence-result-stochastic-block-v1}
        \sum_{j=1}^m \frac{\delta \tilde c_j \bar\gamma_j}{2}\E\bigl[\norm{x_j^N-\realoptx_j}\bigr]^2
        + g_{p,N}
        \le\textstyle
        \frac{\norm{u^0-\realoptu}_{\Test_0\Precond_0}^2 + 18 C_x (\sum_{j=1}^m \rho_j)N+\sum_{\ell=1}^n \sigmaTest_{\ell,0}\bigl(C_*N^{2p-1} + \delta_*\bigr)}{\displaystyle 2N^{p+1}},
    \end{equation}    
    when $N \ge 4$ and the weighted gap on the ergodic variables,
    \[
    	g_{p,N} \defeq
    	\begin{cases}
    		 c_p \gap(\tilde x_N, \tilde y_N), & \text{\cref{alg:alg-blockcp}},\, \tilde\gamma_j \le \gamma_j/2 \text{ for all } j, \\
    		 c_{*,p} \gap(\tilde x_{*,N}, \tilde y_{*,N}), & \text{\cref{alg:alg-blockcp-fulldual}},\, \tilde\gamma_j \le \gamma_j/2 \text{ for all } j, \\
    		 0, & \text{otherwise}.
    	\end{cases}
    \]
    The constants $C_*,\delta_* \ge 0$ are zero if $p=1/2$ while the constants $c_{p}, c_{*,p} > 0$.
\end{theorem}

\begin{remark}
    \label{rem:p-sigmatest}
	If $p=1/2$, \eqref{eq:convergence-result-stochastic-block-v1} yields a mixed $O(1/N^{3/2})+O(1/N^{1/2})$ convergence rate.
	If $p=1$, we get a mixed $O(1/N^2)+O(1/N)$ convergence rate.
\end{remark}

\begin{remark}
	\Cref{thm:convergence-basic} is valid (with suitable constants) for general primal update rules as long as \eqref{eq:tautest-conditions-for-eta} holds and $i \mapsto \tauTest_{j,i}$ is non-decreasing.
	This is the case for the deterministic rule of \cref{lemma:tautest-rule-det}. For the random rule of \cref{example:tautest-rule-rnd}, the rest of the conditions hold, but we have not been able to verify \eqref{eq:invtautest-estim2}. This has the implication that only the gap estimates hold.
\end{remark}

\begin{Algorithm}[t!]
    \caption{Doubly-stochastic primal--dual method}
    \label{alg:alg-blockcp}
    \begin{subequations}
    \begin{algorithmic}[1]
    \REQUIRE $K \in \linear(X; Y)$, $G \in \convex(X)$, and $F^* \in \convex(Y)$ with the separable structures \eqref{eq:g-fstar-separable}.
    \REQUIRE Rules for $\tauTest_{j,i}$, $\sigmaTest_{\ell,i+1}$, $\eta_{i+1}, \eta_{\tau,i+1}^\perp, \eta_{\sigma,i+1}^\perp > 0$ from \cref{thm:convergence-basic}, \cref{cor:convergence-unaccelerated}, or \ref{cor:convergence-second-det-bounded}.
    \REQUIRE Sampling patterns for $S(i), \iset S(i), V(i+1)$, and $\iset V(i+1)$, ($i \in \N$), subject to the nesting condition \eqref{eq:nesting} (p.~\pageref{eq:nesting}) with iteration-independent probabilities \eqref{eq:constant-probabilities}; see \cref{sec:sampling}.
    \STATE Choose initial iterates $x^0 \in X$ and $y^0 \in Y$.
    \STATE Initialise $\tau_{j,-1},\sigma_{\ell,0} \defeq 0$, ($j=1,\ldots,m$; $\ell=1,\ldots,m$).
    \FORALL{$i \ge 0$ \textbf{until} a stopping criterion is satisfied}
    \STATE Sample $\iset S(i) \subset S(i) \subset \{1,\ldots,m\}$ and $\iset V(i+1) \subset V(i+1) \subset \{1,\ldots,n\}$.
    \STATE For each $j \in \iset S(i)$, compute
    {\belowdisplayskip=0ex\abovedisplayskip=0.7ex
    \begin{align}
        \notag
        \tau_{j,i}& \textstyle \defeq \frac{\eta_{i} - \tauTest_{j,i-1}\tau_{j,i-1}\chi_{S(i-1) \setminus \iset S(i-1)}(j)}
                    {\tauTest_{j,i}\iset \pi_{j,i}}, \quad\text{and}
        \\
        \notag
        \nextx_j & \textstyle
 \defeq \inv{(I+\tau_{j,i} \subdiff G_j)}\left(
            \thisx_j - \tau_{j,i} \sum_{\ell \in \Neigh(j)} K_{\ell, j}^* \thisy_\ell
        \right),
        \quad{where}\quad K_{\ell, j}\defeq Q_\ell K P_j.
    \end{align}
    }
    \STATE For each $\ell \in \iset V(i+1)$, compute
    {\belowdisplayskip=0ex\abovedisplayskip=0.7ex
    \begin{align}
        \notag
        \sigma_{j,i+1} &  \textstyle \defeq \frac{\eta_{i} - \sigmaTest_{j,i}\sigma_{j,i}\chi_{V(i) \setminus \iset V(i)}(j)}
                    {\sigmaTest_{j,i+1}\iset \nu_{\ell,i+1}}, \quad\text{and}
        \\
        \notag
        \nexty_\ell & \textstyle \defeq \inv{(I+\sigma_{\ell,i+1} \subdiff F^*_\ell)}\left(
            \thisy_\ell + \sigma_{\ell,i+1} \sum_{j \in \inv\Neigh(\ell)} K_{\ell, j} \thisx_j
        \right).
    \end{align}
    }
    \STATE\label{step:xiset} For each $j \in \iset S(i)$ and $\ell \in \Neigh(j)$, set
    {\belowdisplayskip=0ex\abovedisplayskip=0.7ex
    \[
        \textstyle
        \nexxt{\tilde w}_{\ell,j}
        \defeq
        \theta_{\ell,j,i+1}(\nextx_j-\thisx_j) + \nextx_j
        \quad\text{with}\quad
        \theta_{\ell,j,i+1}
        \defeq
        \frac{\tau_{j,i}\tauTest_{j,i}}{\sigma_{\ell,i+1}\sigmaTest_{\ell,i+1}}.
    \]
    }
    \STATE\label{step:yiset} For each $\ell \in \iset V(i+1)$ and $j \in \inv\Neigh(\ell)$, set
    {\belowdisplayskip=0ex\abovedisplayskip=0.7ex
    \[
        \textstyle
        \nexxt{\tilde v}_{\ell,j}
        \defeq
        b_{\ell,j,i+1}(\nexty_\ell-\thisy_\ell) + \thisy_\ell
        \quad\text{with}\quad
        b_{\ell,j,i+1}
        \defeq
        \frac{\sigma_{\ell,i+1}\sigmaTest_{\ell,i+1}}{\tau_{j,i}\tauTest_{j,i}}.
    \]
    }
    \STATE\label{step:xorth} For each $j \in S(i) \setminus \iset S(i)$, compute
    {\belowdisplayskip=0ex\abovedisplayskip=0.7ex
    \begin{align}
        \notag
        \tau_{j,i}& \defeq \textstyle \frac{\eta_{\tau,i}^\perp}
                    {\tauTest_{j,i}(\pi_{j,i}-\iset \pi_{j,i})}, \quad\text{and}
        \\
        \notag
        \nextx_j & \defeq \textstyle \inv{(I+\tau_{j,i} \subdiff G_j)}\left(
            \thisx_j - \tau_{j,i} \sum_{\ell \in \Neigh(j)} K_{\ell, j}^* \nexxt{\tilde v}_{\ell,j}
        \right).
    \end{align}
    }
    \STATE\label{step:yorth} For each $\ell \in V(i+1) \setminus \iset V(i+1)$ compute
    {\belowdisplayskip=0ex\abovedisplayskip=0.7ex
    \begin{align}
        \notag
        \sigma_{j,i+1} &  \textstyle \defeq \frac{\eta_{\sigma,i}^\perp}
                    {\sigmaTest_{j,i+1}(\nu_{\ell,i+1}-\iset\nu_{\ell,i+1})}, \quad\text{and}
        \\
        \notag
        \nexty_\ell & \textstyle \defeq \inv{(I+\sigma_{\ell,i+1} \subdiff F^*_\ell)}\left(
            \thisy_\ell + \sigma_{\ell,i+1} \sum_{j \in \inv\Neigh(\ell)} K_{\ell, j} \nexxt{\tilde w}_{\ell,j}
        \right).
    \end{align}
    }
    \ENDFOR
    \end{algorithmic}
    \end{subequations}
\end{Algorithm}

\begin{Algorithm}[t!]
    \caption{Block-stochastic primal--dual method, primal randomisation only}
    \label{alg:alg-blockcp-fulldual}
    \begin{subequations}
    \begin{algorithmic}[1]
    \REQUIRE $K \in \linear(X; Y)$, $G \in \convex(X)$, and $F^* \in \convex(Y)$ with the separable structures \eqref{eq:g-fstar-separable}.
    \REQUIRE Rules for $\tauTest_{j,i}, \sigmaTest_{\ell,i+1}, \eta_{i+1} \in \Random(\SAlg_i; (0, \infty))$ from \cref{thm:convergence-basic}, \cref{cor:convergence-unaccelerated}, or \ref{cor:convergence-second-det-bounded}.
    \REQUIRE Iteration-independent \eqref{eq:constant-probabilities} sampling pattern for the set $S(i)$, ($i \in \N$); see \cref{sec:sampling}.
    \STATE Choose initial iterates $x^0 \in X$ and $y^0 \in Y$.
    \FORALL{$i \ge 0$ \textbf{until} a stopping criterion is satisfied}
    \STATE Sample $S(i) \subset \{1,\ldots,m\}$.
    \STATE For each $j \not\in S(i)$, set $\nextx_j \defeq \thisx_j$.
    \STATE For each $j \in S(i)$, 
         with $\tau_{j,i} \defeq \eta_i\inv\pi_{j,i}\inv\tauTest_{j,i}$, compute
    {\belowdisplayskip=0ex\abovedisplayskip=1ex
    \[
        \textstyle
        \nextx_j \defeq \inv{(I+\tau_{j,i} \subdiff G_j)}\left(
            \thisx_j - \tau_{j,i} \sum_{\ell \in \Neigh(j)} K_{\ell, j}^* \thisy_\ell
        \right),
        \quad{where}\quad K_{\ell, j}\defeq Q_\ell K P_j.
    \]
    }
    \STATE For each $j \in S(i)$ set
    {\belowdisplayskip=0ex
    \[
        \overnextx_j
        \defeq
        \theta_{j,i+1}(\nextx_j-\thisx_j) + \nextx_j
        \quad\text{with}\quad
        \theta_{j,i+1}
        \defeq
        \frac{\eta_i}{\pi_{j,i}\eta_{i+1}}.
    \]
    }
    \STATE For each $\ell \in \{1,\ldots,n\}$ using $\sigma_{\ell,i+1} \defeq \eta_{i+1}\inv\sigmaTest_{\ell,i+1}$, compute
    {\belowdisplayskip=0ex
    \[
        \textstyle
        \nexty_\ell \defeq \inv{(I+\sigma_{\ell,i+1} \subdiff F^*_\ell)}\left(
            \thisy_\ell + \sigma_{\ell,i+1} \sum_{j \in \inv\Neigh(\ell)} K_{\ell, j} \overnextx_j
        \right).
    \]
    }
    \ENDFOR
    \end{algorithmic}
    \end{subequations}
\end{Algorithm}

\begin{proof}
    We use \cref{prop:general-block-conditions}, so need to verify \cref{ass:structural,ass:step}. First of all, \eqref{eq:sigma-algebra-cond} follows from the updates rules for the testing and step length parameters, that only depend on previous realisations of $S(i)$ and $V(i+1)$. The rest of the conditions of \cref{ass:structural} are clear from \cref{lemma:pp-equiv}, the derivation of \cref{alg:alg-blockcp,alg:alg-blockcp-fulldual} from \eqref{eq:cpoper-theta}, and the requisite nesting condition \eqref{eq:nesting} within the algorithms themselves.

    Regarding the requirements \ref{item:sigmatest-choice-1}--\ref{item:c2-y-primeprime} of \cref{ass:step}, we proceed as follows:
    \begin{itemize}
        \item[\ref{item:sigmatest-choice-1}] The choice $\sigmaTest_{\ell,i+1} \defeq \eta_i^{2-1/p}\sigmaTest_{\ell,0}$ in \ref{item:convergence-basic-sigmatest} shows that \eqref{eq:eta-choice-second} is equivalent to the formula \eqref{eq:eta-choice-basic} for $\eta_i$.
        Thus \cref{lemma:eta-choice-basic} verifies \ref{item:sigmatest-choice-1}.

        \item[\ref{item:eta-relationships}]
        It is clear that $i \mapsto \tauTest_{j,i}$ and $i \mapsto \sigmaTest_{\ell,i}$ are non-decreasing. Therefore \eqref{eq:eta-choice-basic} shows that $i \mapsto \eta_i$ is non-decreasing. Moreover, \cref{lemma:eta-lower-estim} verifies that $\eta_i \in \Random(\SAlg_{i-1}; (0,\infty))$.

        \Cref{alg:alg-blockcp-fulldual} by construction satisfies \cref{ass:step}\,\ref{item:eta-singly} and has both $\iset V(i+1)=\emptyset$ and $V(i+1)=\{1,\ldots,n\}$. It therefore suffices to refer to \cref{lemma:verif-etatwo}.

        \Cref{alg:alg-blockcp}, by its own statement, satisfies \eqref{eq:constant-probabilities}.
        Therefore, \cref{lemma:etaperp-examples} shows \cref{ass:step}\,\ref{item:eta-relationships} \& \ref{item:eta-doubly}, and that also $i \mapsto \eta_{\tau,i}^\perp$ is non-decreasing.

        \item[\ref{item:eta-options}] Proved together with \ref{item:eta-relationships} above.

        \item[\ref{item:step-update-formulas}]
        These choices are encoded into \cref{alg:alg-blockcp}.
        For \cref{alg:alg-blockcp-fulldual} we recall \cref{lemma:verif-etatwo}.

        \item[\ref{item:c2-x-primeprime}]
        We use \cref{lemma:tautest-rule-det}.
        We have already showed \cref{ass:step}\,\ref{item:eta-relationships} \& \ref{item:step-update-formulas}. Moreover, the algorithms satisfy the iteration-independent probability assumption \eqref{eq:constant-probabilities}.
        By \ref{item:conv:primal}, either $\sup_j \rho_j = 0$ or \eqref{eq:c2-x-primeprime-bound} holds. We still need to satisfy \eqref{eq:det-rule-delta-cond}.
        Using \cref{def:kappa}\,\ref{item:kappa-nondegeneracy} in \eqref{eq:eta-choice-second}, we estimate
        \begin{equation}
            \label{eq:eta-upper-p-estim}
            \eta_i
            \le
            \biggl(
            \frac{(1-\delta)\sigmaTest_{\ell^*(j),0}}{\MIN\kappa \wcase_j} \tauTest_{j,i}
            \biggr)^p.
        \end{equation}
        By \ref{item:conv:gamma}, therefore, either $\tilde\gamma_j=\bar\gamma_j=0$, or $2\tilde\gamma_j\bar\gamma_j\eta_i \le \delta(\tilde\gamma_j-\bar\gamma_j)\tauTest_{j,0}^{1-p}\tauTest_{j,i}^p$.  By \ref{item:convergence-basic-tautest}, $i \mapsto \tauTest_{j,i}$, so this gives \eqref{eq:det-rule-delta-cond}.
        \Cref{lemma:tautest-rule-det} now shows \cref{ass:step}\,\ref{item:c2-x-primeprime}.

        \item[\ref{item:c2-y-primeprime}]
        If $p=1/2$, by \cref{rem:p-sigmatest}, $\sigmaTest_{\ell,i} \equiv \sigmaTest_{\ell,0}$. Therefore \eqref{eq:c2-y-primeprime-zero}  holds.
        If $p \ne 1/2$, the same remark and \ref{item:conv:primal} guarantee \eqref{eq:c2-y-primeprime-bound}.
    \end{itemize}

    With \cref{ass:step,ass:structural} now verified, \cref{prop:general-block-conditions} provides the estimate
    \begin{equation}
        \label{eq:convergence-result-stochastic-block-repeat}
        \sum_{k=1}^m \frac{\delta}{2\E[\inv\tauTest_{k,N}]} \cdot \E\left[\norm{x_k^N-\realoptx_k}\right]^2
        + \tilde g_N
        \le
        \frac{1}{2}\norm{u^0-\realoptu}_{\Test_0 \Precond_0}^2
        +
        \sum_{j=1}^m
        \frac{1}{2}d_{j,N}^x(\tilde\gamma_j)
        +
        \sum_{\ell=1}^n \frac{1}{2}d_{\ell,N}^y,
    \end{equation}
    where $\tilde g_N$, $d_{j,N}^x(\tilde\gamma_j)$ and $d_{\ell,N}^y$ are given in \eqref{eq:convergence-result-stochastic-block-parts}
	To obtain convergence rates, we still need to further analyse this estimate, mainly $\zeta_N$ and $\zeta_{*,N}$ within $\tilde g_N$.

    We start with $\zeta_N$ and $\zeta_{*,N}$.
    By \cref{lemma:etaperp-examples} for \cref{alg:alg-blockcp} and directly by \cref{ass:step}\,\ref{item:eta-singly} for \cref{alg:alg-blockcp-fulldual} , $i \mapsto \eta_{\tau,i}^\perp$ is non-decreasing (as is $i \mapsto \eta_{\sigma,i}^\perp$).
    We recall the coupling variable $\bar\eta_i$ from \eqref{eq:eta-conds}.
    Observe that  \eqref{eq:invtautest-estim} holds as we have verified the conditions of \cref{lemma:tautest-rule-det} above.
    By \cref{cor:step-fullex,lemma:eta-lower-estim}, therefore, in both cases, \eqref{eq:cond-eta} and \eqref{eq:cond-etatwo}, for some constant $c_\eta>0$,
    \[
        \bar\eta_i = \E[\eta_{i} + \eta_{\tau,i}^\perp - \eta_{\tau,i-1}^\perp] \ge \E[\eta_i] \ge c_\eta^{p} i^{p}.
    \]
	Thus we estimate $\zeta_N$ from \eqref{eq:zeta} as
	\begin{equation}
	    \label{eq:partconv-zeta-rate}
	    \begin{split}
	    \zeta_N &
	    =
	    \sum_{i=0}^{N-1} \bar\eta_i \ge \sum_{i=0}^{N-1} \E[\eta_i]
	    \ge
	    c_\eta^{p} \sum_{i=0}^{N-1} i^{p}
	    \ge
	    c_\eta^{p} \int_0^{N-2} x^{p} \d x
	    \\
	    &
	    \ge
	    \frac{c_\eta^p}{p+1}(N-2)^{p+1}
	    \ge
	    \frac{c_\eta^p}{2^{p+1}(p+1)}N^{p+1}
		=:
	    c_p N^{p+1}	   
	    \quad (N \ge 4).
	    \end{split}
	\end{equation}
	Similarly, for some $c_{*,p}>0$, the quantity $\zeta_{*,N}$ defined in \eqref{eq:zeta} satisfies
	\begin{equation}
	    \label{eq:partconv-zetatwo-rate}
	    \zeta_{*,N} \ge \sum_{i=1}^{N-1} \E[\eta_i]
	    \ge \frac{c_\eta^p}{p+1}((N-2)^{p+1}-1)
	    \ge
	    c_{*,p} N^{p+1}
	    \quad (N \ge 4).
	\end{equation}

   	If $p=1/2$, \ref{item:convergence-basic-sigmatest} clearly implies $d^y_{\ell,N}=\E[\sigmaTest_{\ell,N}-\sigmaTest_{\ell,0}] \equiv 0$. Therefore, we can take $C_*,\delta_*=0$.
	Otherwise, since $0 \le 2-1/p \le 1$, the map $t \mapsto t^{2-1/p}$ is concave.
    Therefore, using \eqref{eq:convergence-result-stochastic-block-parts}, \ref{item:convergence-basic-sigmatest}, and Jensen's inequality, we deduce
	\begin{equation*}
	    d_{y,\ell}^N
        =
        \sum_{i=0}^{N-1} 9C_y\E[\sigmaTest_{\ell,i+2}-\sigmaTest_{\ell,i+1}]
	    =
	    9C_y\sigmaTest_{\ell,0}(\E[\eta_{N+1}^{2-1/p}]-\E[\eta_{1}^{2-1/p}])
        \le
        9C_y\sigmaTest_{\ell,0}\E[\eta_{N+1}]^{2-1/p}.
	\end{equation*}
    The condition \ref{item:conv:dual} provides $j^* \in \{1,\ldots,m\}$ with $\gamma_{j^*}=0$, so that a referral to \eqref{eq:tautest-expectation} shows
    $
        \E[\tauTest_{j^*,N}]
        = \tauTest_{j^*,0} + 2N\rho_{j^*}
    $.
    By \eqref{eq:eta-upper-p-estim} for some $C_*,\delta_* > 0$ then
    \begin{equation}
        \label{eq:sigmatest-increasing-dy}
         d_{y,\ell}^N
         \le 9C_y\sigmaTest_{\ell,0}
            \biggl(
            \frac{(1-\delta)\sigmaTest_{\ell^*(j^*),0}}{\MIN\kappa \wcase_{j^*}} \E[\tauTest_{j^*,i}]
            \biggr)^{2p-1}
         \le \sigmaTest_{\ell,0}(C_*N^{2p-1}+\delta_*).
    \end{equation}

    Finally, to estimate $d_{j,N}^x(\tilde\gamma_j)$, \cref{lemma:eta-lower-estim} shows $\eta_i \ge b_\eta^p \min_j \tauTest_{j,i}^p$, ($j=1,\ldots,m$). Thus \eqref{eq:invtautest-estim2} and \eqref{eq:djnx-estim} in \cref{lemma:tautest-rule-det} give $1/\E[\inv\tauTest_{j,N}] \ge \bar\gamma_j \tilde c_j N^{p+1}$ for $N \ge 4$, and $d_{j,N}^x(\tilde\gamma_j) = 18 \rho_j C_x N$.
   	Now \eqref{eq:convergence-result-stochastic-block-v1} is immediate by applying these estimates and \eqref{eq:partconv-zeta-rate}--\eqref{eq:sigmatest-increasing-dy} to \eqref{eq:convergence-result-stochastic-block-repeat}.
\end{proof}

\subsection{Unaccelerated algorithm}
\label{sec:unaccelerated}

If $\rho_j=0$ and $\bar\gamma_j=\tilde\gamma_j=0$ for all $j=1,\ldots,m$, then $\tauTest_{j,i} \equiv \tauTest_{j,0}$. Consequently \cref{lemma:eta-choice-basic} gives $\eta_i \equiv \eta_0$. Recalling $\zeta_N$ from \eqref{eq:zeta}, we see that $\zeta_N=N\eta_0$. 
Likewise $\zeta_{*,N}$ from \eqref{eq:zeta} satisfies $\zeta_{*,N}=(N-1)\eta_0$.
Clearly also $d^y_{\ell,N}=0$ and $d^x_{j,N}(\tilde\gamma_j)=0$.
Inserting this information into \eqref{eq:convergence-result-stochastic-block-repeat}, we immediately obtain:

\begin{corollary}
    \label{cor:convergence-unaccelerated}
    Assume the block-separable structure \eqref{eq:g-fstar-separable}.
    Let $\delta \in (0, 1)$ and $(\kappa_1,\ldots,\kappa_n) \in \kappafamily$.
    In \cref{alg:alg-blockcp} or \ref{alg:alg-blockcp-fulldual}, take
    \begin{enumerate}[label=(\roman*)]
        \item $\tauTest_{j,i} \equiv \tauTest_{j,0}$ for some fixed $\tauTest_{j,0}>0$.
        \item $\sigmaTest_{\ell,i} \equiv \sigmaTest_{\ell,0}$ for some fixed $\sigmaTest_{\ell,0}>0$.
        \item $\eta_i \equiv \eta_0$ given by \eqref{eq:eta-choice-basic} and (in \cref{alg:alg-blockcp}) $\eta_{\tau,i}^\perp,\eta_{\sigma,i}^\perp > 0$ following \cref{lemma:etaperp-examples}.
    \end{enumerate}
    Then
    \begin{enumerate}[label=(\Roman*)]
    	\item The iterates of \cref{alg:alg-blockcp} satisfy $\gap(\tilde x_N, \tilde y_N) \le C_0\inv\eta_0/(2N)$, ($N \ge 1$).
    	\item The iterates of \cref{alg:alg-blockcp-fulldual} satisfy $\gap(\tilde x_{*,N}, \tilde y_{*,N}) \le C_0/[2\eta_0(N-1)]$, ($N \ge 2$).
    \end{enumerate}
\end{corollary}


\subsection{Full primal strong convexity}
\label{sec:sigmatest-fullstrong-rnd}

If $G$ is fully strongly convex, we can naturally derive an $O(1/N^2)$ algorithm.

\begin{corollary}
    \label{cor:convergence-second-det-bounded}
    Assume the block-separable structure \eqref{eq:g-fstar-separable}, assuming each $G_j$, ($j=1,\ldots,m$), strongly convex with the corresponding factor $\gamma_j > 0$.
    Let $\delta \in (0, 1)$ and $(\kappa_1,\ldots,\kappa_n) \in \kappafamily$.
    In \cref{alg:alg-blockcp} or \cref{alg:alg-blockcp-fulldual}, take
    \begin{enumerate}[label=(\roman*)]
        \item $\tauTest_{j,0} > 0$ freely and $\tauTest_{j,i+1} \defeq \tauTest_{j,i}(1+2\bar\gamma_j\tau_{j,i})$ for some fixed $\bar\gamma_j \in (0, \gamma_j)$.
        \item $\sigmaTest_{\ell,0}>0$ freely and $\sigmaTest_{\ell,i} \defeq \sigmaTest_{\ell,0}$.
        \item $\eta_i$ according to \eqref{eq:eta-choice-basic}, and (in \cref{alg:alg-blockcp}) $\eta_{\tau,i}^\perp,\eta_{\sigma,i}^\perp > 0$ following \cref{lemma:etaperp-examples}.
    \end{enumerate}
    Suppose the initialisation bound \cref{thm:convergence-basic}\,\ref{item:conv:gamma} holds.
    Then
    \begin{gather*}
        \sum_{j=1}^m \frac{\delta\tilde c_j\bar\gamma_j}{2}\E\bigl[\norm{x_j^N-\realoptx_j}\bigr]^2
        + \tilde g_{1,N}
        \le
        \frac{\norm{u^0-\realoptu}_{\Test_0\Precond_0}^2}{2N^{2}}
        \quad (N \ge 4)
    \shortintertext{for}
        \tilde g_{1,N} \defeq
        \begin{cases}
             q_1 \gap(\tilde x_N, \tilde y_N), & \text{\cref{alg:alg-blockcp}},\, \tilde\gamma_j \le \gamma_j/2 \text{ for all } j, \\
             q_{*,1} \gap(\tilde x_{*,N}, \tilde y_{*,N}), & \text{\cref{alg:alg-blockcp-fulldual}},\, \tilde\gamma_j \le \gamma_j/2 \text{ for all } j, \\
             0, & \text{otherwise}.
        \end{cases}
    \end{gather*}
    The constants $\tilde c_j > 0$ are provided by \cref{lemma:tautest-rule-det} while $q_1, q_{*,1} > 0$.
\end{corollary}

\begin{proof}
    We adapt the argumentation of \cref{thm:convergence-basic} for the case $p=1/2$. Indeed, with this choice, our present assumptions satisfy the conditions of that theorem hold:
    \begin{itemize}[nosep]
        \item[\ref{item:convergence-basic-tautest}] with $\rho_j=0$ becomes the present one. Since we take $\bar\gamma_j>0$, $\rho_j+\bar\gamma_j>0$ as required.
        \item[\ref{item:convergence-basic-sigmatest}] reduces to the present one with $p=1/2$.
        \item[\ref{item:convergence-basic-eta}] becomes the present one since \eqref{eq:eta-choice-second} with $p=1/2$ equals \eqref{eq:eta-choice-basic}.
        \item[\ref{item:conv:primal}] trivially holds since $\rho_j=0$ for all $j=1,\ldots,m$.
        \item[\ref{item:conv:dual}] trivially holds when $p=1/2$.
        \item[\ref{item:conv:gamma}] we have assumed.
    \end{itemize}
    Since $C_*,\delta_*=0$ when $p=1/2$, the estimate \eqref{eq:convergence-result-stochastic-block-v1} therefore holds with the right hand side $C_0/(2N^{1+1/2})$. We need to improve this to $C_0/(2N^2)$ by improving testing variable estimates.

	Indeed, the update rule \eqref{eq:tautest-rule-det} now gives
	\begin{equation*}
	    \tauTest_{j,N}
	    \ge \MIN\tauTest_0 + \MIN\gamma \sum_{i=0}^{N-1} \eta_i
	    \ge \MIN\tauTest_0 + \MIN\gamma \sum_{i=0}^{N-1} \eta_i
	    \quad\text{with}\quad
	    \MIN\tauTest_0 \defeq \min_j \tauTest_{j,0} > 0.
	\end{equation*}
	\Cref{lemma:eta-lower-estim} shows $\eta_i^2 \ge \MIN b \min_j \tauTest_{j,i}$ for some $\MIN b$. Therefore
	$\eta_N^2
	    \ge
	    \MIN b \MIN\tauTest_0 + \MIN b \MIN\gamma\sum_{i=0}^{N-1} \eta_i$.	
	Otherwise written this says $\eta_N^2 \ge \tilde\eta_N^2$, where
	\[
	    \tilde\eta_N^2= 
	    \MIN b \MIN\tauTest_0 + \MIN b \MIN\gamma\sum_{i=0}^{N-1} \tilde\eta_i
	    =
	    \tilde\eta_{N-1}^2 + c^2\MIN\gamma\tilde\eta_{N-1} = \tilde\eta_{N-1}^2+\MIN b \MIN\gamma\inv{\tilde\eta_{N-1}}.
	\]
	This implies $\eta_i \ge \tilde\eta_i \ge c_\eta' i$ for some $c_\eta'>0$; cf.~the estimates for \eqref{eq:cpaccel} in \cite{chambolle2010first,tuomov-cpaccel}.
    Working through the final estimation stage of the proof of \cref{thm:convergence-basic} with $p=1/2$, we can now use use in \eqref{eq:partconv-zeta-rate} and \eqref{eq:partconv-zetatwo-rate} the estimate  $\eta_i \ge c_\eta' i$ that would otherwise correspond to $p=1$.
    In our final result, we write the constants $c_p$ and $c_{*,p}$ from the proof as $q_1, q_{*,1} > 0$.
\end{proof}

\begin{remark}[Linear rates]
    If both $G$ and $F^*$ are strongly convex, it is possible to derive linear rates.
    We refer to \cite{tuomov-proxtest} for the single-block deterministic case.
\end{remark}

\begin{remark}[Variance estimates]
    Variance can be estimated following \cite[Remark 3.4]{tuomov-proxtest}.
\end{remark}

\section{Numerical experience}
\label{sec:numerical}

\begin{figure}
    \centering
    \begin{subfigure}{0.245\textwidth}
        \includegraphics[width=\textwidth]{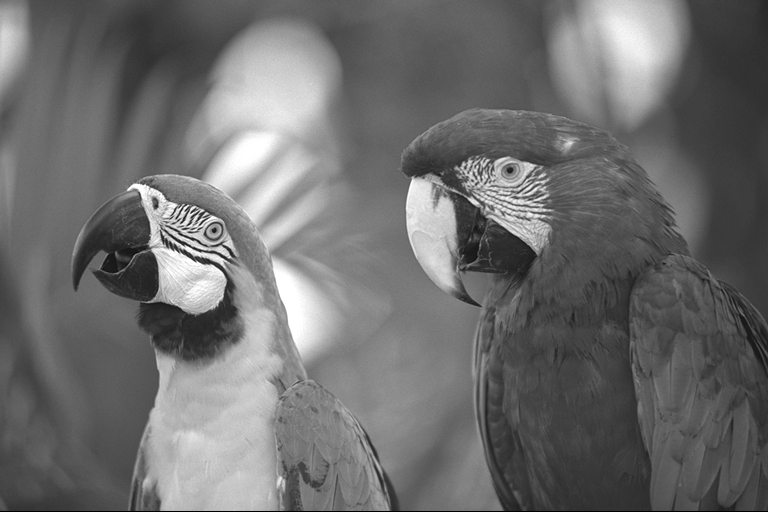}%
        \caption{True image}
    \end{subfigure}\,%
    \begin{subfigure}{0.245\textwidth}
        \includegraphics[width=\textwidth]{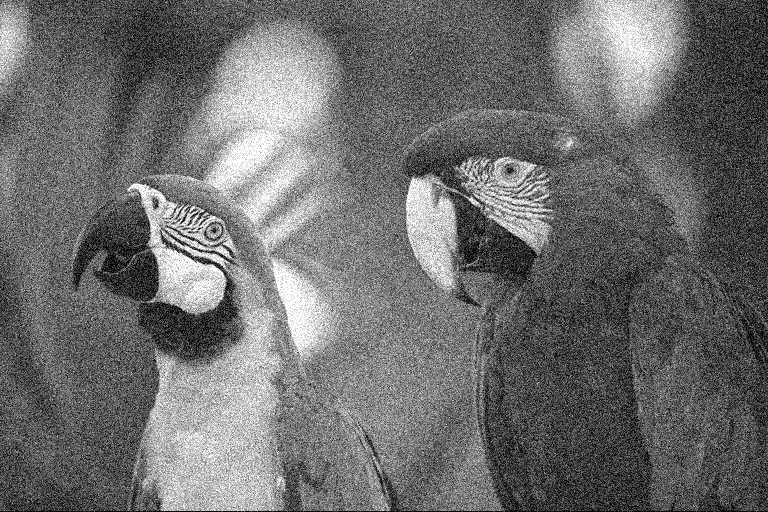}%
        \caption{Noisy image}
        \label{fig:noisy}
    \end{subfigure}\,%
    \begin{subfigure}{0.245\textwidth}
        \includegraphics[width=\textwidth]{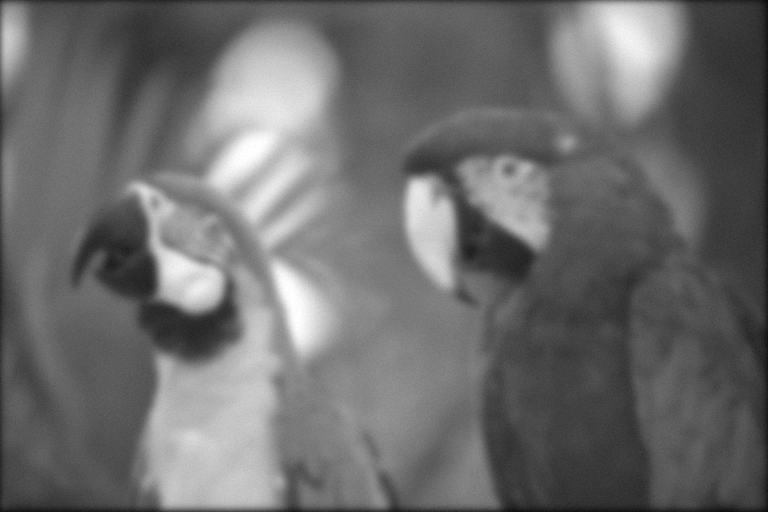}%
        \caption{Blurry image}
        \label{fig:blurry}
    \end{subfigure}\,%
    \begin{subfigure}{0.245\textwidth}
        \includegraphics[width=\textwidth]{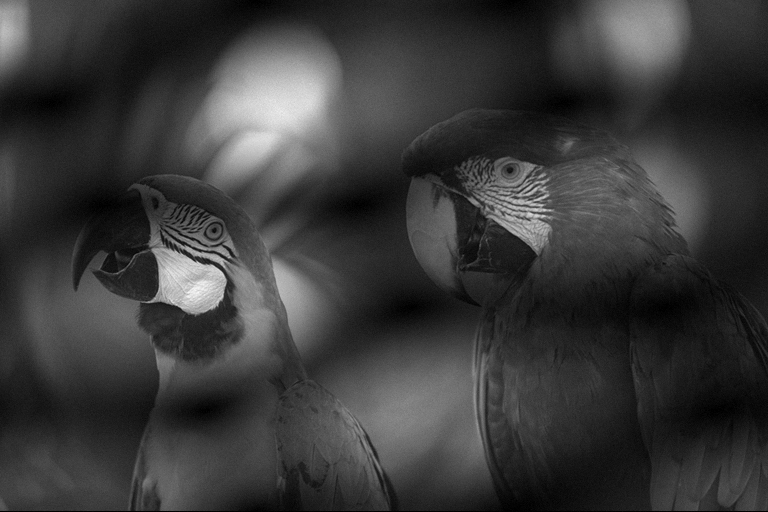}%
        \caption{Dimmed image}
        \label{fig:dimmed}
    \end{subfigure}
    \caption{Sample images for denoising, deblurring, and undimming experiments.
    } 
    \label{fig:image}
\end{figure}

We now apply several variants of the proposed algorithms to image processing problems.
We consider discretisations, as our methods are formulated in Hilbert spaces, but the space of functions of bounded variation---where image processing problems are typically formulated---is only a Banach space.
Our specific example problems will be $\TGV^2$ denoising, $\TV$ deblurring, and $\TV$ undimming.

We present the corrupt and ground-truth images in \cref{fig:image}, with values in the range $[0, 255]$. 
We use the images both at the original resolution of $n_1 \times n_2 = 768 \times 512$, and scaled down to $192 \times 128$ pixels. To the noisy high-resolution test image in \cref{fig:noisy}, we have added Gaussian noise with standard deviation $29.6$ ($12$dB). In the downscaled image, this becomes $6.15$ ($25.7$dB). 
The image in \cref{fig:blurry} we have distorted with Gaussian blur of standard deviation $4$. To avoid inverse crimes, we have added Gaussian noise of standard deviation $2.5$. The dimmed image in \cref{fig:dimmed}, we have distorted by multiplying the image with a sinusoidal mask $\gamma$; see \cref{fig:blurry}.
Again, we have added the small amount of noise.

Besides the unaccelerated PDHGM---our examples lack strong convexity for acceleration of basic methods---we evaluate our algorithms against the relaxed PDHGM of \cite{chambolle2014ergodic,he2012convergence}. In our precursor work \cite{tuomov-cpaccel}, we have also evaluated these two algorithms against the mixed-rate method of \cite{chen2015optimal}, and the adaptive PDHGM of \cite{goldstein2015adaptive}.
To keep our tables and figures easily legible, we also do not include the algorithms of \cite{tuomov-cpaccel} in our evaluations. It is worth noting that even in the two-block case, the algorithms presented in this paper will not reduce to those of that paper: our rules for $\sigma_{\ell, i}$ are very different from the rules for the single $\sigma_i$ therein. 

We define abbreviations of our algorithm variants in \cref{tab:alg-variants}. We do not report the results or apply all variants to all example problems, as this would not be informative. 
We demonstrate the performance of the stochastic variants on $\TGV^2$ denoising only. This merely serves as an example, as our problems are not large enough to benefit from being split on a computer cluster, where the benefits of stochastic approaches would be apparent.

\begin{table}
\caption{Algorithm variant name construction}
\label{tab:alg-variants}
\centering
\small
\begin{tabular}{l|l|l|l|l}
	Letter: & 1st & 2nd & 3rd & 4th
	\\ \hline
	&
	Randomisation &
	$\phi$ rule &
	$\eta$ and $\psi$ rules &
	$\kappa$ choice
	\\ \hline
	\hfill A- &
    D: Deterministic &
    R: Random, Lem.~\ref{example:tautest-rule-rnd} &
    B: Bounded: $p=\frac{1}{2}$ &
    O: Balanc., Ex.~\ref{ex:kappa-refined}
    \\
    &
    P: Primal only &
    D: Determ., Lem.~\ref{lemma:tautest-rule-det} &
    I: Increasing: $p=1$ &
    M: Max., Ex.~\ref{ex:kappa-max}
    \\
    &
    B: Primal \& Dual &
    C: Constant
    &
\end{tabular}
\end{table}

To rely on \cref{thm:convergence-basic} for convergence, we still need to satisfy \eqref{eq:c2-y-primeprime-bound} and \eqref{eq:c2-x-primeprime-bound}, or take $\rho_j=0$.
The bound $C_y$ in \cref{ass:step}\,\ref{item:c2-y-primeprime} is easily calculated, as in all of our example problems, the functional $F^*$ will restrict the dual variable to lie in a ball of known size.
The primal variable, on the other hand, is not explicitly bounded. It is however possible to prove data-based conservative bounds on the optimal solution, see, e.g., \cite[Appendix A]{tuomov-dtireg}. We can therefore add an artificial bound to the problem to force all iterates to be bounded, replacing $G$ by $\tilde G(x) \defeq G(x) + \delta_{B(0, C_x)}(x)$.
In practise, to avoid figuring out the exact magnitude of $C_x$, we update it dynamically. This avoids the constraint from ever becoming active and affecting the algorithm at all.
In \cite{tuomov-dtireg} a ``pseudo duality gap'' based on this idea was introduced to avoid problems with numerically infinite duality gaps. 
We will also use them in our reporting: we take the bound $C_x$ as the maximum over all iterations of all tested algorithms, and report the duality gap for the problem with $\tilde G$ replacing $G$.
We always report the pseudo-duality gap in decibels $10\log_{10}(\text{gap}^2/\text{gap}_0^2)$ relative to the initial iterate.

In addition to the pseudo-duality gap, we report for each algorithm the distance to a target solution and function value.
We report the distance in decibels $10\log_{10}(\norm{v^i-\realoptv}^2/\norm{\realoptv}^2)$ and the primal objective value $\text{val}(x) \defeq G(x)+F(Kx)$ relative to the target as
$10\log_{10}((\text{val}(x)-\text{val}(\hat x))^2/\text{val}(\hat x)^2)$.
The target solution $\realoptx$ we compute by taking one million iterations of the PDHGM.
We performed our computations with Matlab+C-MEX on a MacBook Pro with 16GB RAM and a 2.8 GHz Intel Core i5 CPU.
The initial iterates are $x^0=0$ and $y^0=0$.

\subsection{$\TGV^2$ denoising}

In this problem, we write $x=(v, w)$ and $y=(\phi,\psi)$, where $v$ is the image of interest, and take
\[
    G(x)=\frac{1}{2}\norm{f-v}^2,
    \quad
    K=\begin{pmatrix} \grad & -I \\ 0 & \Eabs \end{pmatrix},
    \quad\text{and}\quad
    F^*(y)=\delta_{B(0, \alpha)^{n_1n_2}}(\phi)+\delta_{B(0,\beta)^{n_1n_2}}(\psi).
\]
Here $\alpha,\beta>0$ are regularisation parameters, $\Eabs$ is the symmetrised gradient, and the balls are pixelwise Euclidean with the product $\Pi$ over image pixels.
Since there is no further spatial non-uniformity in this problem, it is natural to take as our projections $P_1x=v$, $P_2x=w$, $Q_1y=\phi$, and $Q_2y=\psi$.
It is then not difficult to calculate the optimal $\kappa_\ell$ of \cref{ex:kappa-refined}, so we use only the `xxxO' variants of the algorithms in \cref{tab:alg-variants}. 

As the regularisation parameters $(\beta, \alpha)$, we choose $(4.4, 4)$ for the downscaled image. For the original image we scale these parameters by $(0.25^{-2}, 0.25^{-1})$ corresponding to the image downscaling factor \cite{tuomov-tgvlearn}.
Since $G$ is not strongly convex with respect to $w$, we have $\tilde\gamma_2=0$.
For $v$ we take $\tilde\gamma_1=1/2$, corresponding to the gap versions of our convergence estimates.

We take $\delta=0.01$, and parametrise the standard PDHGM with $\sigma_0=1.9/\norm{K}$ and $\tau_0 \approx 0.52/\norm{K}$ solved from $\tau_0\sigma_0=(1-\delta)\norm{K}^2$. These are values that typically work well. For forward-differences discretisation of $\TGV^2$ with cell width $h=1$, we have $\norm{K}^2 \le 11.4$ \cite{tuomov-dtireg}.
For the `Relax' method from \cite{chambolle2014ergodic}, we use the same $\sigma_0$ and $\tau_0$, as well as the value $1.5$ for the inertial $\rho$ parameter.
For the increasing-$\psi$ `xxIx' variants of our algorithms, we take $\rho_1=\rho_2=5$, $\tau_{1,0}=\tau_0$, and $\tau_{2,0}=3\tau_0$. For the bounded-$\psi$ `xxBx' variants we take $\rho_1=\rho_2=5$, $\tau_{1,0}=\tau_0$, and $\tau_{2,0}=8\tau_0$. For both methods we also take $\eta_0=1/\tau_{0,1}$.
\emph{These parametrisations force $\phi_{1,0}=1/\tau_{1,0}^2$, and keep the initial step length $\tau_{1,0}$ for $v$ consistent with the basic PDHGM.} This justifies our algorithm comparisons using just a single set of parameters.
We plot the step length evolution for the A-DDBO variant in \cref{fig:step-length}\subref{fig:tgv2-steps}.

The results for deterministic variants of our algorithm are in \cref{table:tgv2-denoising-deterministic,fig:tgv2-denoising-deterministic}. We display the first 5000 iterations in a logarithmic fashion. To reduce computational overheads, we compute the reported quantities only every 10 iterations.
To reduce the effects of other processes occasionally occupying the computer, the CPU times reported are the average $\text{iteration\_time} = \text{total\_time}/\text{total\_iterations}$, excluding time spent initialising the algorithm.

Our first observation is that the variants `xDxx' based on the deterministic $\phi$ rule perform better than the ``random'' rule `xRxx'. Presently, with no randomisation, the only difference is the value of $\bar\gamma$. The value $0.0105$ from the initialisation bound \cref{thm:convergence-basic}\,\ref{item:conv:gamma} for $p=1/2$ and the value $0.0090$ for $p=1$ appear to give better performance than the maximal value $\tilde\gamma_1=0.5$.
Generally, the A-DDBO seems to have the best asymptotic performance, with A-DRBO close. A-DDIO has good initial performance, although especially on the higher resolution image, the PDHGM and `Relax' perform initially the best.
Overall, however, the question of the best performer seems to be a rather fair competition between `Relax' and A-DDBO.

\newlength{\swh}
\setlength{\swh}{0.23\textwidth}
\def\zinit{}
\def\finit{}

\begin{figure}[tbp!]
    \centering%
    \subcaptionbox{Gap: lo-res\zinit\label{fig:tgv2-denoising-deterministic-gap-lq-zinit}}{
        \includegraphics[height=\swh]{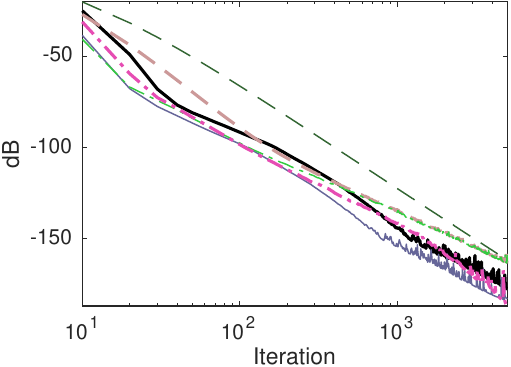}%
    }
    \subcaptionbox{Target: lo-res\zinit\label{fig:tgv2-denoising-deterministic-target-lq-zinit}}{
        \includegraphics[height=\swh]{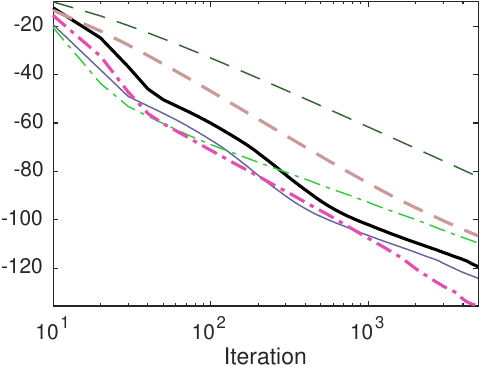}%
    }
    \subcaptionbox{Value: lo-res\zinit\label{fig:tgv2-denoising-deterministic-value-lq-zinit}}{
        \includegraphics[height=\swh]{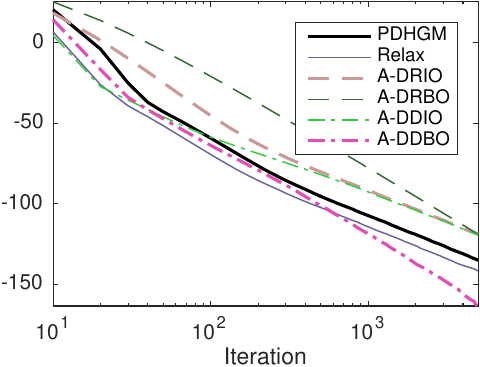}%
    }
    \subcaptionbox{Gap: hi-res\zinit\label{fig:tgv2-denoising-deterministic-gap-hq-zinit}}{
        \includegraphics[height=\swh]{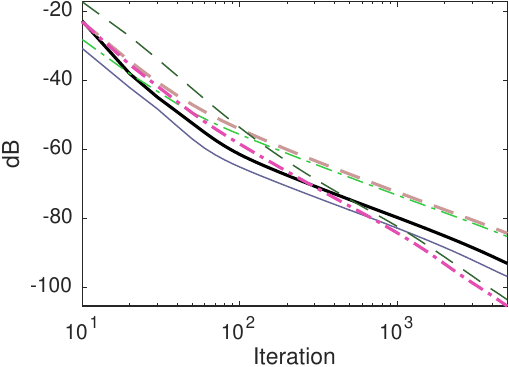}%
    }
    \subcaptionbox{Target: hi-res\zinit\label{fig:tgv2-denoising-deterministic-target-hq-zinit}}{
        \includegraphics[height=\swh]{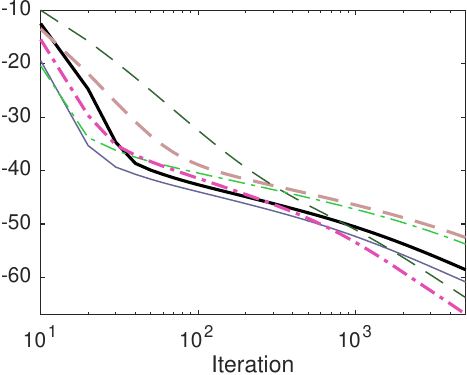}%
    }
    \subcaptionbox{Value: hi-res\zinit\label{fig:tgv2-denoising-deterministic-value-hq-zinit}}{
        \includegraphics[height=\swh]{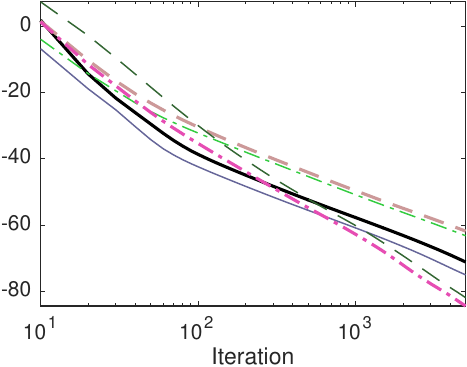}%
    }
    \caption{$\TGV^2$ denoising, deterministic variants of our algorithms with pixelwise step lengths, 5000 iterations, high (hi-res) and low (lo-res) resolution images.   
    }
    \label{fig:tgv2-denoising-deterministic}
\end{figure}

\begin{table}[tbp!]
    \caption{$\TGV^2$ denoising performance: CPU time and number of iterations (at a resolution of 10) to reach given duality gap, distance to target, or primal objective value.}
    \label{table:tgv2-denoising-deterministic}
    \centering
    \small
    \setlength{\tabcolsep}{2pt}
    \begin{tabular}{l|rr|rr|rr}
\multicolumn{7}{c}{low resolution}\\
\hline
 & \multicolumn{2}{c}{gap $\le -60$dB} & \multicolumn{2}{|c}{tgt $\le -60$dB} & \multicolumn{2}{|c}{val $\le -60$dB}\\
Method & iter & time & iter & time & iter & time\\
\hline
PDHGM & 30 & 0.21s & 100 & 0.72s & 110 & 0.79s\\
Relax & 20 & 0.20s & 70 & 0.71s & 70 & 0.71s\\
A-DRIO & 40 & 0.26s & 230 & 1.55s & 180 & 1.22s\\
A-DRBO & 80 & 0.54s & 890 & 6.07s & 500 & 3.41s\\
A-DDIO & 20 & 0.14s & 50 & 0.36s & 110 & 0.80s\\
A-DDBO & 30 & 0.19s & 50 & 0.32s & 90 & 0.58s\\
\end{tabular}

    \,%
    \begin{tabular}{rr|rr|rr}
\multicolumn{6}{c}{high resolution}\\
\hline
\multicolumn{2}{c}{gap $\le -50$dB} & \multicolumn{2}{|c}{tgt $\le -50$dB} & \multicolumn{2}{|c}{val $\le -50$dB}\\
iter & time & iter & time & iter & time\\
\hline
50 & 6.31s & 870 & 111.83s & 370 & 47.49s\\
40 & 6.93s & 580 & 102.89s & 250 & 44.25s\\
70 & 9.17s & 2750 & 365.52s & 1050 & 139.48s\\
80 & 10.56s & 860 & 114.81s & 420 & 56.00s\\
60 & 7.37s & 2140 & 267.29s & 900 & 112.34s\\
60 & 7.85s & 600 & 79.67s & 340 & 45.09s\\
\end{tabular}

\end{table}

\begin{figure}[tbp!]
    \centering%
    \subcaptionbox{denoising\label{fig:tgv2-steps}}{
        \includegraphics[height=\swh]{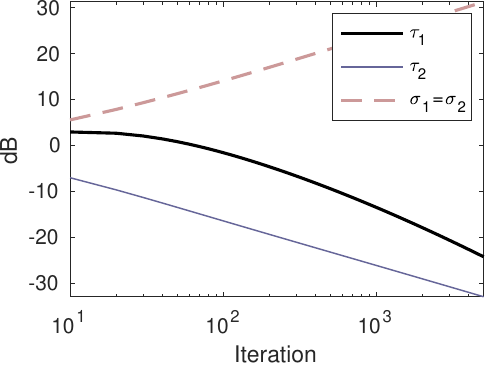}%
    }
    \subcaptionbox{deblurring $\sigma$}{
        \includegraphics[height=\swh]{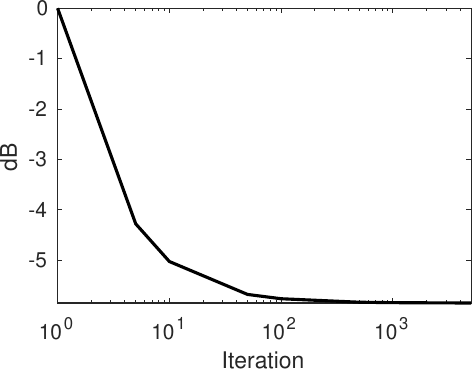}%
    }
    \subcaptionbox{deblurring $\tau$ colour coding}{
        \includegraphics[width=1.3\swh]{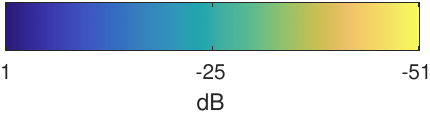}%
    }\\%
    \subcaptionbox{deblurring $\tau$, $i=10$}{
        \includegraphics[height=0.9\swh]{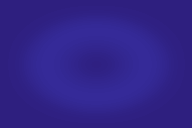}%
    }%
    \subcaptionbox{deblurring $\tau$, $i=50$}{
        \includegraphics[height=0.9\swh]{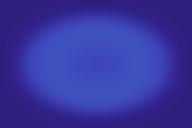}%
    }%
    \subcaptionbox{deblurring $\tau$, $i=100$}{
        \includegraphics[height=0.9\swh]{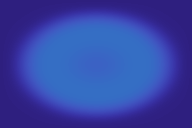}%
    }\\%
    \subcaptionbox{deblurring $\tau$, $i=500$}{
        \includegraphics[height=0.9\swh]{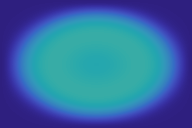}%
    }%
    \subcaptionbox{deblurring $\tau$, $i=1000$}{
        \includegraphics[height=0.9\swh]{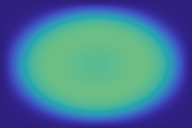}%
    }%
    \subcaptionbox{deblurring $\tau$, $i=5000$}{
        \includegraphics[height=0.9\swh]{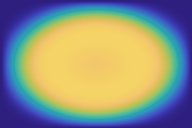}%
    }%
    \caption{Step length evolution (logarithmic from initialisation). A-DDBO $\TGV^2$ denoising and A-DDIM $\TV$ deblurring. The $\tau$ plots of the latter are images in the Fourier domain, lighter colour means smaller value of $\tau$ \emph{relative to initialisation} (for that specific Fourier component). Note that the images depict logarithm \emph{change}, not absolute values.}
    \label{fig:step-length}
\end{figure}

\subsection{$\TGV^2$ denoising with stochastic algorithm variants}

We also test stochastic variants of our algorithms based on the alternating sampling of \cref{ex:sampling-fixed-number} with $M=1$ and, when appropriate, \cref{ex:sampling-alternating}. We take all probabilities equal to $0.5$, that is $\mathbb{p}_x=\tilde\pi_1=\tilde\pi_2=\tilde\nu_1=\tilde\nu_2=0.5$.
In the doubly-stochastic `Bxxx' variants of the algorithms, we take $\eta_{\tau,i}^\perp=\eta_{\sigma,i}^\perp=0.9 \cdot 0.5 \eta_i$ following the proportional rule \cref{lemma:etaperp-examples}\ref{item:etaperp-proportional}. 

The results are in \cref{fig:tgv2-denoising-random}. To conserve space, we have only included a few descriptive algorithm variants.
On the $x$ axis, to better describe to the amount of actual work performed by the stochastic methods, the ``iteration'' count refers to the \emph{expected} number of full primal--dual updates. 
For all the displayed stochastic variants, with the present choice of probabilities, the expected number of full updates in each iteration is $0.75$.

We run each algorithm 50 times, and plot for each iteration the 90\% confidence interval according to Student's $t$-distribution. Towards the 5000th iteration, these generally become very narrow, indicating reliability of the random method.
Overall, the full-dual-update `Pxxx' variants perform better than the doubly-stochastic `Bxxx' variants. In particular, A-PDBO has performance comparable to or even better than the PDHGM.

\begin{figure}[tbp!]
    \centering%
    \subcaptionbox{Gap: lo-res\zinit\label{fig:tgv2-denoising-random-gap-lq-zinit}}{
        \includegraphics[height=\swh]{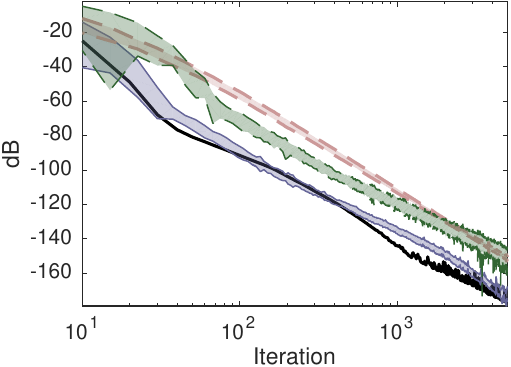}%
    }
    \subcaptionbox{Target: lo-res\zinit\label{fig:tgv2-denoising-random-target-lq-zinit}}{
        \includegraphics[height=\swh]{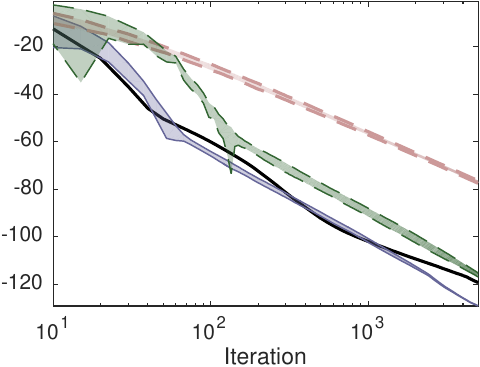}%
    }
    \subcaptionbox{Value: lo-res\zinit\label{fig:tgv2-denoising-random-value-lq-zinit}}{
        \includegraphics[height=\swh]{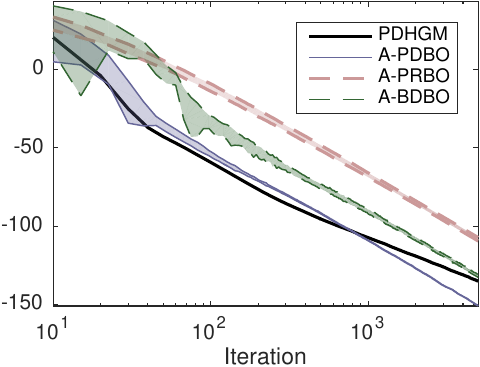}%
    }
    \caption{$\TGV^2$ denoising, stochastic variants of our algorithms: 5000 iterations, low resolution images. 
    Iteration number scaled by the fraction of blocks updated on average. For each iteration, 90\% confidence interval according to the $t$-distribution over 50 random runs.}
    \label{fig:tgv2-denoising-random}
\end{figure}

\subsection{TV deblurring}

We want to remove the blur in \cref{fig:blurry}.
We do this by taking
\[
	G(x)=\frac{1}{2}\norm{f-\mathcal{F}^*(a\mathcal{F}x)}^2,
	\quad
	K=\grad,
	\quad\text{and}\quad
	F^*(y)=\delta_{B(0,\alpha)^{n_1n_2}}(y),
\]
where the balls are again pixelwise Euclidean, and  $\mathcal{F}$ the discrete Fourier transform. The factors $a=(a_1,\ldots,a_m)$ model the blurring operation in Fourier basis. 

We take $\alpha=2.55$ for the high resolution image and scale this to $\alpha=2.55*0.15$ for the low resolution image.
We parametrise the PDHGM and `Relax' algorithms exactly as for $\TGV^2$ denoising above, taking into account the estimate $8 \ge \norm{K}^2$ \cite{chambolle2004meanalgorithm}. 
We take $Q_1=I$ and $P_j$ as the projection to the $j$:th Fourier component so $m=n_1n_2$ and $n=1$. Thus \emph{each primal Fourier component has its own step length parameter}.
We initialise the latter as $\tau_{j,0}=\tau_0/(\lambda+(1-\lambda)\gamma_j)$, where the componentwise factor of strong convexity $\gamma_j=\abs{a_j}^2$.
For the bounded-$\psi$ `xxBx` algorithm variants we take $\lambda=0.01$, and for the increasing-$\psi$ `xxIx' variants $\lambda=0.1$.
We illustrate the step length evolution of the variant A-DDIM in \cref{fig:step-length}.

We only experiment with deterministic algorithms, as we do not expect small-scale randomisation to be beneficial. We also use the maximal $\kappa$ `xxxM' variants, as a more optimal $\kappa$ would be difficult to compute. 
The results are in \cref{table:tv-deblurring-deterministic,fig:tv-deblurring-deterministic}.
Similarly to A-DDBO in our $\TGV^2$ denoising experiments, A-DDBM performs reliably well, indeed better than the PDHGM or `Relax'. However, in many cases, A-DRBM and A-DDIM are even faster. 

\begin{figure}[tbp!]
    \centering%
    \subcaptionbox{Gap: lo-res\zinit\label{fig:tv-deblurring-deterministic-gap-lq-zinit}}{
        \includegraphics[height=\swh]{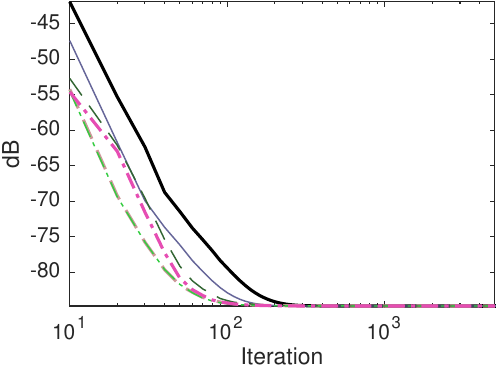}%
    }
    \subcaptionbox{Target: lo-res\zinit\label{fig:tv-deblurring-deterministic-target-lq-zinit}}{
        \includegraphics[height=\swh]{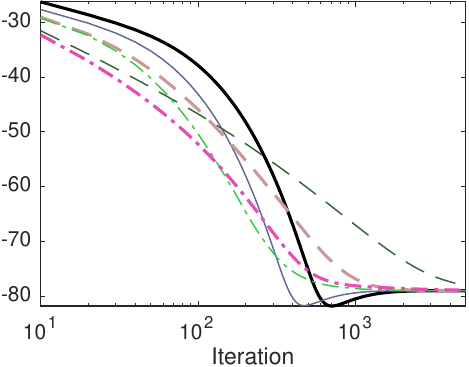}%
    }
    \subcaptionbox{Value: lo-res\zinit\label{fig:tv-deblurring-deterministic-value-lq-zinit}}{
        \includegraphics[height=\swh]{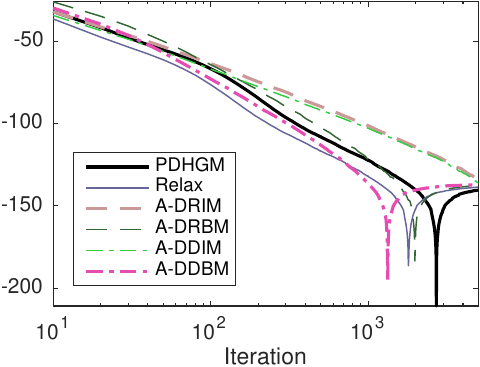}%
    }
    \subcaptionbox{Gap: hi-res\zinit\label{fig:tv-deblurring-deterministic-gap-hq-zinit}}{
        \includegraphics[height=\swh]{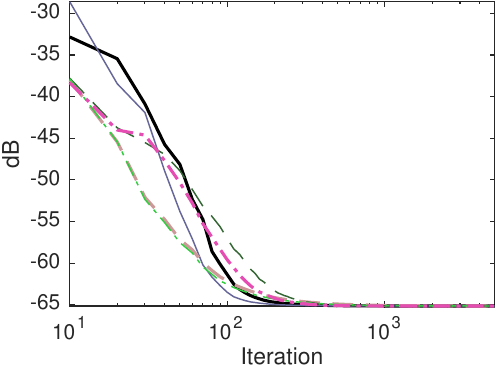}%
    }
    \subcaptionbox{Target: hi-res\zinit\label{fig:tv-deblurring-deterministic-target-hq-zinit}}{
        \includegraphics[height=\swh]{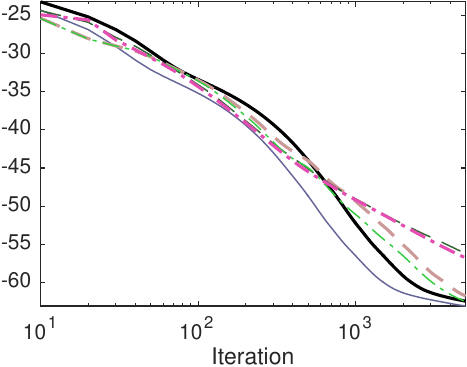}%
    }
    \subcaptionbox{Value: hi-res\zinit\label{fig:tv-deblurring-deterministic-value-hq-zinit}}{
        \includegraphics[height=\swh]{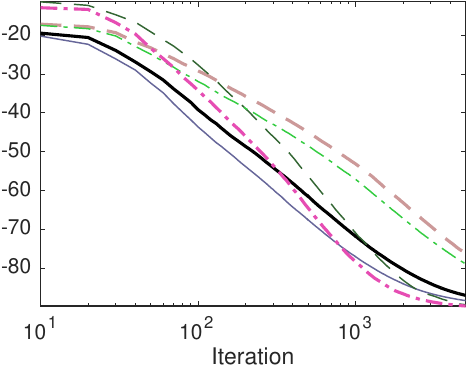}%
    }
    \caption{$\TV$ deblurring, deterministic variants of our algorithms with pixelwise step lengths, first 5000 iterations, high (hi-res) and low (lo-res) resolution images.
    }
    \label{fig:tv-deblurring-deterministic}
\end{figure}

\begin{table}[tbp!]
    \caption{$\TV$ deblurring performance: CPU time and number of iterations (at a resolution of 10) to reach given duality gap, distance to target, or primal objective value.}
    \label{table:tv-deblurring-deterministic}
    \centering
    \small
    \setlength{\tabcolsep}{3pt}
    \begin{tabular}{l|rr|rr|rr}
\multicolumn{7}{c}{low resolution}\\
\hline
 & \multicolumn{2}{c}{gap $\le -60$dB} & \multicolumn{2}{|c}{tgt $\le -60$dB} & \multicolumn{2}{|c}{val $\le -60$dB}\\
Method & iter & time & iter & time & iter & time\\
\hline
PDHGM & 30 & 0.18s & 330 & 2.05s & 70 & 0.43s\\
Relax & 20 & 0.11s & 220 & 1.30s & 50 & 0.29s\\
A-DRIM & 20 & 0.14s & 280 & 2.08s & 80 & 0.59s\\
A-DRBM & 20 & 0.14s & 490 & 3.58s & 90 & 0.65s\\
A-DDIM & 20 & 0.14s & 170 & 1.25s & 70 & 0.51s\\
A-DDBM & 20 & 0.15s & 180 & 1.37s & 60 & 0.45s\\
\end{tabular}

    \ %
    \begin{tabular}{rr|rr|rr}
\multicolumn{6}{c}{high resolution}\\
\hline
\multicolumn{2}{c}{gap $\le -50$dB} & \multicolumn{2}{|c}{tgt $\le -40$dB} & \multicolumn{2}{|c}{val $\le -40$dB}\\
iter & time & iter & time & iter & time\\
\hline
60 & 5.04s & 330 & 28.12s & 110 & 9.31s\\
50 & 4.32s & 220 & 19.30s & 90 & 7.84s\\
30 & 3.27s & 280 & 31.41s & 320 & 35.92s\\
60 & 6.48s & 240 & 26.27s & 220 & 24.07s\\
30 & 3.17s & 260 & 28.35s & 230 & 25.06s\\
50 & 5.56s & 230 & 25.98s & 150 & 16.90s\\
\end{tabular}

\end{table}

\subsection{TV undimming}

We take $K$ and $F^*$ as for TV deblurring, but $G(u) := \frac{1}{2}\norm{f-\gamma \cdot u}^2$ for the sinusoidal dimming mask $\gamma: \Omega \to \R$.
Our experimental setup is also nearly the same as TV deblurring, with the natural difference that the projection $P_j$ are no longer to the Fourier basis, but to individual image pixels. The results are in \cref{fig:tv-undimming-deterministic}, and \cref{table:tv-undimming-deterministic}. They tell roughly the same story as TV deblurring, with A-DDBM performing well and reliably.

\begin{figure}[tbp!]
    \centering%
    \subcaptionbox{Gap: lo-res\zinit\label{fig:tv-undimming-deterministic-gap-lq-zinit}}{
        \includegraphics[height=\swh]{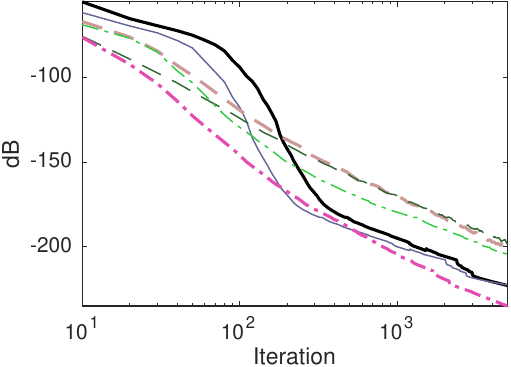}%
    }
    \subcaptionbox{Target: lo-res\zinit\label{fig:tv-undimming-deterministic-target-lq-zinit}}{
        \includegraphics[height=\swh]{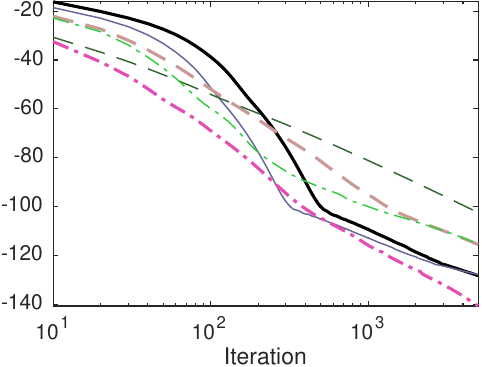}%
    }
    \subcaptionbox{Value: lo-res\zinit\label{fig:tv-undimming-deterministic-value-lq-zinit}}{
        \includegraphics[height=\swh]{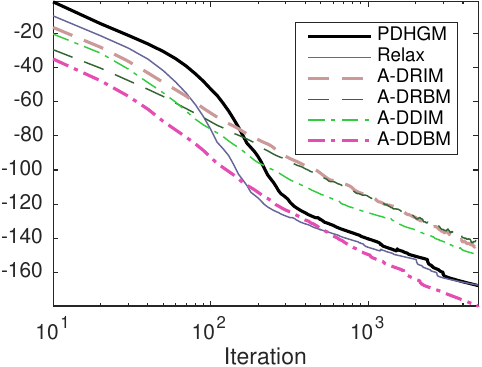}%
    }
    \subcaptionbox{Gap: hi-res\zinit\label{fig:tv-undimming-deterministic-gap-hq-zinit}}{
        \includegraphics[height=\swh]{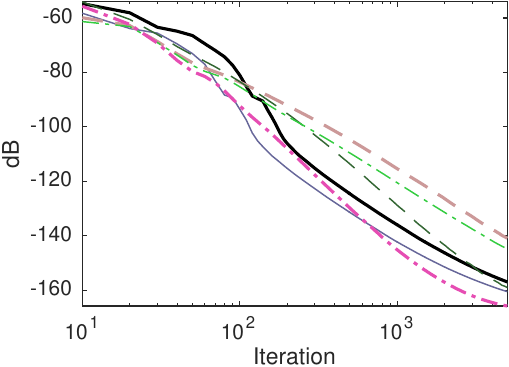}%
    }
    \subcaptionbox{Target: hi-res\zinit\label{fig:tv-undimming-deterministic-target-hq-zinit}}{
        \includegraphics[height=\swh]{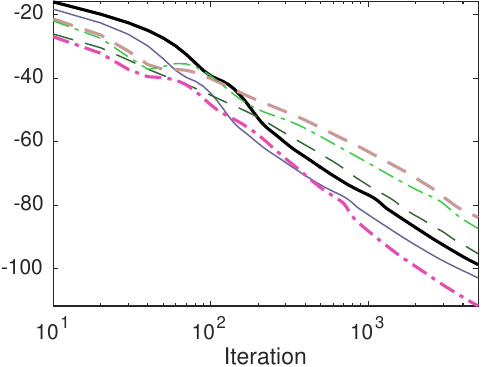}%
    }
    \subcaptionbox{Value: hi-res\zinit\label{fig:tv-undimming-deterministic-value-hq-zinit}}{
        \includegraphics[height=\swh]{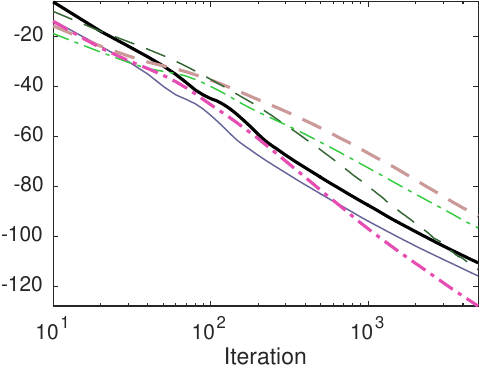}%
    }
    \caption{$\TV$ undimming, deterministic variants of our algorithms with pixelwise step lengths, 5000 iterations, high (hi-res) and low (lo-res) resolution images. 
    }
    \label{fig:tv-undimming-deterministic}
\end{figure}

\begin{table}[tbp!]
    \caption{$\TV$ undimming performance: CPU time and number of iterations (at a resolution of 10) to reach given duality gap, distance to target, or primal objective value.}
    \label{table:tv-undimming-deterministic}
    \centering
    \small
    \setlength{\tabcolsep}{3pt}
    \begin{tabular}{l|rr|rr|rr}
\multicolumn{7}{c}{low resolution}\\
\hline
 & \multicolumn{2}{c}{gap $\le -80$dB} & \multicolumn{2}{|c}{tgt $\le -60$dB} & \multicolumn{2}{|c}{val $\le -60$dB}\\
Method & iter & time & iter & time & iter & time\\
\hline
PDHGM & 70 & 0.18s & 200 & 0.51s & 120 & 0.30s\\
Relax & 50 & 0.16s & 130 & 0.41s & 80 & 0.25s\\
A-DRIM & 30 & 0.10s & 160 & 0.57s & 80 & 0.28s\\
A-DRBM & 20 & 0.05s & 170 & 0.47s & 60 & 0.16s\\
A-DDIM & 30 & 0.08s & 110 & 0.30s & 60 & 0.16s\\
A-DDBM & 20 & 0.05s & 70 & 0.18s & 40 & 0.10s\\
\end{tabular}

    \ %
    \begin{tabular}{rr|rr|rr}
\multicolumn{6}{c}{high resolution}\\
\hline
\multicolumn{2}{c}{gap $\le -80$dB} & \multicolumn{2}{|c}{tgt $\le -60$dB} & \multicolumn{2}{|c}{val $\le -60$dB}\\
iter & time & iter & time & iter & time\\
\hline
100 & 3.41s & 300 & 10.31s & 210 & 7.21s\\
70 & 3.03s & 200 & 8.73s & 140 & 6.10s\\
80 & 3.52s & 760 & 33.82s & 640 & 28.48s\\
90 & 3.95s & 370 & 16.39s & 380 & 16.84s\\
70 & 3.05s & 580 & 25.57s & 430 & 18.94s\\
60 & 2.63s & 230 & 10.22s & 200 & 8.88s\\
\end{tabular}

\end{table}

\section*{Conclusions}

We have derived several accelerated block-proximal primal--dual methods, both stochastic and deterministic. We have concentrated on applying them deterministically, taking advantage of blockwise---indeed pixelwise---factors of strong convexity, to obtain improved performance compared to standard methods. In future work, it will be interesting to evaluate the methods on real large scale problems to other state-of-the-art stochastic optimisation methods. Moreover, interesting questions include heuristics and other mechanisms for optimal initialisation of the pixelwise parameters, as well as combination with over-relaxation and inertial schemes, such as the extensions of the PDHGM considered in \cite{condat2013primaldual,vu2013splitting,he2014convergence,tuomov-inertia}.

\section*{Acknowledgements}

The author would like to thank Peter Richtárik and Olivier Fercoq for several fruitful discussions, and for introducing him to stochastic optimisation.
Moreover, the support of the EPSRC grant EP/M00483X/1 ``Efficient computational tools for inverse imaging problems'' is acknowledged during the initial two months of the research.

\section*{A data statement for the EPSRC}

Implementations of the algorithms described in the paper, and relevant boilerplate codes, are available on Zenodo at \doi{10.5281/zenodo.1042419}.
The sample photo, also included in the archive, is from the free Kodak image suite, at the time of writing at \url{http://r0k.us/graphics/kodak/}.

\def\bibnamefont#1{\textsc{\sffamily #1}}
\bibliographystyle{texbase/poor_approx_of_shinybib}

 \providecommand{\eprint}[1]{\href{http://arxiv.org/abs/#1}{arXiv:#1}}
  \providecommand{\eprint}[1]{\href{http://arxiv.org/abs/#1}{arXiv:#1}}
  \providecommand{\noopsort}[1]{}


\end{document}